\title{Interaction between skew-representability, tensor products, extension properties, and rank inequalities}

\documentclass{article}
\usepackage{fullpage}

\usepackage[utf8]{inputenc}
\usepackage{amsmath}
\usepackage{amsthm,amssymb}
\usepackage{nicefrac,amsfonts}
\usepackage{thmtools,mathtools,leftindex,leftidx}
\usepackage{caption,subcaption}
\usepackage{epsfig,graphicx,graphics,color,tikz,tikz-cd,tcolorbox}
\usepackage[table]{xcolor}
\usepackage{enumerate,enumitem}
\usepackage[sort,nocompress]{cite}
\usepackage{xspace}
\usepackage{bbm}
\usepackage[hyperfootnotes=false]{hyperref}
\usepackage[capitalise]{cleveref}
\usepackage{float}
\hypersetup{
    colorlinks=true,       
    linkcolor=blue,        
    citecolor=magenta,         
    filecolor=magenta,     
    urlcolor=cyan,         
    linktoc=all
}
\usepackage{titling}
\usepackage{thm-restate}
\numberwithin{equation}{section}

\newlength{\bibitemsep}
\newlength{\bibparskip}
\let\oldthebibliography\thebibliography
\renewcommand\thebibliography[1]{%
  \oldthebibliography{#1}%
  \setlength{\parskip}{\bibitemsep}%
  \setlength{\itemsep}{\bibparskip}%
}

\makeatletter
\renewcommand{\paragraph}{%
  \@startsection{paragraph}{4}%
  {\z@}{1.6ex \@plus 1ex \@minus .2ex}{-0.5em}%
  {\normalfont\normalsize\bfseries}%
}
\makeatother

\theoremstyle{plain}
\newtheorem{thm}{Theorem}[section]
\newtheorem{lem}[thm]{Lemma}
\newtheorem{cor}[thm]{Corollary}
\newtheorem{cla}[thm]{Claim}
\newtheorem{prop}[thm]{Proposition}

\theoremstyle{definition}

\newtheorem{ex}[thm]{Example}

\newtheorem{rem}[thm]{Remark}

\crefname{cla}{Claim}{Claims}
\Crefname{cla}{Claim}{Claims}

\newcommand*{\claimproofname}{Proof of claim.}
\newenvironment{claimproof}[1][\claimproofname]{\begin{proof}[#1]}{\end{proof}}

\makeatletter
\newcommand{\leqnomode}{\tagsleft@true}
\newcommand{\reqnomode}{\tagsleft@false}
\makeatother

\makeatletter
 \newcommand{\linkdest}[1]{\Hy@raisedlink{\hypertarget{#1}{}}}
\makeatother

\newcommand\prd[2]{%
  {\vphantom{#2}}^{#1}\!#2%
}

    \newcommand\tpm[1]{{#1}_{\bullet}}

\DeclareMathOperator\si{si}

\def\obt#1{\overbracket[.5pt][2pt]{#1}}

\newcommand{\bP}{\mathbb{P}}
\newcommand{\bR}{\mathbb{R}}
\newcommand{\bQ}{\mathbb{Q}}
\newcommand{\bZ}{\mathbb{Z}}

\newcommand{\bF}{\mathbb{F}}

\newcommand{\cB}{\mathcal{B}}

\newcommand{\cI}{\mathcal{I}}

\newcommand{\cL}{\mathcal{L}}
\newcommand{\cM}{\mathcal{M}}

\newcommand{\cP}{\mathcal{P}}
\newcommand{\cQ}{\mathcal{Q}}
\newcommand{\cR}{\mathcal{R}}

\def\cl{\mathrm{cl}}

\let\Right\bigr
\let\Left\bigl
\makeatletter
\def\bigr#1{\Right#1\@ifnextchar){\!\bigr}{}}
\def\bigl#1{\Left#1\@ifnextchar({\!\bigl}{}}
\makeatother

\thanksmarkseries{alph}
\author{
Kristóf Bérczi\thanks{MTA-ELTE Matroid Optimization Research Group and HUN-REN–ELTE Egerváry Research Group, Department of Operations Research, ELTE Eötvös Loránd University, and HUN-REN Alfréd Rényi Institute of Mathematics, Budapest, Hungary. Email: \texttt{kristof.berczi@ttk.elte.hu}.}
\and
Boglárka Gehér\thanks{Department of Applied Analysis and Computational Mathematics, ELTE Eötvös Loránd University, and HUN-REN Alfréd Rényi Institute of Mathematics, Budapest, Hungary. Email: \texttt{bogigeher@gmail.com}.}
\and
András Imolay\thanks{Department of Operations Research, ELTE Eötvös Loránd University, Budapest, Hungary. Email: \texttt{andras.imolay@ttk.elte.hu}.}
\and
László Lovász\thanks{HUN-REN Alfréd Rényi Institute of Mathematics, Budapest, Hungary. Email: \texttt{laszlo.lovasz@ttk.elte.hu}.}
\and Carles Padró\thanks{Universitat Polit\`ecnica de Catalunya, Barcelona, Spain. Email: \texttt{carles.padro@upc.edu}.}
\and
Tamás Schwarcz\thanks{HUN-REN Alfréd Rényi Institute of Mathematics, Budapest, Hungary. Most of this work was done while this author was affiliated to the Department of Mathematics, London School of Economics and Political Science, London, England, United Kingdom. Email: \texttt{tamas.schwarcz@ttk.elte.hu}.} 
}
\date{}

\begin{document}
\hypersetup{pageanchor=false}
\maketitle

\thispagestyle{empty}
\begin{abstract} 
Skew-representable matroids form a fundamental class in matroid theory, bridging combinatorics and linear algebra. They play an important role in areas such as coding theory, optimization, and combinatorial geometry, where linear structure is crucial for both theoretical insights and algorithmic applications. 
Since skew-representability is undecidable even for rank-$3$ matroids, structural characterizations and explicit certificates of non-skew-representability are particularly interesting.

In this paper, we introduce an approach to studying skew-representability and structural properties of matroids and polymatroid functions via tensor products. We characterize skew-representable matroids, as well as matroids representable over skew fields of a prescribed characteristic, in terms of iterated tensor products.
In particular, a connected matroid is non-skew-representable if and only if, for some positive integer $k$, no $k$-fold iterated tensor product with $U_{2,3}$ exists. Thus, non-skew-representability admits a finite, computably verifiable matroid-theoretic obstruction; an analogous statement holds when the characteristic is prescribed. We also prove that every rank-3 matroid admits a tensor product with every uniform matroid and give a construction yielding the unique freest tensor product in this setting. Finally, as an application of the tensor product framework, we give a new proof of Ingleton's inequality and, more importantly, derive the first known linear rank inequality for folded skew-representable matroids that does not follow from the common information property.

\medskip


\medskip


\end{abstract}
 \newpage
\pagenumbering{roman}
\hypersetup{pageanchor=true}
\tableofcontents
\newpage
\pagenumbering{arabic}
\setcounter{page}{1}

\section{Introduction}
\label{sec:intro}

The class of representable matroids serves as a fundamental bridge between linear algebra and combinatorial structures.
A matroid is {\it (linearly) representable} over a field $\bF$ if the elements of its ground set correspond to vectors in a vector space over $\bF$ in such a way that the rank of every set equals the dimension of the spanned subspace.
More generally, a polymatroid function is {\it representable} if the elements of its ground set correspond to subspaces of a vector space, with the value of each set given by the dimension of the subspace spanned by the corresponding subspaces.
Replacing fields by skew fields\footnote{Skew fields are sometimes called {\it division rings} in the literature.} gives the corresponding notions of {\it skew-representable} matroids and polymatroid functions.

Extensive research has focused on identifying necessary conditions for representability.
One important source of such conditions is given by {\it extension properties} that require a matroid or polymatroid function to admit certain extensions.
Another one is provided by {\it linear rank inequalities}, that is, homogeneous linear inequalities satisfied by every representable polymatroid function.
The first linear rank inequality that is not a consequence of the submodularity and monotonicity inequalities was found by Ingleton~\cite{ingleton1971representation}, and infinite families of linear rank inequalities have been identified since then~\cite{dougherty2009linear,kinser2011new}.
{\it Characteristic-dependent linear rank inequalities} -- that is, inequalities that hold for matroids representable over fields of prescribed characteristic -- have attracted considerable interest as well~\cite{blasiak2011lexicographic,dougherty2015characteristic, dougherty2015achievable, macias2022theorem, pena2019characteristic, pena2021characteristic, pena2019find, pena2023access}.
Interestingly, all characteristic-independent linear rank inequalities known prior to this work can be derived from a certain extension property, the so-called {\it common information property}~\cite{dougherty2009linear,bamiloshin2021common}.

Several standard operations on linear representations, such as direct sums, deletions, contractions, and duality, have natural counterparts for all matroids.
Tensor products behave differently. 
A {\it tensor product} of matroids $M_1=(S_1,r_1)$ and $M_2=(S_2,r_2)$ is a matroid $M=(S_1\times S_2,r)$ such that $r(Y_1\times Y_2)=r_1(Y_1)\cdot r_2(Y_2)$ holds for all $Y_1\subseteq S_1$ and $Y_2\subseteq S_2$. 
Although any two matroids represented over the same field admit a tensor product, obtained from the Kronecker product of the representing matrices, Las Vergnas~\cite{las1981products} showed that two abstract matroids need not admit a tensor product at all -- for example, the Vámos matroid and the rank-2 uniform matroid on three elements admit none.

\subsection{Our results and techniques} \label{sec:results}

Our first main result characterizes skew-representability of connected matroids using the existence of iterated tensor products.
We say that a matroid $M$ is \textit{$k$-tensor-compatible} with a matroid $N$ if there exists a sequence $M=P_0,P_1,\dots, P_k$ of matroids such that $P_i$ is a tensor product of $P_{i-1}$ and $N$ for every $i \in \{1,\dots, k\}$.
Taking $N=U_{2,3}$, the rank-2 uniform matroid on three elements, already gives the following.

\begin{restatable}{thm}{skewrepr} \label{thm:skewrepr}
    A connected matroid $M$ is skew-representable if and only if it is $k$-tensor-compatible with $U_{2,3}$ for every positive integer $k$.
\end{restatable}

Equivalently, if a connected matroid is not skew-representable, then for some finite $k$ it fails to be $k$-tensor-compatible with $U_{2,3}$. 
Thus, \cref{thm:skewrepr} provides a finite matroid-theoretic certificate of non-skew-representability.
As $k$-tensor-compatibility is decidable for each fixed $k$, this gives a matroid-theoretic proof that skew-representability of connected matroids is co-recursively enumerable.
Since skew-representability is undecidable even for connected matroids~\cite{kuhne2023von}, the least $k$ for which compatibility fails cannot be bounded by a computable function of the size of the ground set of $M$ (see \cref{cor:core} and \cref{rem:unbounded} for related statements).

For skew-representable matroids, iterated tensor products with $U_{2,3}$ can be constructed using Kronecker products.
The more involved reverse direction of \cref{thm:skewrepr} is obtained by linking both skew-representability and tensor products to extension properties.
Roughly speaking, a matroid is {\it fully modular extendable} if it is the restriction of a -- possibly infinite -- matroid whose pairs of flats satisfy the submodular rank inequality with equality.
We first show, using the Veblen--Young theorem on representations of projective spaces, that a matroid is fully modular extendable precisely when each of its connected components is either of rank~$3$ or skew-representable, not necessarily over the same skew field (\cref{thm:equiv}).
In particular, a connected matroid of rank at least $4$ is fully modular extendable if and only if it is skew-representable (\cref{cor:rank4}).
The second ingredient is that any tensor product of $M$ with $U_{2,3}$ yields, for every prescribed pair of flats of $M$, an extension in which that pair of flats becomes modular (\cref{thm:1modular}). 
Iterating this construction, $k$-tensor-compatibility with $U_{2,3}$ for every $k$ implies that the matroid is fully modular extendable (\cref{cor:fme}).

The tensor-product characterization can be refined to prescribe the characteristic of the representing skew field.
Here, an explicit sequence of testing matroids can be fixed independently of $M$, and the connectivity assumption can be dropped.

\begin{restatable}{thm}{fixed} \label{thm:fixed} 
Let $p$ be either $0$ or a prime.  For a positive integer $k$, let $N_{p,k}$ denote the matroid represented by the matrix $A_k\coloneqq \begin{bmatrix} 1 & 0 & 1 \\ 0 & 1 & 1 \end{bmatrix}^{\otimes k}$ over $\bF_p$ if $p > 0$, and over $\bQ$ if $p=0$.
Then, a matroid $M$ is representable over a skew field of characteristic $p$ if and only if $M$ admits a tensor product with $N_{p,k}$ for every positive integer $k$.
\end{restatable}

Note that $N_{p,1}=U_{2,3}$ and $N_{p,k}$ is one particular tensor product of $k$ copies of $U_{2,3}$.
Both \cref{thm:skewrepr} and \cref{thm:fixed} follow from a more general result, which characterizes matroids whose connected components are representable over skew fields with characteristic in a prescribed set $C$, in terms of iterated tensor products with a matroid depending only on $C$ (\cref{thm:char}).

The link between tensor products and modular extensions also has a polymatroid form: if a polymatroid function admits a tensor product with the rank function of an iterated tensor product of $k$ copies of $U_{2,3}$, then it is $k$-fractionally modular extendable, and the conclusion can be strengthened to $k$-modular extendability when the tensor product can be chosen integer-valued (\cref{thm:fme} and \cref{cor:ffme}). 

For rank-3 matroids, we show that $1$-tensor-compatibility with $U_{2,3}$ never fails.
In fact, we prove the following stronger result.
Here, a tensor product is {\it freest} if any of its dependent sets is dependent in each tensor product of the two matroids. 

\begin{restatable}{thm}{uniform} \label{thm:uniform}
    Every rank-3 matroid has a freest tensor product with every uniform matroid.
\end{restatable}

Our construction gives an explicit description of the bases.
We first derive necessary conditions that every basis of any tensor product must satisfy, and then show that, when the second factor is uniform, the sets satisfying these conditions form the bases of a matroid.

Tensor-product obstructions also provide a method for deriving linear rank inequalities.
First, we show how Las Vergnas's proof~\cite{las1981products} of the nonexistence of a tensor product of the Vámos matroid with $U_{2,3}$ leads to a simple proof that every polymatroid function admitting a tensor product with the rank function of $U_{2,3}$ satisfies Ingleton's inequality (\cref{thm:ingleton}).
We next show that the Fano and non-Fano matroids admit no tensor product with each other, and extract from the argument inequalities valid for folded skew-representable polymatroid functions in characteristic 2 and in characteristic different from 2, respectively, where we say that a polymatroid function $\varphi$ is \emph{folded skew-representable (in characteristic $p$)} if $k\cdot\varphi$ is representable over a skew field (of characteristic $p$) for some $k\in\mathbb{Z}_+$ (\cref{thm:Fano_non_Fano,cor:characteristic_dependent}).
Our main result on linear rank inequalities is the following, where $M(K_4)$ denotes the graphic matroid of the complete graph on four vertices. The inequality is stated explicitly in \cref{thm:new,cor:skdf}.

\begin{thm} \label{thm:intro-inequality}
    There exists a linear rank inequality satisfied by every polymatroid function admitting a tensor product with the rank function of $M(K_4)$, and hence by every folded skew-representable polymatroid function, but violated by the rank function of the non-Desargues matroid.
\end{thm}

Since the non-Desargues matroid has rank 3 and every rank-3 matroid satisfies the common information property, the inequality cannot be derived from it.
To the best of our knowledge, this is the first known characteristic-independent linear rank inequality with this feature.
\Cref{thm:intro-inequality} is derived from a proof of the nonexistence of a tensor product of the non-Desargues matroid and $M(K_4)$.
By \cref{thm:uniform}, the non-Desargues matroid is $1$-tensor-compatible with every uniform matroid, while the nonexistence of a tensor product with $M(K_4)$ implies that it is not $2$-tensor-compatible with $U_{2,3}$.

\subsection{Related work}

Next, we discuss the works most related to our results and connect them to Theorems~\ref{thm:skewrepr} -- \ref{thm:intro-inequality} above.

\paragraph{Skew-representability.} Kühne, Pendavingh, and Yashfe~\cite{kuhne2023von} showed that skew-representability is undecidable, even for rank-$3$ matroids (see~\cref{prop:undecidable}). At the same time, non-skew-representability admits certificates in a formal sense. 
Indeed, for every finite matroid $M$, one can effectively construct a first-order sentence $\sigma_M$ asserting the existence of a matrix whose column-dependence relations are exactly those of $M$. Thus $M$ is skew-representable if and only if the first-order axioms for skew fields together with $\sigma_M$ have a model.
By Gödel's completeness theorem~\cite[Section~2.5]{enderton2001mathematical}, if no such model exists, then there is a finite formal derivation of a contradiction from these axioms. Hence skew-representability is co-recursively enumerable, or in other words, non-skew-representability admits finite, algorithmically verifiable certificates. The same completeness argument applies when the characteristic is prescribed, and more generally to the finite or cofinite characteristic restrictions considered below. Combined with the undecidability result of~\cite{kuhne2023von}, this also shows that skew-representability is not recursively enumerable, and therefore skew-representability itself admits no general system of finite, algorithmically verifiable certificates. The certificates provided by the completeness theorem, however, are purely logical.  In contrast, \cref{thm:skewrepr} provides a transparent matroid-theoretic obstruction to skew-representability.

\paragraph{Linear rank inequalities and extension properties.} 

In addition to the results recalled in the introduction, it was shown in~\cite{berczi2025matroid} that the existence of a tensor product with the uniform matroid $U_{2,3}$ already implies Ingleton's inequality.  The first characteristic-dependent linear rank inequalities are due to Blasiak, Kleinberg, and Lubetzky~\cite{blasiak2011lexicographic}, who used the dependencies of the Fano and the non-Fano matroids.  
The inequalities in Section~\ref{sec:two} use the same two matroids, but are obtained instead from the nonexistence of a tensor product between them. Extension properties are closely related to linear rank inequalities, and have proved to be useful both for studying matroid representability and in applications to secret sharing~\cite{bamiloshin2021common,bamiloshin2025note}. 
The matroid counterpart of the common information property is full modular extendability; see~\cite{bamiloshin2025note} for a discussion of these and related extension properties. The tensor product framework of this paper is also closely related to those, and our results develop this connection further for both matroids and polymatroid functions.

\paragraph{Rank-3 matroids.}  Rank-3 matroids play a key role in questions related to representability. One of the classical examples of nonrepresentable matroids is the rank-3 non-Desargues matroid~\cite{ingleton1971representation} (see also \cite[Chapter~6.1]{oxley2011matroid}): it follows from Desargues's theorem in projective geometry that this matroid is not representable over any skew field. 
Another classical example is the rank-3 non-Pappus matroid: it follows from Pappus's theorem that it is not representable over any field, while it is representable over a skew field~\cite{ingleton1971representation}. Every matroid of rank 3 satisfies all known extension properties, so no previously known linear rank inequality can certify the non-representability of any rank-3 matroid. There is also a general reason why rank 3 is particularly relevant to representability: Mason~\cite{mason1977geometric} showed that, in the field case, Dilworth truncation can be used to reduce representability in any prescribed characteristic to the rank-3 case. The same reduction appears to remain valid in the skew-field setting, although we are not aware of a reference. This gives further motivation for focusing on the rank-3 case. Our results show that tensor products can nevertheless detect non-skew-representability in rank $3$. Theorem~\ref{thm:uniform} shows that every rank-$3$ matroid has a tensor product with every uniform matroid, while Theorem~\ref{thm:intro-inequality} derives a linear rank inequality, using a tensor-product obstruction, that is violated by the non-Desargues matroid.

\paragraph{Folded and almost folded representable matroids.} Besides representability over fields or skew fields, several more general notions of representability have been introduced. 
Given a positive integer $k$, a polymatroid function $\varphi$ is said to be \textit{$k$-folded representable}~ \cite{bamiloshin2021common,kuhne2023von} or \textit{$k$-multilinear}~\cite{pendavingh2013skew} if the polymatroid function $k \cdot \varphi$ is representable over a field.
Analogously, replacing the field by a skew field, one gets the notion of \textit{$k$-folded skew-representability}.
A matroid is said to admit these properties if its rank function does.
Representable matroids form a proper subclass of folded representable matroids, as witnessed by the non-Pappus matroid~\cite{pendavingh2013skew}, and skew-representable matroids form a proper subclass of folded skew-representable matroids~\cite{pendavingh2013skew}.
The classes of skew-representable and folded representable matroids are incomparable~\cite{pendavingh2013skew,kuhne2023von}, and folded representability was shown to be undecidable~\cite{kuhne2022arrangements}. Kühne and Yashfe~\cite{kuhne2026entropic} studied almost folded representability, a limiting version of folded representability: roughly speaking, a matroid is almost folded over a field $\bF$ if its rank function is a limit of normalized polymatroid functions $\varphi_i/c_i$ , where $c_i\in\bZ_+$ and each $\varphi_i$ is representable over $\bF$; see~\cite{kuhne2026entropic} for the precise definition.
These notions are also related to our results: the linear rank inequalities derived in Section~\ref{sec:rank} for folded skew-representable polymatroid functions remain valid for the corresponding almost folded class by taking limits. Moreover, Remark~\ref{rem:forintro} shows that folded skew-representability implies the relevant tensor-product conditions for rank functions in the polymatroid sense, although we do not obtain a characterization in these terms.

\paragraph{Freest tensor products.}
Mason~\cite{mason1981glueing} asked whether any two matroids admit a freest tensor-product analogue. For the broader notion of matroid products, Las Vergnas~\cite{las1981products} answered this question negatively. To our knowledge, it remains open whether two matroids admitting a tensor product necessarily admit a freest one in the weak order. The special case of two uniform matroids is of particular interest, as Brakensiek, Dhar, Gao, Gopi, and Larson~\cite{brakensiek2024rigidity} related generic tensor products of uniform matroids to bipartite rigidity matroids. Here, the generic tensor product is conjectured to be the freest tensor product~\cite{cruickshank2026rigidity}. Very recently, Tyomkyn~\cite{tyomkyn2026linear} proved that the generic tensor product of two uniform matroids is freest among their linearly representable tensor products; similar results were obtained independently by Garamvölgyi and Vermant~\cite{garamvolgyi2026inpreparation}. In the case of $U_{3,m}$ and $U_{r,n}$, the generic tensor product coincides with the freest tensor product in Theorem~\ref{thm:uniform}; see Remark~\ref{rem:uniuni}. Thus, in this case Theorem~\ref{thm:uniform} gives freeness among all matroid tensor products, and more generally applies when the rank-$3$ factor is arbitrary. On the other hand, Example~\ref{ex:nonskewfreest} shows that even when both factors are representable, their freest tensor product need not be skew-representable.

\subsection{Organization}
\label{sec:organization}

The rest of the paper is organized as follows. In \cref{sec:prelim}, we provide essential definitions and background on matroids, polymatroids, skew fields, and modular extensions. \cref{sec:fme} presents a characterization of matroids whose connected components have rank 3 or are skew-representable. Technical lemmas concerning tensor products and related operations appear in \cref{sec:prep}, serving as a toolkit for subsequent sections. The characterization of skew-representability via tensor products is developed in \cref{sec:tensor}. Since this section is quite technical, we suggest that first-time readers skip the proofs and simply skim the statements.
\cref{sec:uniform} addresses tensor products involving rank-3 and uniform matroids, and provides a construction for the unique freest tensor product. \cref{sec:rank} focuses on linear rank inequalities, including new proofs and inequalities derived through our tensor-based framework. Finally, we conclude the paper with a list of open problems in \cref{sec:open}.

\section{Preliminaries}
\label{sec:prelim}

In this section, we give an overview of the notation and background used throughout the paper. We recall the necessary definitions and results concerning polymatroid functions and tensor products, matroids and modular extensions, skew-field representability, computability, and projective spaces.

\paragraph{General notation.}

We denote the sets of {\it reals}, {\it rationals}, and {\it integers} by $\bR$, $\bQ$, and $\bZ$, respectively, and add $\geq 0$ or $+$ as a subscript when considering nonnegative or positive values only, respectively. The set of \textit{prime numbers} is denoted by $\bP$. For $k\in\bZ_+$, we use $[k]\coloneqq \{1,\dots,k\}$ while $[0] = \emptyset$ by convention. Given a ground set $S$, a set $X\subseteq S$ and an element $y\in S$, the sets $X\setminus \{y\}$ and $X\cup \{y\}$ are abbreviated as $X-y$ and $X+y$, respectively. 
For a set function $\varphi\colon 2^S\to \bR$ and an element $x\in S$, we often write $\varphi(x)$ for $\varphi(\{x\})$.
If $S=S_1\times S_2$ for some sets $S_1$ and $S_2$, then for any $X\subseteq S_1$, $a\in S_1$, $Y\subseteq S_2$ and $b\in S_2$, we use $X^b=X\times \{b\}$ and ${}^aY=\{a\}\times Y$. 

\paragraph{Polymatroid functions and tensor product.}

A set function $\varphi \colon 2^S \to \bR$ is called {\it increasing} (or \textit{monotone}) if $\varphi(X) \leq \varphi(Y)$ whenever $X \subseteq Y$, and {\it submodular} if $\varphi(X) + \varphi(Y) \geq \varphi(X \cap Y) + \varphi(X \cup Y)$ for all $X, Y \subseteq S$. The latter is equivalent to the inequality $\varphi(X \cup Z) - \varphi(X) \geq \varphi(Y \cup Z) - \varphi(Y)$ for all $X \subseteq Y \subseteq S$ and $Z \subseteq S \setminus Y$. We call $\varphi$ a {\it polymatroid function} if it is increasing, submodular, and $\varphi(\emptyset)=0$. Note that every polymatroid function is nonnegative, since $\varphi(X)\geq\varphi(\emptyset)=0$ for every $X\subseteq S$. We will also use the following equivalent form of these conditions: a set function $\varphi$ with $\varphi(\emptyset)=0$ is a polymatroid function if and only if $\varphi(X+a)\geq \varphi(X)$ for every $X\subseteq S$ and $a\in S\setminus X$, and 
\begin{equation} \label{eq:submod}
\varphi(X+a)+\varphi(X+b)\geq\varphi(X)+\varphi(X+a+b) \text{ for every $X\subseteq S$ and distinct $a,b\in S\setminus X$.}
\end{equation}
The following is a standard fact about polymatroid functions; we include a proof to keep the paper self-contained.

\begin{lem} \label{lem:polymatroid-direct-sum}
    Let $\varphi\colon 2^S\to \bR$ be a polymatroid function, and let $A_1,\dots, A_k\subseteq S$ be pairwise disjoint sets such that $\varphi(\bigcup_{i=1}^k A_i) = \sum_{i=1}^k \varphi(A_i)$. Then, $\varphi(W)=\sum_{i=1}^k \varphi(A_i \cap W)$ holds for each $W\subseteq \bigcup_{i=1}^k A_i$.
\end{lem}
\begin{proof}
    Repeated submodularity, together with $\varphi(\emptyset)=0$, gives $\varphi(W)\leq\sum_{i=1}^k\varphi(A_i\cap W)$. 
    For the reverse inequality, we use that 
    \[\varphi(A_i)\ge \varphi(W\cap A_i)+\varphi(W\cup \bigcup_{j=1}^i A_j)-\varphi(W\cup \bigcup_{j=1}^{i-1}A_j)\]
    holds for each $i \in [k]$ by submodularity.
    Since $W\subseteq \bigcup_{i=1}^ k A_i$, summing these inequalities over $i \in [k]$ yields
    \[\varphi(\bigcup_{i=1}^k A_i) =\sum_{i=1}^k \varphi(A_i) \ge \sum_{i=1}^k \varphi(W\cap A_i)+\varphi(\bigcup_{i=1}^k A_i)-\varphi(W),\]
    which concludes the proof of the lemma.
\end{proof}

Let $\varphi_1\colon 2^{S_1}\to\bR$ and $\varphi_2\colon 2^{S_2}\to\bR$ be polymatroid functions defined on ground sets $S_1$ and $S_2$, respectively. Then $\varphi\colon 2^{S_1\times S_2}\to\bR$ is a {\it tensor product} of $\varphi_1$ and $\varphi_2$ if $\varphi$ is also a polymatroid function and $\varphi(X_1\times X_2)=\varphi_1(X_1)\cdot\varphi_2(X_2)$ holds for all $X_1\subseteq S_1$ and $X_2\subseteq S_2$. Since a tensor product need not exist, and when it exists it need not be unique, we denote the (possibly empty) {\it set of tensor products} of $\varphi_1$ and $\varphi_2$ by $\varphi_1\otimes \varphi_2$. It is worth noting that the tensor product may not be integer-valued even if $\varphi_1$ and $\varphi_2$ are.

Given polymatroid functions $\varphi$ and $\psi$, we define the set $T_k(\varphi,\psi)$ to consist of all polymatroid functions that can be obtained by performing a sequence of $k$ (left-associative) tensor products of $\varphi$ with $\psi$. That is, $T_0(\varphi,\psi)=\{\varphi\}$, and $T_k(\varphi,\psi)=\bigcup_{\varphi'\in T_{k-1}(\varphi,\psi)}\varphi'\otimes \psi$ for $k\in \bZ_+$. We say that $\varphi$ is {\it $k$-tensor-compatible with $\psi$} if $T_k(\varphi, \psi)$ is nonempty.

Let $\psi\colon 2^T\to\bR$ be a set function on some finite ground set $T$, and let $\cQ=(T_1,\dots,T_q)$ be a partition of $T$ into $q$ (possibly empty) parts. The {\it quotient} of $\psi$ with respect to $\cQ$ is the set function $(\psi/\cQ)$ on $T_\cQ=\{t_1,\dots,t_q\}$ defined by $(\psi/\cQ)(X)=\psi(\bigcup_{t_i\in X} T_i)$ for $X\subseteq T_\cQ$. Observe that taking a quotient preserves the properties of being zero on the empty set, increasing, and submodular; hence, the quotient of a matroid rank function is an integer-valued polymatroid function. 

\paragraph{Matroids.} 

For basic definitions on matroids such as independent sets, bases, circuits, loops, and rank, we refer the reader to~\cite{oxley2011matroid}; see also~\cite{berczi2025matroid} for a quick introduction. Recall that a {\it uniform matroid} of rank $k$ on a ground set of size $n$ has rank function $r(X)=\min\{|X|,k\}$ and is denoted by $U_{k,n}$.

A matroid is {\it simple} if it has no loops or parallel elements. Given a matroid $M=(S,r)$ on a finite ground set $S$, we denote by $\si(M)$ the matroid obtained by deleting all loops and retaining exactly one element from each parallel class; we refer to $\si(M)$ as the {\it simplification of $M$}. A {\it flat} is a set $F\subseteq S$ such that $r(F+e)=r(F)+1$ holds for every $e\in S\setminus F$. We refer to a flat of rank $2$ as a {\it line}, and call a line {\it nontrivial} if it contains at least three parallel classes of non-loop elements. For a set $X\subseteq S$, the unique smallest flat containing $X$ is called the {\it closure} of $X$ and is denoted by $\cl_M(X)$; we omit the subscript when the matroid $M$ is clear from the context. A pair $F_1,F_2$ of flats is called {\it modular} if $r(F_1)+r(F_2)=r(F_1\cap F_2)+r(F_1\cup F_2)$. A matroid is {\it modular} if any pair of its flats is modular. For sets $X,Y\subseteq S$, we say that $X$ {\it spans} $Y$ in $M$ if $r(X\cup Y)=r(X)$, or equivalently, $Y\subseteq\cl_M(X)$. A {\it spanning set} is a set spanning the ground set of the matroid.

For a subset $S' \subseteq S$, the {\it restriction} of $M$ to $S'$ and the {\it deletion} of $S\setminus S'$ from $M$ yield the same matroid, denoted by $M|S' = M \backslash (S\setminus S') = (S', r')$, where $r'$ is the restriction of $r$ to subsets of $S'$. The {\it contraction of $S'$} and the {\it contraction to $S\setminus S'$} likewise result in the same matroid, denoted by $M/S' = M.(S\setminus S') = (S\setminus S', r')$, with rank function $r'(X)=r(X\cup S')-r(S')$ for each $X \subseteq S\setminus S'$. A matroid $N$ that can be obtained from $M$ by a sequence of restrictions and contractions is called a {\it minor} of $M$. 

Let $M_1=(S_1,r_1)$ and $M_2=(S_2,r_2)$ be matroids on disjoint ground sets. Their {\it direct sum}  $M_1\oplus M_2$ is the matroid $M=(S_1\cup S_2,r)$, where $r(X)=r_1(X\cap S_1)+r_2(X\cap S_2)$ for all $X\subseteq S_1\cup S_2$. For a matroid $M=(S,r)$, let $S=S_1\cup\dots\cup S_q$ be the partition of its ground set into the connected components of the hypergraph of circuits of $M$; it is known that two distinct elements fall in the same $S_i$ if and only if the matroid has a circuit containing both of them. Then the matroids $M_i=M|S_i$ for $i\in[q]$ are called the {\it components} of $M$ and $M=M_1\oplus\dots\oplus M_q$. We say that $M$ is {\it connected} if $q=1$, or in other words, there exists a circuit through any two distinct elements. The following simple observation will play an important role in this paper.

\begin{lem}\label{lem:u23minor}
Any connected matroid $M=(S,r)$ of rank at least $2$ contains $U_{2,3}$ as a minor.
\end{lem}
\begin{proof}
Let $x,y\in S$ be such that $r(\{x,y\})=2$. Since the matroid is connected, there exists a circuit $C$ such that $\{x,y\}\subseteq C$ and $|C|=r(C)+1\geq r(\{x,y\})+1=3$. Any circuit of size at least $3$ contains $U_{2,3}$ as a minor, so the lemma follows.
\end{proof}

We adapt the terminology and notation from polymatroid functions to matroids as follows. We call a matroid $M=(S_1\times S_2,r)$ a {\it tensor product} of $M_1$ and $M_2$ if $r$ is a tensor product of $r_1$ and $r_2$, that is, $r(X\times Y)=r_1(X)\cdot r_2(Y)$ holds for all $X\subseteq S_1,Y\subseteq S_2$. In particular, this implies that if $x\in S_1$ and $y\in S_2$ are nonloops, then $M|S^y_1$ and $M|{}^xS_2$ are isomorphic to $M_1$ and $M_2$, respectively. The {\it set of tensor products} of $M_1$ and $M_2$ is denoted by $M_1\otimes M_2$. More generally, if $\cM_1$ and $\cM_2$ are sets of matroids, we define $\cM_1\otimes \cM_2\coloneqq\bigcup_{M_1\in\cM_1,M_2\in\cM_2}(M_1\otimes M_2)$.
Analogously to the polymatroid case, for matroids $M$ and $N$, we set $T_0(M,N)\coloneqq \{M\}$ and $T_k(M,N)\coloneqq T_{k-1}(M, N) \otimes N$ for $k \in \bZ_+$. It is not difficult to verify that if $N$ has rank at most $1$ or is a rank-$2$ matroid on two elements, then $T_k(M,N)\neq\emptyset$ for all matroids $M$ and all $k\in\bZ_+$. Thus, the smallest matroid $N$ for which the set $T_k(M,N)$ is nontrivial is $U_{2,3}$. This matroid is also of particular interest due to \cref{lem:u23minor}.
Las Vergnas~\cite{las1981products} gave the following characterization of tensor products.

\begin{prop}[Las Vergnas] \label{prop:tensor}
Let $M_1=(S_1, r_1)$, $M_2=(S_2, r_2)$ and $M=(S, r)$ be matroids such that $S=S_1 \times S_2$. Then, $M$ is a tensor product of $M_1$ and $M_2$ if and only if $r(\prd{e_1}{Y_2}) = r_1(e_1)\cdot r_2(Y_2)$ for each $Y_2\subseteq S_2$ and $e_1 \in S_1$, $r(Y_1^{e_2}) = r_1(Y_1)\cdot r_2(e_2)$ for each $Y_1\subseteq S_1$ and $e_2 \in S_2$, and $r(S) = r_1(S_1)\cdot r_2(S_2)$.
\end{prop}

We will also use the following properties of tensor products, also from~\cite{las1981products}. 

\begin{prop}[Las Vergnas] \label{prop:crosses}
    Let $M=(S,r)$ be a tensor product of the matroids $M_1 = (S_1, r_1)$ and $M_2=(S_2, r_2)$, and let $X_1 \subseteq X_2 \subseteq S_1$ and $Y_2 \subseteq Y_1 \subseteq S_2$.
    Then, 
    \[ r(X_1\times Y_1)+r(X_2\times Y_2)=r((X_1\times  Y_1)\cap (X_2\times Y_2))+r((X_1\times Y_1)\cup (X_2\times Y_2)).\]
    Moreover, if $X_1$ and $X_2$ are flats of $M_1$ and $Y_1$ and $Y_2$ are flats of $M_2$, then $(X_1\times Y_1)\cup (X_2\times Y_2)$ is a flat of $M$. 
\end{prop}

In some cases, we consider matroids on infinite ground sets. Given a possibly infinite ground set $S$, a family $\cI\subseteq 2^S$ forms the independent sets of a {\it finitary matroid} if (I1) $\emptyset\in\cI$, (I2) if $I\in\cI$ and $I'\subseteq I$, then $I'\in\cI$, (I3) for all $I,I'\in\cI$ with $|I|<|I'|$ and $I'$ finite, there exists $e\in I'\setminus I$ such that $I+e\in\cI$, and (I4) $I\in\cI$ if and only if every finite subset of $I$ belongs to $\cI$. In other words, the independence system of a finitary matroid behaves locally like that of a finite matroid: dependence can always be witnessed on finite subsets. The {\it rank function} of a finitary matroid is the function $r\colon 2^S\to \bZ_{\ge 0}\cup \{\infty\}$ defined by $r(X)\coloneqq \sup \{|I|\mid I \subseteq X, I \in \cI\}$ for $X\subseteq S$.
Note that $\cI = \{I\subseteq S \mid r(I')=|I'| \text{ for every finite $I'\subseteq I$}\}$; hence we may specify a finitary matroid by its rank function as well. We call a finitary matroid $M$ {\it modular} if its lattice of flats is modular; that is, for flats $F_1,F_2,F_3$ with $F_1\subseteq F_3$, $F_1\vee(F_2\cap F_3)=(F_1\vee F_2)\cap F_3$, where $F_1\vee F_2=\cl_M(F_1\cup F_2)$. For finite-rank matroids, this is equivalent to requiring $r(F_1)+r(F_2)=r(F_1\cap F_2)+r(F_1\vee F_2)$ for every pair of flats $F_1,F_2$. 

Throughout the paper, we use `matroid' to refer to a matroid on a finite ground set, and explicitly write `finitary matroid' when the ground set may be infinite.

\paragraph{Computability theory.}

We refer the reader to \cite{papadimitriou1994computational} for the basics of computability theory. A language $L$ is {\it recursively enumerable}\footnote{The recursively enumerable property is sometimes called {\it semi-decidable} or {\it computably enumerable} in the literature.} if there exists a Turing machine that, given an input string $x$, outputs ``yes'' if $x \in L$ and never halts if $x \notin L$. Similarly, $L$ is {\it co-recursively enumerable} if there exists a Turing machine that, given an input string $x$, outputs ``no'' if $x \notin L$ and never halts if $x \in L$. It is not difficult to see that a language $L$ that is both recursively enumerable and co-recursively enumerable is {\it decidable}; that is, there exists a Turing machine that, given an input string $x$, outputs ``yes'' if $x \in L$ and ``no'' if $x \notin L$.
For computability questions, matroids are given by any standard explicit finite encoding, for example by their rank functions; the particular choice is immaterial, since the usual finite encodings are computably interconvertible.

\paragraph{Skew fields and representability.}  

We refer the reader to \cite{cohn1995skew} on skew fields. Let $\bF$ be a skew field. The {\it center} of $\bF$ is $Z(\bF) \coloneqq \{x \in \bF \mid xy = yx \text{ for all }y\in \bF\}$, which is a subfield of $\bF$. The subfield generated by $1$, which is contained in $Z(\bF)$, is isomorphic to either $\bQ$ or to $\bF_p$, the field of integers modulo $p$, for some prime number $p$. Accordingly, $\bF$ is said to have characteristic $0$ or $p$.

The \emph{rank} $r(A)$ of a matrix $A \in \bF^{m \times n}$ is the dimension of its right column space. Most basic linear algebra carries over from fields to the more general skew field setting; in particular, $r(A)$ also equals the dimension of the left row space of $A$, as well as the largest number $k$ such that $A$ has a $k\times k$ invertible submatrix (left and right inverses of square submatrices coincide); see e.g.~\cite[Theorem~1.4.7]{cohn2006free} for a more general statement.
Given matrices $A \in \bF^{m \times n}$ and $B\in \bF^{p \times q}$, their {\it Kronecker product} $A\otimes B \in \bF^{mp \times nq}$ is the matrix
\[A \otimes B = \begin{bmatrix} a_{11}B & \dots & a_{1n}B  \\ \vdots & \ddots & \vdots\\ a_{m1}B & \dots & a_{mn}B \end{bmatrix}.\]
For a matrix $A$ and $k\in\bZ_+$, we write $A^{\otimes k}$ for the $k$-fold Kronecker product of $A$ with itself.
Since some properties of Kronecker products do not carry over from the commutative to the skew field setting, we will review some of their properties in \cref{subsec:skew}.

For ease of reading, we recall some of the definitions from the introduction. A matroid is {\it representable over a skew field $\bF$} (or {\it $\bF$-representable} for short) if there exists a family of vectors in a right vector space over $\bF$ whose linear independence relation matches the matroid's, or equivalently, there exists a matrix $A\in \bF^{m \times n}$ whose columns are indexed by the ground set such that the rank of every set equals the rank of the submatrix formed by the corresponding columns. It is known that $\bF$-representability is preserved under taking minors. Similarly, a polymatroid function $\varphi\colon 2^S \to \bZ$ is {\it representable over $\bF$} if there exists a family of subspaces of a right vector space over $\bF$ such that for each subset $X \subseteq S$, $\varphi(X)$ equals the dimension of the subspace spanned by the union of the corresponding subspaces. Equivalently, $\varphi$ is representable over $\bF$ if there is a matrix $A\in \bF^{m \times n}$ and a partition $\cP = \{P_s \mid s \in S\}$ of $[n]$ such that for each $X\subseteq S$, $\varphi(X)$ equals the rank of the submatrix of $A$ formed by the columns with index set $\bigcup_{x \in X} P_x$. Note that a matroid is $\bF$-representable exactly when its rank function is an $\bF$-representable polymatroid function. A matroid or  polymatroid function is {\it representable} if it is representable over some field, and {\it skew-representable} if it is representable over some skew field. A matroid is {\it regular} if it is representable over all fields.  A polymatroid function $\varphi$, or a matroid with rank function $\varphi$, is called $k$-folded representable (respectively, $k$-folded skew-representable) if $k\cdot\varphi$ is representable (respectively, skew-representable). It is folded representable (respectively, folded skew-representable) if this holds for some positive integer $k$.

The {\it (skew) characteristic set} of a matroid $M$ is the set of $p\in\bP\cup\{0\}$ for which $M$ is representable over a (skew) field of characteristic $p$. The following statement summarizes several results on the characteristic sets of matroids~\cite{rado1957indep,vamos1975necessary,kahn1982characteristic,brylawski1980matroids}.

\begin{prop}[Rado, Vámos, Kahn, and Reid] \label{prop:charsets}
    Let $C \subseteq \bP \cup \{0\}$. If $C$ is the characteristic set of a matroid, then $C$ satisfies one of the following conditions:
    \begin{enumerate}[label=(\arabic*)] \itemsep0em
        \item $0 \notin C$ and $|C| < \infty$. \label{it:finite_charset}
        \item $0 \in C$ and $|\bP \setminus C| < \infty$. \label{it:cofinite_charset}
    \end{enumerate}
    Moreover, if $C$ is the skew characteristic set of a matroid and $0\notin C$, then $|C|<\infty$. Conversely, if $C$ satisfies \ref{it:finite_charset} or \ref{it:cofinite_charset}, then there exists a connected rank-3 matroid $M$ such that $C$ is both the characteristic and the skew characteristic set of $M$, and for each $p \in C$, $M$ is representable over all infinite fields of characteristic $p$. Moreover, given $C$ in case \ref{it:finite_charset} and $\bP \setminus C$ in case \ref{it:cofinite_charset}, such a matroid $M$ is computable.
\end{prop}
It was first shown by Rado~\cite{rado1957indep} and Vámos~\cite{vamos1975necessary} that the characteristic set of a matroid satisfies \ref{it:finite_charset} or \ref{it:cofinite_charset}.
The observation that a skew characteristic set not containing 0 is finite is attributed to Vámos; see \cite{evans1991projective}.
Unlike the commutative case, $|\bP\setminus C|<\infty$ need not hold for a skew characteristic set $C$ containing $0$; for example, $\{0\}$ arises as a skew characteristic set~\cite{evans1991projective}.
Kahn~\cite{kahn1982characteristic} showed that each set $C$ satisfying \ref{it:finite_charset} is the characteristic set of a matroid. In his proof, given the set $C$, he constructs, in a computable manner, a rank-3 matroid $M$ whose skew characteristic set is $C$ and which is representable over all sufficiently large fields of characteristic $p$ for all $p\in C$.
By deleting loops, we may assume that $M$ is loopless. Every disconnected loopless matroid of rank at most three is representable over every infinite field, so, as $0\notin C$, $M$ must be connected.
Reid proved, as published in \cite[pages 101--102]{brylawski1980matroids}, that any set $C$ satisfying \ref{it:cofinite_charset} is the characteristic set of a matroid; see also \cite[Lemma~6.8.5]{oxley2011matroid}. The proof is constructive (given the set $\bP\setminus C$), and it provides a rank-3 matroid $M$ with characteristic set $C$ such that $M$ is representable over all infinite fields with characteristic in $C$. The condition that $C$ also coincides with the skew characteristic set of $M$ can further be imposed; this also follows from the combination of \cite[Lemma~3.4.1]{evans1991projective} and \cite[Lemma~10]{cartwright2024characteristic}. 
By deleting loops, we may assume that $M$ is loopless. If $C\neq\mathbb{P}\cup\{0\}$, then $M$ must be connected, since every disconnected loopless matroid of rank at most three is representable over every infinite field. If $C=\mathbb{P}\cup\{0\}$, we may instead take the connected regular matroid $U_{3,4}$. 

The following result describes how (skew) characteristic sets behave under direct sums.

\begin{prop} \label{prop:direct-sum-charset}
    The (skew) characteristic set of a matroid equals the intersection of the (skew) characteristic sets of its components.
\end{prop}

Indeed, the (skew) characteristic set of a matroid is clearly contained in the intersection of the (skew) characteristic sets of its components.
Fields of the same characteristic can be embedded into a common field.
The analogous statement for skew fields is due to
Cohn~\cite{cohn1971embedding}; see also
\cite[Chapter~5]{cohn1995skew}. Together, these facts imply \cref{prop:direct-sum-charset}, as the direct sum of $\bF$-representable matroids is $\bF$-representable for any skew field $\bF$.

Recently, Kühne, Pendavingh, and Yashfe~\cite{kuhne2023von} showed the following hardness result. While they did not include the rank-3 condition in their theorem, their proof shows that it holds in this restricted case as well. 
By deleting loops, we may restrict to loopless matroids.
Since every disconnected loopless rank-$3$ matroid is representable over every infinite field, the connectivity condition can also be imposed.

\begin{prop}[see Kühne, Pendavingh, and Yashfe] \label{prop:undecidable}
    The following problems are undecidable.
    \begin{enumerate}[label = (\alph*)]\itemsep0em
        \item \label{it:skew-undec} Given a connected rank-3 matroid, decide whether it is skew-representable.
        \item \label{it:fixed-undec} Let $p$ be a prime number or zero. Given a connected rank-3 matroid, decide whether there exists a skew field of characteristic $p$ over which it is representable.
    \end{enumerate}
\end{prop}

\paragraph{Modular extensions.}  

We call every integer-valued polymatroid function $\varphi\colon 2^S\to\bZ$ {\it $0$-modular extendable}. For $A,B\subseteq S$, an {\it extension step} with respect to $A$ and $B$ results in a polymatroid function $\varphi'\colon 2^{S'}\to\bZ$ with $S\subseteq S'$ such that there exists $Z\subseteq S'$ satisfying $\varphi'(X)=\varphi(X)$ for all $X\subseteq S$, $\varphi'(A\cup Z)=\varphi(A)$, $\varphi'(B\cup Z)=\varphi(B)$, and $\varphi'(Z)=\varphi(A)+\varphi(B)-\varphi(A\cup B)$. We call such a set $Z$ a \emph{witnessing set} for the extension step. For simplicity, we call $\varphi'$ a {\it one-step extension} of $\varphi$. 
As elements of $S'\setminus (S\cup Z)$ may be deleted, we may assume $S'=S\cup Z$. 
Then, the extension step does not increase the function value on the entire ground set, so $\varphi'(S') = \varphi(S)$. 
Moreover, when considering a single extension step, we may assume that $S'=S+z$ for some new element $z$ and that the witnessing set is $\{z\}$. Indeed, for each $e\in Z\cap S$, add a new copy $e'$ of $e$ and replace $e$ by $e'$ in $Z$. The resulting witnessing set is disjoint from $S$, and taking its quotient to a single new element gives the desired one-step extension.

For $k\in\bZ_+$, $\varphi$ is said to be {\it $k$-modular extendable} if, for all subsets $A,B\subseteq S$, there exists an integer-valued $(k-1)$-modular extendable polymatroid function that is obtained from $\varphi$ by an extension step with respect to $A$ and $B$. A polymatroid function is {\it fully modular extendable} if it is $k$-modular extendable for all $k\in\bZ_+$. If the integrality assumption is dropped throughout, the corresponding properties are called {\it $k$-fractionally modular extendable} and {\it fully fractionally modular extendable}, respectively. 

We adapt the notation from set functions to matroids in the following sense: define a matroid $M=(S,r)$ to be {\it $k$-modular extendable}, {\it fully modular extendable}, {\it $k$-fractionally modular extendable}, or {\it fully fractionally modular extendable} when $r$ has these properties and, in addition, every extension occurring in the recursive definition can be chosen subcardinal, that is, it satisfies $r'(X)\leq |X|$ for all $X\subseteq S'$.\footnote{For polymatroid functions, the $k$-fractionally modular and fully fractionally modular extension properties are sometimes referred to as {\it $k$-CI-compliant} and {\it common information} properties, respectively. Meanwhile, for matroids, $k$-modular and fully modular extensions are also called {\it $k$-complete Euclidean} and {\it complete Euclidean} extensions, respectively. To emphasize that both cases use the same modular extension construction, with the difference between the integer and fractional cases indicated by adding ``fractionally'', and that, in the matroid case, the extended function must be subcardinal, we decided to adopt the new terminology. For further details, see~\cite{bamiloshin2025note}.} In the non-fractional case this means that all of these extensions are matroids. It is not difficult to see that for a matroid $M'=(S',r')$ with $S\subseteq S'$ and $M'|S=M$, $M'$ is a one-step extension of $M=(S,r)$ with respect to $A$ and $B$ if and only if the closures $\cl_{M'}(A)$ and $\cl_{M'}(B)$ form a modular pair. Indeed, the `if' direction follows by setting $Z=\cl_{M'}(A)\cap\cl_{M'}(B)$. For the `only if' direction, observe that 
$r'(\cl_{M'}(A)\cap \cl_{M'}(B))
\geq  r'(Z)
= r(A)+r(B)-r(A\cup B)
=r'(\cl_{M'}(A))+r'(\cl_{M'}(B))-r'(\cl_{M'}(A)\cup \cl_{M'}(B))$.
Since the reverse inequality holds by submodularity of $r'$, equality holds throughout. Unlike in the polymatroid case, we cannot assume that $|S'\setminus S| \leq 1$. However, we may assume that $S'\setminus S$ is independent. Indeed, taking a maximal independent set of $Z\setminus S$ and deleting the rest of the elements yields a matroid that is still an extension of $M$. Consequently, we may assume $|S'\setminus S| \leq r(S)$.

\paragraph{Projective spaces.}  

Let $\cP$ be a nonempty, not necessarily finite set, and let $\mathbf{P}=(\cP, \mathcal{L})$ be a pair with  $\mathcal{L}\subseteq 2^{\cP}$. We call the elements of $\cP$ {\it points}, the elements of $\mathcal{L}$ {\it lines} and we say that a point $p$ and a line $L$ are {\it incident} if $p\in L$. We call $(\cP,\mathcal{L})$ a {\it projective space} if it satisfies the following axioms:
 \begin{enumerate}[label=(P\arabic*)]\itemsep0em
     \item For any two distinct lines, there is at most one point that is incident to both of them. \label{ax:p1}
     \item For any two distinct points $p,q\in\cP$, there is exactly one line that is incident to both of them, denoted $\overline{pq}$. \label{ax:p2}
     \item If lines $\overline{ab}$ and $\overline{cd}$ intersect for pairwise distinct points $a$, $b$, $c$ and $d$, then so do the lines $\overline{ac}$ and $\overline{bd}$. \label{ax:p3}
     \item Every line is incident to at least three points. \label{ax:p4}
 \end{enumerate}
Note that the above definition allows for the degenerate case where $\cP$ consists of a single element and $\cL$ is empty. 
If $(\cP, \mathcal{L})$ satisfies \ref{ax:p1}, \ref{ax:p2}, \ref{ax:p3}, and 
\begin{enumerate}\itemsep0em
    \item[(P4')] every line contains at least two points, \label{ax:p4prime}
\end{enumerate}
then we call $(\cP, \mathcal{L})$ a {\it generalized projective space}. Given a possibly infinite family of projective spaces $\mathbf{P}_j=(\cP_j,\cL_j)$ for $j\in J$ over pairwise disjoint point sets, their {\it direct sum} $\bigoplus_{j\in J}\mathbf{P}_j$ is defined as the pair $(\cP,\cL)$, where $\cP=\bigcup_{j\in J}\cP_j$ and $\cL=(\bigcup_{j\in J} \mathcal{L}_j) \cup \left\{\{p,q\}\mid j,k\in J, j \ne k, p \in \cP_j, q \in \cP_k \right\}$.
Note that the direct sum of projective spaces is not necessarily a projective space, but it is always a generalized projective space. Conversely, every generalized projective space arises as such a direct sum.
 
 \begin{prop}[see {\cite[Chapter 1, Exercise 23]{beutelspacher}}]\label{prop:projsum}
    For every generalized projective space $\mathbf{P}=(\cP,\cL)$, there exists a collection $\{\mathbf{P}_j\}_{j\in J}$ of projective spaces such that $\mathbf{P}= \bigoplus_{j\in J} \mathbf{P}_j$. 
\end{prop}

The proof idea of \cref{prop:projsum} is to define an equivalence relation on $\cP$ by declaring $a\sim b$ if $a=b$ or the line $\overline{ab}$ contains at least three points. The transitivity of $\sim$ follows from axiom~\ref{ax:p3}. The equivalence classes then induce projective spaces yielding the desired decomposition.

Let $\mathbf{P}=(\cP,\cL)$ be a generalized projective space. We say that a subset $W$ of $\cP$ is a {\it subspace} of $\mathbf{P}$ if, for any pair $p,q\in W$ of distinct points, $\overline{pq}\subseteq W$.
For a subspace $W$ of $\mathbf{P}$, define its \emph{dimension} as
\[\dim(W)\coloneqq \sup \left\{n \in \bZ \mid \emptyset = W_{-1}\subsetneq W_0 \subsetneq \dots \subsetneq W_n = W \text{ is a chain of subspaces}, n \ge -1\right\}.\]
Note that if $\dim(W)< \infty$, then every maximal chain $\emptyset = W_{-1} \subsetneq W_0 \subsetneq \dots \subsetneq W_n = W$ of subspaces has the same length.
The \emph{dimension} of $\mathbf{P}$ is defined as $\dim(\mathbf{P})\coloneqq \dim(\cP)$.
The subspace \emph{generated} by a subset $T$ of $\cP$ is the smallest subspace containing $T$, and is denoted by $\langle T\rangle$.

The following correspondence is standard. Its lattice-theoretic form, relating projective spaces and modular geometric lattices, can be found in \cite[Theorem~425]{gratzer2011lattice}; see also \cite[Theorem~1.7.5]{oxley2011matroid} for the finite case. Since we will use the correspondence for possibly infinite ground sets, we sketch the proof below.

\begin{prop}\label{prop:proj-matroid} \phantom{x}
    \begin{enumerate}[label=(\alph*)]\itemsep0em
        \item Let $\mathbf{P}=(\cP, \cL)$ be a generalized projective space. Define $r\colon 2^\cP\to \bZ_{\ge 0}\cup\{+\infty\}$ by $r(T)=\dim(\langle T\rangle)+1$ for $T\subseteq \cP$. Then, $M_{\mathbf{P}}\coloneqq (\cP, r)$ is a simple, modular, finitary matroid, whose flats are exactly the subspaces of $\mathbf{P}$. \label{it:proj-to-matroid}
        \item Let $M=(S,r)$ be a simple, modular, finitary matroid. Let $\cL$ denote the set of lines of $M$. Then, $\mathbf{P}_M\coloneqq (S,\cL)$ is a generalized projective space, the subspaces of which are exactly the flats of the matroid. \label{it:matroid-to-proj}
    \end{enumerate}
    Moreover, these two constructions are inverses to each other: $\mathbf{P}_{M_{\mathbf{P}}}=\mathbf{P}$ and $M_{\mathbf{P}_M}=M$ hold for each generalized projective space $\mathbf{P}$ and each simple, modular, finitary matroid $M$.
\end{prop}

For~\ref{it:proj-to-matroid}, the dimension formula for subspaces gives modularity on subspaces and hence submodularity of $r$, while the remaining rank axioms are immediate. The resulting matroid is simple and finitary, while the flats are precisely the subspaces. For~\ref{it:matroid-to-proj}, the rank-$2$ flats satisfy the axioms of a generalized projective space, with axiom~\ref{ax:p3} following from modularity. The subspaces of the resulting generalized projective space are exactly the flats of the matroid, where the finitary assumption takes care of the possibly infinite case. Finally, lines are precisely the rank-$2$ flats and rank is recovered from the dimension of the corresponding subspaces, so the two constructions are inverse to each other.

The correspondence in \cref{prop:proj-matroid} is compatible with direct sums: if $\mathbf{P}=\bigoplus_{j\in J}\mathbf{P}_j$, then $M_{\mathbf{P}}=\bigoplus_{j\in J}M_{\mathbf{P}_j}$.

We next recall the connection between projective spaces and vector spaces over skew fields. Let $\bF$ be a skew field and $V$ be a nonzero right vector space over $\bF$.
For $v, w \in V\setminus \{0\}$, let $v\sim w$ if there exists a scalar $\lambda\in\bF\setminus \{0\}$ such that $w = v\cdot \lambda$. 
The projective space $\mathbf{P}(V)\coloneqq ((V\setminus \{0\})/\sim, \cL(V))$, where $\cL(V)$ consists of the images of the two-dimensional subspaces of $V$ under $\sim$, is called the \emph{projectivization} of $V$.
Its subspaces are precisely the images of the linear subspaces of $V$ under $\sim$.
For $n\in\bZ_{\ge 0}$, $\bF\mathbf{P}^n\coloneqq \mathbf{P}(\bF^{n+1})$ is called the {\it $n$-dimensional projective space over $\bF$}. For a prime $p$, $PG(n,p)$ denotes the matroid $M_{\bF_p\mathbf{P} ^n}$. Equivalently, $PG(n, p)$ is the matroid represented by the columns of a matrix over $\bF_p$ whose columns form a set of representatives of the one-dimensional subspaces of $\bF_p^{n+1}$.
 
Two projective spaces $\mathbf{P}_1,\mathbf{P}_2$ are {\it isomorphic}, denoted by $\mathbf{P}_1\cong\mathbf{P}_2$, if there is a bijection between their point sets that preserves lines when applied in either direction. We will use the following fundamental result of Veblen and Young~\cite{veblen1918projective} on representations of projective spaces; see also~\cite[Theorem 3.4.2]{beutelspacher} for a more recent treatment.

\begin{prop}[Veblen and Young]\label{prop:vy}
Any projective space of dimension at least $3$ is isomorphic to $\mathbf{P}(V)$ for some nonzero vector space $V$ over a skew field $\bF$.
\end{prop}

\section{Fully modular extendable matroids}
\label{sec:fme}

The goal of this section is to establish a connection between full modular extendability and skew-rep\-re\-sen\-ta\-bi\-li\-ty of matroids. The correspondence between generalized projective spaces and modular matroids from Proposition~\ref{prop:proj-matroid}, together with the Veblen--Young theorem, will be the main geometric ingredient.

Recall that a matroid is called fully modular extendable if it is $k$-modular extendable for every $k\in\bZ_+$. Roughly speaking, this means that for every $k$, the matroid can be extended $k$ times while satisfying the required rank conditions. However, note that the definition does not guarantee that the extensions for different values of $k$ are compatible in the sense that any matroid obtained after $k$ extensions can always be further extended by one more step. In principle, it could happen that for each finite $k$ a different sequence of extensions is required. The main message of the following lemma is that there exists a single compatible infinite sequence of extensions that works simultaneously for all $k$. 

\begin{lem}\label{lem:closure}
Let $M=(S,r)$ be a fully modular extendable matroid. Then, there exists a finitary matroid $M'=(S',r')$ with $S \subseteq S'$ such that $M$ is the restriction of $M'$ to $S$ and $M'$ is modular.
\end{lem}
\begin{proof} 
We may assume that $r(S)\ge 1$. Since $M$ is $1$-modular extendable, for any pair $A,B \subseteq S$, there exists a set $Z$ and a matroid $M'=(S \cup Z,r')$ such that $r'(A \cup Z) = r(A)$, $r'(B \cup Z) = r(B)$, and $r'(Z) = r(A) + r(B) - r(A \cup B)$. That is, there is an extension that is modular on the flats generated by $A$ and $B$ and whose rank equals the rank of the original matroid. Recall that we may assume that the extension contains at most $r(S)$ new elements. By adding loops if necessary, we may therefore assume
that its ground set contains exactly $r(S)$ new elements.
The basic idea is that, given there are only finitely many matroids on $|S|+r(S)$ elements and $M$ is fully modular extendable, at least one of these extensions is fully modular extendable as well, allowing us to continue the extension process. We formalize this as follows.

Let $S_0 \coloneqq S$, let $\cP_0$ be the list of all pairs of subsets of $S_0$, and set $\cM_0 \coloneqq \{M\}$. For $i\in\bZ_+$, let $(A_i, B_i)$ be the $i$-th pair in the list $\cP_{i-1}$, and let $S_{i} \supseteq S_{i-1}$ be a set such that $|S_{i} \setminus S_{i-1}| = r(S)$. Define $\cM_{i}$ to be the collection of all one-step extensions of matroids in $\cM_{i-1}$ on the ground set $S_{i}$ with respect to the pair $(A_i, B_i)$. Set $\cP_{i}$ to be the list obtained from $\cP_{i-1}$ by adding all pairs in $(2^{S_{i}}\times 2^{S_{i}})\setminus (2^{S_{i-1}}\times 2^{S_{i-1}})$ to the end of the list $\cP_{i-1}$.

Since each $\cP_i$ is an initial segment of $\cP_{i+1}$, these lists determine a single infinite list $\cP$, in which every pair
retains the index assigned when it first appears.
We define an infinite rooted tree $T$ whose vertices are partitioned into levels, with the $i$-th level consisting of the matroids in $\cM_i$. Two vertices $N_1\in\cM_i$ and $N_2\in\cM_{i+1}$ are connected by an edge if and only if $N_2$ is a one-step extension of $N_1$. The tree $T$ is locally finite, that is, all vertices of $T$ have finite degree, as each level contains only finitely many matroids. By the full modular extendability of $M$, $T$ has a vertex at every level. So by K\H{o}nig's lemma~\cite{konig1927schlussweise}, there exists an infinite path in $T$ starting at $M$. Let $M_0=M,M_1,M_2,\dots$ be such an infinite path, where $M_j=(S_j, r_j)\in\cM_j$. Set $S'\coloneqq \bigcup_{i=0}^\infty S_i$. For each $Y\subseteq S'$, the sequence $(r_j(Y\cap S_j))_{j \ge 0}$ is nondecreasing and it is bounded above by $r(S)$, hence it is eventually constant.
Hence the function $r'$ defined by $r'(Y)\coloneqq \sup\{r_j(Y \cap S_j)\mid j\in\bZ_{\ge 0}\}$ for $Y\subseteq S'$ is the rank function of a finitary matroid. The rank of the ground set remains unchanged throughout the process, so $r'(S')=r(S)$. To prove that $r'$ is modular on flats, take any two flats and choose a pair of finite generating sets. This pair must occur in $\cP$ at some index $i$, so by the construction in step $i$, the rank function is already modular on the corresponding flats, and this property is maintained under further extensions.
\end{proof}

We will use the fact that, for modular extensions, it suffices to study the components of the matroid.

\begin{lem}\label{lem:compkme}
For a nonnegative integer $k$, a matroid is $k$-modular extendable if and only if each of its components is $k$-modular extendable.
\end{lem}
\begin{proof}
It is clear that any restriction of a $k$-modular extendable matroid is $k$-modular extendable, so it suffices to show that the direct sum $M=(S,r)$ of any two $k$-modular extendable matroids $M_1=(S_1,r_1)$ and $M_2=(S_2,r_2)$ is $k$-modular extendable. We proceed by induction on $k$. The case $k=0$ is immediate; assume that $k\ge 1$. Let $A,B\subseteq S$ and define $A_i\coloneqq A\cap S_i$ and $B_i\coloneqq B\cap S_i$ for $i\in\{1,2\}$. Since each $M_i$ is $k$-modular extendable, there exists a $(k-1)$-modular extendable matroid $M'_i=(S'_i,r'_i)$ and a subset $Z_i\subseteq S'_i$ such that $S_i\subseteq S'_i$, $M'_i|S_i=M_i$, $r'_i(A_i\cup Z_i)=r_i(A_i)$, $r'_i(B_i\cup Z_i)=r_i(B_i)$, and $r'_i(Z_i)=r_i(A_i)+r_i(B_i)-r_i(A_i\cup B_i)$. Then $M'\coloneqq M'_1\oplus M'_2=(S',r')$ is $(k-1)$-modular extendable by the induction hypothesis, and $M'|S=M$. For the subset $Z\coloneqq Z_1\cup Z_2$, we have $r'(A\cup Z)=r(A)$, $r'(B\cup Z)=r(B)$, and $r'(Z)=r(A)+r(B)-r(A\cup B)$. This shows that $M$ is $k$-modular extendable.
\end{proof}

Now we are ready to prove the main result of the section, which establishes a connection between full modular extendability and skew-representability.

\begin{thm}\label{thm:equiv}
Given a matroid $M=(S,r)$, the following are equivalent.
\begin{enumerate}[label=(\roman*)]\itemsep0em
    \item $M$ is fully modular extendable. \label{it:i}
    \item $\si(M)\cong (M_{\mathbf{P}_1}\oplus\dots\oplus M_{\mathbf{P}_q})|T$ for some projective spaces $\mathbf{P}_1,\dots,\mathbf{P}_q$ and a finite subset $T$ of the ground set of $M_{\mathbf{P}_1}\oplus\dots\oplus M_{\mathbf{P}_q}$. \label{it:ii}
    \item Each connected component of $M$ is either of rank $3$ or skew-representable. \label{it:iii}
\end{enumerate}
\end{thm}
\begin{proof}
We prove the implications $\ref{it:i}\Rightarrow\ref{it:ii}\Rightarrow\ref{it:iii}\Rightarrow\ref{it:i}$.
\smallskip

\noindent $\ref{it:i}\Rightarrow\ref{it:ii}$: The matroid $\si(M)$ is also fully modular extendable. Hence, by \cref{lem:closure}, there exists a finitary modular matroid $M'=(S',r')$ such that $\si(M)$ is a restriction of $M'$ to a finite subset. By taking the simplification, we can assume that $M'$ is also simple. Then, \cref{prop:proj-matroid}\ref{it:matroid-to-proj} gives a generalized projective space $\mathbf{P}_{M'}$ whose subspaces are precisely the flats of $M'$. By \cref{prop:projsum}, there exists  a decomposition $\mathbf{P}_{M'} = \bigoplus_{j \in J} \mathbf{P}_j$, where each $\mathbf{P}_j$ is a projective space. This implies that $M'=M_{\mathbf{P}_{M'}} = \bigoplus_{j \in J} M_{\mathbf{P}_j}$.
Since $\si(M)$ is the restriction of $M'$ to a finite set that intersects only finitely many of these projective spaces, the implication follows.
\smallskip

\noindent $\ref{it:ii}\Rightarrow\ref{it:iii}$: Clearly, if each connected component of $\si(M)$ is either of rank $3$ or skew-representable, then the same holds for $M$ as well. Therefore, the implication directly follows from the Veblen--Young theorem, Proposition~\ref{prop:vy}.
\smallskip

\noindent $\ref{it:iii}\Rightarrow\ref{it:i}$: By \cref{lem:compkme}, it suffices to show that each component of $M$ is fully modular extendable. For rank-$3$ components, this holds by~\cite[Proposition 3.18]{bamiloshin2021common}. Consider now a skew-representable component $N$ with a representation by vectors in $\bF^n$. Then, $\si(N)$ is a restriction of $M_{\bF\mathbf{P}^{n-1}}$. By adding the necessary loops and parallel elements to $M_{\bF\mathbf{P}^{n-1}}$, we obtain a modular finitary matroid having $N$ as a restriction. Hence $N$ is fully modular extendable. 
\end{proof}

As a corollary, we get that for connected matroids of rank at least $4$, skew-representability is equivalent to full modular extendability.

\begin{cor}\label{cor:rank4}
Let $M$ be a connected matroid of rank at least $4$. Then, $M$ is skew-representable if and only if it is fully modular extendable. 
\end{cor}

\begin{rem}
Ingleton~\cite{ingleton1971representation} observed that every simple rank-3 matroid can be embedded in some projective plane, and Welsh~\cite{welsh1971combinatorial} asked whether this plane can always be taken to be finite; see also~\cite[Problem~15.9.2]{oxley2011matroid}. In our setting, this question translates to asking whether the extensions in the proof of \cref{lem:closure} can be chosen so that a finite modular extension is obtained after finitely many steps.
\end{rem}

\section{Properties of tensor products}
\label{sec:prep}

The goal of this section is to establish several technical properties of tensor products that will be used later. We first study their behavior under basic operations on matroids and polymatroid functions, including the factorization and associativity properties needed for iterated products. We then turn to tensor products arising from representations over skew fields.

\subsection{Tensor products}

This subsection aims to analyze how the tensor product behaves under basic operations on matroids and polymatroid functions. Using \cref{prop:crosses}, we show that the existence of a tensor product is preserved under taking minors: if two matroids admit a tensor product, then every minor of one has a tensor product with every minor of the other.

\begin{lem}\label{lem:minor}
Let $M_1$ and $M_2$ be matroids, and let $M\in M_1\otimes M_2$. For every minor $M_1'$ of $M_1$ and every minor $M_2'$ of $M_2$, there exists a minor $M'$ of $M$ such that $M'\in M_1'\otimes M_2'$. Consequently, if $T_k(M_1,M_2)\neq\emptyset$, then $T_k(M_1',M_2')\neq\emptyset$.
\end{lem}
\begin{proof}
Let $M_1=(S_1,r_1)$ and $M_2=(S_2,r_2)$ be matroids and $M=(S_1\times S_2,r)$ be a tensor product of $M_1$ and $M_2$. It suffices to show that every minor $M_1'$ of $M_1$ has a tensor product with $M_2$; the statement then follows by repeating the proof with $M_2$ and $M_1'$. If $M_1'$ is a restriction of $M_1$, then the assertion is trivial. So it suffices to consider the case when $M'_1=M_1/A$ for some $A\subseteq S_1$.

Let $M'= M/(A\times S_2)$ and let $r'$ denote the rank function of $M'$. We claim that $M'$ is a tensor product of $M'_1$ and $M_2$. To see this, we need to verify that $r'(X\times Y)=r'_1(X)\cdot r_2(Y)$ holds for all $X\subseteq S_1\setminus A$ and $Y\subseteq S_2$, where $r'_1$ denotes the rank function of $M'_1$. By \cref{prop:crosses}, when applied to $X_1=A$, $X_2=A\cup X$, $Y_1=S_2$ and $Y_2=Y$, we get
\begin{align*}
    r'(X\times Y)
    &=
    r((X\times Y)\cup (A\times S_2))-r(A\times S_2)\\
    &=
    r((A\cup X)\times Y)-r(A\times Y)\\
    &=
    r_1(A\cup X)\cdot r_2(Y)-r_1(A)\cdot r_2(Y)\\
    &=
    r'_1(X)\cdot r_2(Y),
\end{align*}
which proves the first statement. The second then follows by induction on $k$, applying the first part at each tensor-product step.
\end{proof}

The next operation we consider is the direct sum of matroids.

\begin{lem}\label{lem:direct_sum}
Let $M_1^1=(S_1^1,r_1^1)$, $M_1^2=(S_1^2,r_1^2)$ and $M_2=(S_2,r_2)$ be matroids. Then, $M_1^1\oplus M_1^2$ has a tensor product with $M_2$ if and only if $M_1^i$ has a tensor product with $M_2$ for $i=1,2$.
\end{lem}
\begin{proof}
Assume first that $N=((S_1^1\cup S_1^2)\times S_2,r)$ is a tensor product of $M_1^1\oplus M_1^2$ and $M_2$. Then, for every $X\subseteq S_1^i$ and $Y\subseteq S_2$, we have $r(X\times Y)=r_1^i(X)\cdot r_2(Y)$; hence $N|(S_1^i\times S_2)$ is a tensor product of $M_1^i$ and $M_2$ for $i=1,2$.

For the other direction, assume that $N_i$ is a tensor product of $M_1^i$ and $M_2$ for $i=1,2$. Let $N=N_1\oplus N_2=((S_1^1\cup S_1^2)\times S_2,r)$ denote their direct sum. Then, for every $X\subseteq S_1^1\cup S_1^2$ and $Y\subseteq S_2$, we have $r(X\times Y)=r((X\cap S_1^1)\times Y)+r((X\cap S_1^2)\times Y)=r_1^1(X\cap S_1^1)\cdot r_2(Y)+r_1^2(X\cap S_1^2)\cdot r_2(Y)=(r_1^1(X\cap S_1^1)+r_1^2(X\cap S_1^2))\cdot r_2(Y)$; hence $N$ is a tensor product of $M_1^1\oplus M_1^2$ and $M_2$. 
\end{proof}

While the previous lemma focused on tensor products of disconnected matroids, the next one addresses the connected case and shows that the tensor product of two connected matroids is connected.

\begin{lem}\label{lem:connected}
The tensor product of two connected matroids of positive rank is connected.
\end{lem}
\begin{proof}
Let $M_1=(S_1,r_1)$ and $M_2=(S_2,r_2)$ be connected matroids and $M=(S_1\times S_2,r)$ be a tensor product of $M_1$ and $M_2$. We need to show that for any $x_1,y_1\in S_1$ and $x_2,y_2\in S_2$, $(x_1, x_2)$ and $(y_1, y_2)$ lie in the same connected component of the circuit hypergraph of $M$. Since $M_1$ is connected and $M|S^{x_2}_1$ is isomorphic to $M_1$, there exists a circuit $C_1$ in $M$ such that $(x_1, x_2), (y_1, x_2)\in C_1$. Similarly, since $M_2$ is connected and $M|{}^{y_1}S_2$ is isomorphic to $M_2$, there exists a circuit $C_2$ in $M$ such that $(y_1, x_2),(y_1, y_2)\in C_2$. This concludes the proof of the lemma.
\end{proof}

With respect to the existence of a tensor product, a simple reduction allows us to assume that all matroids are simple, that is, without loops or parallel elements.

\begin{lem} \label{lem:simple}
    Let $M_1 = (S_1, r_1)$ and $M_2 =(S_2, r_2)$ be matroids. For $i=1,2$, let $L_i$ denote the set of loops of $M_i$ and let $\cP_i = \{P_1^i, \dots, P_{q_i}^i\}$ denote the partition of $S_i \setminus L_i$ into parallel classes of $M_i$. For any matroid $N$ on ground set $\cP_1 \times \cP_2$, let $\widetilde{N}$ denote the matroid on $S_1\times S_2$ obtained from $N$ by replacing each element $(P_j^1, P_k^2)\in \cP_1 \times \cP_2$ by a parallel class on $P^1_j \times P^2_k$, and adding the elements of $(S_1 \times L_2) \cup (L_1 \times S_2)$ as loops. Then, \[M_1 \otimes M_2 = \{\widetilde{N} \mid N \in \si(M_1) \otimes \si(M_2)\}.\] 
    In particular, two matroids admit a tensor product if and only if their simplifications do.
\end{lem}
\begin{proof}
    Assume first that $P = (S_1 \times S_2, r)$ is a tensor product of $M_1$ and $M_2$. Then, for any pair $(e_1, e_2) \in S_1 \times S_2$ such that $e_1 \in L_1$ or $e_2\in L_2$, $r(\{(e_1,e_2)\}) = r_1(\{e_1\}) \cdot r_2(\{e_2\}) = 0$, thus $(e_1, e_2)$ is a loop of $P$. For $j \in [q_1]$ and $k \in [q_2]$, $F^1_j\coloneqq P^1_j\cup L_1$ and $F^2_k \coloneqq P^2_k\cup L_2$ are rank-1 flats of $M_1$ and $M_2$, respectively. Hence, by \cref{prop:crosses}, $F^1_j \times F^2_k$ is a flat of $P$, and $r(F_j^1 \times F_k^2) = r_1(F^1_j)r_2(F^2_k) = 1$. Since every element of this flat outside $P^1_j\times P_k^2$ is a loop of $P$, it follows that $P^1_j\times P^2_k$ is a parallel class of $P$. These together with the fact that $P \in M_1 \otimes M_2$ imply that $N \coloneqq \si(P) \in \si(M_1) \otimes \si(M_2)$ and $P = \widetilde{N}$.

    To see the other direction, let $N \in \si(M_1) \otimes \si(M_2)$, $X_1 \subseteq S_1$, and $X_2 \subseteq S_2$. For $i \in \{1,2\}$, define $X'_i \coloneqq \{P^i_j \in \cP_i \mid P^i_j \cap X_i \ne \emptyset \}$. Then,  $r_{\widetilde{N}}(X_1 \times X_2) = r_N(X'_1 \times X'_2) = r_{\si(M_1)}(X'_1) \cdot r_{\si(M_2)}(X'_2) = r_1(X_1) \cdot r_2(X_2)$, thus $\widetilde{N}$ is a tensor product of $M_1$ and $M_2$.
\end{proof}

\cref{lem:polymatroid-direct-sum} implies that if $\varphi$ is a tensor product of $\varphi_1$ and $\varphi_2$, and $\varphi_1$ is modular on some set, then the corresponding part of $\varphi$ can be expressed as a direct sum.

\begin{lem}\label{lem:direct}
    Let $\varphi\colon2^{S_1\times S_2}\to\bR$ be a tensor product of polymatroid functions $\varphi_1\colon 2^{S_1}\to\bR$ and $\varphi_2\colon 2^{S_2}\to\bR$. If $\varphi_1(X)=\sum_{x\in X}\varphi_1(x)$ for some $X\subseteq S_1$, then $\varphi(W)=\sum_{x\in X}\varphi(W\cap \prd{x}{S_2})$ for all $W\subseteq X\times S_2$. Symmetrically, if $\varphi_2(Y)=\sum_{y\in Y}\varphi_2(y)$ for some $Y\subseteq S_2$, then $\varphi(W)=\sum_{y\in Y}\varphi(W\cap S_1^y)$ for all $W\subseteq S_1\times Y$.
\end{lem}
\begin{proof}
    We have $\varphi(X\times S_2)=\varphi_1(X)\cdot \varphi_2(S_2)=\left(\sum_{x \in X} \varphi_1(x)\right)\cdot \varphi_2(S_2) = \sum_{x \in X} \varphi({}^xS_2)$. The first statement follows directly from \cref{lem:polymatroid-direct-sum}. The second statement follows by symmetry.
\end{proof}

In our proofs, we will frequently use the following statement, which roughly states that if we extend the ground set in one factor of the tensor product without increasing its rank in the given matroid, then the rank in the tensor product remains unchanged.

\begin{lem}\label{lem:gen}
    Let $\varphi\colon2^{S_1\times S_2}\to\bR$ be a tensor product of polymatroid functions $\varphi_1\colon 2^{S_1}\to\bR$ and $\varphi_2\colon 2^{S_2}\to\bR$. 
    Then, for all subsets $X, X' \subseteq S_1$, $Y, Y'\subseteq S_2$, and $Z \subseteq S_1\times S_2$,
    \begin{align*}
    \varphi(((X\cup X')\times Y)\cup Z) & \le \varphi((X\times Y)\cup Z) + (\varphi_1(X\cup X')-\varphi_1(X)) \cdot \varphi_2(Y), \\
    \varphi((X \times (Y\cup Y'))\cup Z) & \le \varphi((X\times Y)\cup Z) + \varphi_1(X)\cdot (\varphi_2(Y\cup Y')-\varphi_2(Y)).
    \end{align*}
     In particular, if $\varphi_1(X\cup X') = \varphi_1(X)$ then 
    \[\varphi(((X\cup X')\times Y)\cup Z) = \varphi((X\times Y)\cup Z),\]
    and if $\varphi_2(Y\cup Y') = \varphi_2(Y)$ then \[\varphi((X\times (Y\cup Y'))\cup Z) = \varphi((X\times Y)\cup Z).\]
\end{lem}
\begin{proof} 
We prove the first inequality; the second follows analogously. As $\varphi$ is the tensor product of $\varphi_1$ and $\varphi_2$, we have $\varphi((X\cup X')\times Y)-\varphi(X\times Y)=(\varphi_1(X\cup X')-\varphi_1(X))\cdot \varphi_2(Y)$. Hence, since $\varphi$ is increasing and submodular, we get
\begin{align*}
    \varphi(((X\cup X')\times Y)\cup Z)
    &\leq
    \varphi((X\times Y)\cup Z) + \varphi((X\cup X')\times Y)-\varphi((X\times Y)\cup((X'\times Y)\cap Z))\\
    &\leq 
    \varphi((X\times Y)\cup Z) + \varphi((X\cup X')\times Y)-\varphi(X\times Y)\\
    &=
    \varphi((X\times Y)\cup Z)+(\varphi_1(X\cup X')-\varphi_1(X))\cdot\varphi_2(Y),
\end{align*}
proving the statement.
\end{proof}

Next, we study conditions under which the tensor product operation is associative. For this, we derive the following technical lemma whose polymatroid results will also be applied later.

\begin{lem} \label{lem:box} 
    Let $k \ge 1$ be an integer and let $M_{k+1}=(S_{k+1},r_{k+1})$ be a connected matroid of positive rank.
    \begin{enumerate}[label=(\alph*)]\itemsep0em
        \item \label{it:poly-box}
        Let $\varphi_i\colon 2^{S_i}\to \bR$ for $i \in [k]$ and $\varphi\colon 2^{\prod_{i=1}^{k+1} S_i} \to \bR$ be polymatroid functions such that $\varphi\bigl(\prod_{i=1}^{k+1} X_i\bigr) = \bigl(\prod_{i=1}^k \varphi_i(X_i)\bigr) r_{k+1}(X_{k+1})$ holds for every $X_i \subseteq S_i$, $i \in [k+1]$. Then, $\varphi \in \psi \otimes r_{k+1}$ for some polymatroid function $\psi\colon 2^{\prod_{i=1}^k S_i} \to \bR$ such that $\psi(\prod_{i=1}^k X_i) = \prod_{i=1}^k \varphi_i(X_i)$ for every $X_i \subseteq S_i$, $i \in [k]$.
        \item \label{it:matroid-box} Let $M_i = (S_i, r_i)$ for $i \in [k]$ and $M=(\prod_{i=1}^{k+1} S_i, r)$ be matroids such that $r(\prod_{i=1}^{k+1} X_i) = \prod_{i=1}^{k+1} r_i(X_i)$ holds for every $X_i \subseteq S_i$, $i \in [k+1]$. Then, $M\in N\otimes M_{k+1}$ for some matroid $N=(\prod_{i=1}^{k}S_i, r_N)$ such that $r_N(\prod_{i=1}^k X_i) = \prod_{i=1}^k r_i(X_i)$ for every $X_i \subseteq S_i$, $i \in [k]$.
    \end{enumerate}
\end{lem}
\begin{proof}
    First, we prove \ref{it:poly-box}.
    For each $e\in S_{k+1}$, define $\psi_e\colon 2^{\prod_{i=1}^k S_i}\to \bR$ by $\psi_e(W)\coloneqq \varphi(W\times \{e\})$ for $W\subseteq \prod_{i=1}^k S_i$.
    $\psi_e$ is a polymatroid function, as it is a restriction of $\varphi$.
    For every $X_i\subseteq S_i$, $i \in [k]$, we have \[\psi_e\bigl(\prod_{i=1}^k X_i\bigr)=\varphi\biggl(\bigl(\prod_{i=1}^k X_i\bigr)\times \{e\}\biggr) = \bigl(\prod_{i=1}^k \varphi_i(X_i)\bigr)r_{k+1}(e) = \prod_{i=1}^k \varphi_i(X_i).\]
    We next show, using the following claim, that $\psi_e$ does not depend on the choice of $e\in S_{k+1}$.

    \begin{cla} \label{cl:prod-rank}
        Let $W\subseteq \prod_{i=1}^k S_i$, $Z\subseteq S_{k+1}$, and let $I$ be a maximal independent subset of $Z$ in $M_{k+1}$.  Then, $\varphi(W\times Z) = \sum_{f \in I} \psi_f(W)$.
    \end{cla}
    \begin{claimproof}
        For each $x=(x_1,\dots,x_k)\in W$, we have \[\varphi(\{x\} \times Z)=\bigl(\prod_{i=1}^k \varphi_i(x_i)\bigr)r_{k+1}(Z)=\bigl(\prod_{i=1}^k \varphi_i(x_i)\bigr)r_{k+1}(I)=\varphi(\{x\} \times I),\]
        that is, the addition of $\{x\} \times (Z\setminus I)$ to the set $\{x\}\times I$ does not increase the value of $\varphi$.
        Hence, by the submodularity and monotonicity of $\varphi$, 
        \[ \varphi(W\times Z)=\varphi(W\times I).\]
        As $I$ is an independent set, 
         \[\varphi\biggl(\bigl(\prod_{i=1}^k S_i\bigr)\times I\biggr) = \bigl(\prod_{i=1}^k \varphi_i(S_i)\bigr) |I|=\sum_{f \in I} \bigl(\prod_{i=1}^k \varphi_i(S_i)\bigr) = \sum_{f \in I} \varphi\biggl(\bigl(\prod_{i=1}^k S_i\bigr) \times \{f\}\biggr).\]
         Applying \cref{lem:polymatroid-direct-sum} to the sets $(\prod_{i=1}^k S_i)\times\{f\}$, $f\in I$, gives \[\varphi(W\times I) = \sum_{f \in I} \varphi(W\times \{f\})=\sum_{f \in I} \psi_f(W).\]
        Combining this and $\varphi(W\times Z)=\varphi(W\times I)$, the claim follows.
    \end{claimproof}

    Let $e_1, e_2 \in S_{k+1}$ be distinct. Since $M_{k+1}$ is connected, it has a circuit $C$ containing both $e_1$ and $e_2$.
    For $t \in [2]$, $C-e_t$ is a maximal independent set in $C$; hence, for each $W\subseteq \prod_{i=1}^k S_i$, \cref{cl:prod-rank} shows that 
    \[\varphi(W\times C) = \sum_{f \in C-e_t} \psi_f(W) = \sum_{f \in C} \psi_f(W)-\psi_{e_t}(W).\]
    It follows that
    \[\psi_{e_1}(W)=\sum_{f \in C} \psi_f(W) - \varphi(W\times C) = \psi_{e_2}(W).\]
    This shows that $\psi_e$ does not depend on the choice of $e\in S_{k+1}$; let $\psi$ denote this common function.
    Recall that $\psi(\prod_{i=1}^k X_i) = \prod_{i=1}^k \varphi_i(X_i)$ for every $X_i\subseteq S_i$, $i \in [k]$.
    Let $W\subseteq \prod_{i=1}^k S_i$, $Z\subseteq S_{k+1}$, and let $I$ be a maximal independent subset of $Z$ in $M_{k+1}$. 
    \cref{cl:prod-rank} shows that 
    \[\varphi(W\times Z)=\sum_{f \in I} \psi_f(W)=|I|\cdot \psi(W)=\psi(W) r_{k+1}(Z),\]
    proving $\varphi \in \psi \otimes r_{k+1}$.
    This finishes the proof of \ref{it:poly-box}.
    Finally, \ref{it:matroid-box} follows by observing that if $\varphi_1, \dots, \varphi_k$, and $\varphi$ are matroid rank functions, then  $\psi$ is a matroid rank function as well.
\end{proof}

An iterated use of \cref{lem:box} shows that the tensor product operation is associative on connected matroids.

\begin{lem} \label{lem:assoc}
    Let $k \ge 2$ be an integer, let $M_1=(S_1,r_1), \dots, M_k = (S_k, r_k)$ be connected matroids of positive rank, and let $M=(\prod_{i=1}^k S_i, r)$ be a matroid. The following are equivalent.
    \begin{enumerate}[label=(\roman*)]\itemsep0em
        \item \label{it:every-par} $M$ belongs to the set of tensor products obtained from every
        parenthesization of $M_1\otimes\cdots\otimes M_k$;
        \item \label{it:some-par} $M$ belongs to the set of tensor products obtained from some
        parenthesization of $M_1\otimes\cdots\otimes M_k$;
        \item \label{it:box} $r(\prod_{i=1}^k X_i) = \prod_{i=1}^k r_i(X_i)$ holds for every $X_i\subseteq S_i$, $i \in [k]$.
    \end{enumerate}
\end{lem}
\begin{proof}
    The implications $\ref{it:every-par} \Rightarrow \ref{it:some-par}$ and $\ref{it:some-par} \Rightarrow \ref{it:box}$ are immediate.
    Assume that $\ref{it:box}$ holds; we show $\ref{it:every-par}$ by induction on $k$.    
    The statement clearly holds for $k=2$.
    Every parenthesization has an outermost tensor product that splits the factors after some $j\in[k-1]$.
    Thus, by the induction hypothesis, it suffices to show that for each $j \in [k-1]$, there exist matroids $N_j =(\prod_{i=1}^j S_i, p_j)$ and $N'_j = (\prod_{i=j+1}^k S_i, p'_j)$ such that $M\in N_j \otimes N'_j$ and $p_j(\prod_{i=1}^j X_i) = \prod_{i=1}^j r_i(X_i)$ and $p'_j(\prod_{i=j+1}^k X_i) = \prod_{i=j+1}^k r_i(X_i)$ for every $X_i \subseteq S_i$, $i \in [k]$.
    First, we construct the matroids $N_j$.
    Let $N_k=M$; for $j=k-1,\dots, 2$, let $N_j=(\prod_{i=1}^j S_i, p_j)$ be a matroid provided by \cref{lem:box}\ref{it:matroid-box} such that $N_{j+1}\in N_j \otimes M_{j+1}$ and $p_j(\prod_{i=1}^j X_i) = \prod_{i=1}^j r_i(X_i)$ for every $X_i \subseteq S_i$, $i \in [j]$; and finally, let $N_1\coloneqq M_1$.
    Note that $N_{j+1}\in N_j\otimes M_{j+1}$ holds for each $j \in [k-1]$, including $j=1$.
    \cref{lem:connected} hence shows that $N_1,\dots, N_k$ are connected, as $M_1,\dots, M_k$ are connected.
    As $N_k = M$ and 
    $N_{j+1}\in N_j \otimes M_{j+1}$ for $j=k-1,\dots, 1$, it follows by repeatedly applying the definition of a tensor product that $r(W\times \prod_{i=j+1}^k X_i) = p_j(W)\prod_{i=j+1}^k r_i(X_i) $ for every $W\subseteq \prod_{i=1}^j S_i$ and $X_i\subseteq S_i$, $i \in \{j+1,\dots, k\}$.
    Interchange the factors $\prod_{i=1}^j S_i$ and $\prod_{i=j+1}^k S_i$ in the ground set of $M$ and apply  
    \cref{lem:box}\ref{it:matroid-box} to $M_{j+1},\dots, M_k,N_j$ and $M$. 
    Interchanging the two factors back gives a matroid  $N'_j=(\prod_{i=j+1}^k S_i, p'_j)$ such that $M\in N_j \otimes N'_j$ and $p'_j(\prod_{i=j+1}^k X_i) = \prod_{i=j+1}^k r_i(X_i)$ for every $X_i \subseteq S_i$, $i \in \{j+1,\dots, k\}$.
    This shows the existence of the desired matroids $N_j$ and $N'_j$ for $j \in [k-1]$, finishing the proof.
\end{proof}

In view of \cref{lem:assoc}, for a connected matroid $M$ of positive rank and $k\in\bZ_+$, we denote by $M^{\otimes k}$ the set of tensor products obtained from $k$ copies of $M$, with an arbitrary parenthesization. 
Observe that for connected matroids $M$ and $N$, \cref{lem:assoc} shows that $M$ is $k$-tensor-compatible with $N$ if and only if $M\otimes N^{\otimes k}$ is nonempty.
More generally, by \cref{lem:assoc} and \cref{lem:direct_sum}, if $M$ has connected components $M_1,\dots, M_t$ and $N$ is still connected, then $M$ is $k$-tensor-compatible with $N$ if and only if $M_i\otimes N^{\otimes k}$ is nonempty for each $i \in [t]$.

\begin{rem}
The connectedness assumption in \cref{lem:assoc} cannot be dismissed. Suppose that $M_1$ and $M_2$ admit two distinct tensor products $N_1,N_2\in M_1\otimes M_2$. Since $M_2\otimes U_{2,2}$ is the direct sum of two copies of $M_2$, \cref{lem:direct_sum} gives $N_1\oplus N_2\in M_1\otimes (M_2\otimes U_{2,2})$. On the other hand, every member of $(M_1\otimes M_2)\otimes U_{2,2}$ is the direct sum of two copies of a single member of $M_1\otimes M_2$, so $N_1\oplus N_2$ does not belong to this set. Thus the set of iterated tensor products may depend on the parenthesization when the factors are disconnected.
\end{rem}

\subsection{Skew-representability} \label{subsec:skew}

If two matroids are representable over the same field $\bF$, then they admit a tensor product representable over $\bF$, which can be constructed via the Kronecker product of their representing matrices. Over skew fields, the same construction is more delicate, since the entries of the two representing matrices need not commute. Nevertheless, a similar result holds if one of the matroids is representable over the center of $\bF$. This follows essentially from the fact that $r(A\otimes B)=r(A)\cdot r(B)$ whenever the entries of $B$ lie in the center of the skew field. We first recall the following standard property of Kronecker products; the usual proof, given for example in~\cite{graham1981kronecker}, carries over without modification.

\begin{lem}[Mixed-product property] \label{lem:mpp}
    Let $\bF$ be a skew field, and let $A \in \bF^{m \times n}$, $B \in \bF^{p \times q}$, $C \in \bF^{n \times k}$, and $D \in \bF^{q \times r}$ be matrices such that each entry of $B$ commutes with each entry of $C$. Then,
    \[(A\otimes B)(C \otimes D) = (AC) \otimes (BD).\]
\end{lem}

The proof of the following lemma is essentially the same as in the commutative case. However, since we could not find a suitable reference and the case over skew fields can sometimes behave in surprising ways, we include a proof here as well.

\begin{lem} \label{lem:Kronecker}
    Let $\bF$ be a skew field and $A \in \bF^{m \times n}$ and $B \in Z(\bF)^{p \times q}$ be matrices. Then, $r(A\otimes B) = r(A)\cdot r(B)$.
\end{lem}
\begin{proof}
    Assume first that both $A$ and $B$ are square and invertible. Let $I_\ell$ denote the $\ell \times \ell$ identity matrix. Then, \cref{lem:mpp} implies that 
    \[I_{mp} = I_m \otimes I_p = (A A^{-1}) \otimes (B B^{-1}) = (A\otimes B)(A^{-1} \otimes B^{-1}),\]
    showing that $A\otimes B$ is invertible as well. 
    
    We turn to the proof of the general case. Let $A'$ and $B'$ be invertible submatrices of $A$ and $B$ having size $r(A)\times r(A)$ and $r(B) \times r(B)$, respectively. Then, $A'\otimes B'$ is an invertible submatrix of $A\otimes B$ having size $(r(A)\cdot r(B)) \times (r(A)\cdot r(B))$, implying $r(A\otimes B) \ge r(A)\cdot r(B)$. To see the other direction, we may assume that the first $r(A)$ columns of $A$ form a basis of the right column space of $A$ and the first $r(B)$ columns of $B$ form a basis of the column space of $B$. Consider the set of columns of $A\otimes B$ with index set $H\coloneqq \{(i-1) \cdot q + j \mid i \in [r(A)], j \in [r(B)] \}$. As the entries of $A$ and $B$ commute and the first $r(B)$ columns of $B$ form a basis of the column space of $B$, for each $i \in [n]$, the columns $\{(i-1)\cdot q + j \mid j \in [r(B)]\}$  span the columns $\{(i-1)\cdot q + j \mid j \in [q]\}$ in the column space of $A\otimes B$; thus the columns corresponding to $H$ span columns $\{(i-1)\cdot q + j \mid i \in [r(A)], j \in [q]\}$. Similarly, for each $j \in [q]$ columns $\{(i-1)\cdot q + j \mid i \in [r(A)]\}$ span columns $\{(i-1)\cdot q + j \mid i \in [n]\}$ in the column space of $A\otimes B$; thus the columns corresponding to $H$ span all the columns. This shows that $r(A\otimes B) \le |H| = r(A)\cdot r(B)$, finishing the proof.
\end{proof}

The preceding lemma allows us to carry out the Kronecker-product construction when one of the two representations is over the center of the skew field.

\begin{lem}\label{lem:representable}
    Let $\bF$ be a skew field, $\varphi_1\colon 2^{S_1} \to \bZ$ an $\bF$-representable polymatroid function, and $\varphi_2\colon 2^{S_2}\to \bZ$ a $Z(\bF)$-representable polymatroid function. Then, $\varphi_1$ and $\varphi_2$ admit an $\bF$-representable tensor product $\varphi$. Moreover, if $\varphi_1$ and $\varphi_2$ are matroid rank functions, then $\varphi$ is the rank function of an $\bF$-representable matroid.  
\end{lem}
\begin{proof}
    Let $A\in \bF^{m \times n}$ be a matrix and $\cP = \{P_{s_1} \mid s_1 \in S_1\}$ a partition of $[n]$ representing $\varphi_1$, and $B\in Z(\bF)^{p\times q}$ a matrix and $\cQ = \{Q_{s_2} \mid s_2 \in S_2\}$ a partition of $[q]$ representing $\varphi_2$. Define $R_{(s_1, s_2)} \coloneqq \{(i-1)\cdot q + j \mid i \in P_{s_1}, j \in Q_{s_2}\}$ for $(s_1, s_2) \in S_1 \times S_2$ and observe that $\cR = \{R_{(s_1, s_2)} \mid (s_1, s_2) \in S_1 \times S_2\}$ is a partition of $[nq]$. Consider the polymatroid function $\varphi\colon 2^{S_1 \times S_2} \to \bZ$ represented by the matrix $A\otimes B\in \bF^{mp \times nq}$ and the partition $\cR$. Then, for $X_1 \subseteq S_1$ and $X_2 \subseteq S_2$, $\varphi(X_1\times X_2)$ equals the rank of the submatrix of $A\otimes B$ formed by the columns with indices from $\bigcup_{(s_1, s_2)\in X_1 \times X_2} R_{(s_1, s_2)}$. This submatrix is the same as $A'\otimes B'$ where $A'$ is the submatrix of $A$ formed by its columns with indices $\bigcup_{s_1 \in X_1} P_{s_1}$ and $B'$ is the submatrix of $B$ formed by its columns with indices $\bigcup_{s_2 \in X_2} Q_{s_2}$. Applying \cref{lem:Kronecker}, we get $\varphi(X_1\times X_2) = r(A'\otimes B') = r(A')\cdot r(B') = \varphi_1(X_1) \cdot \varphi_2(X_2)$, finishing the proof.
\end{proof}

\begin{rem} \label{rem:assoc} 
    We observe that the previous proof implies a stronger consequence for the product of multiple polymatroid functions using the associativity of Kronecker products. For a positive integer $k$ and polymatroid functions $\varphi_i\colon 2^{S_i}\to \bR$ for $i \in [k]$, we define the {\it associative tensor product} of the $k$-tuple $(\varphi_1,\dots, \varphi_k)$ recursively. For $k=1$, we define it to be $\varphi_1$, while for $k\ge 2$, we define it to be a polymatroid function $\varphi\colon 2^{S_1\times \dots \times S_k} \to \bR$ such that for each $i \in [k-1]$, $\varphi$ is a tensor product of an associative tensor product of $(\varphi_1,\dots, \varphi_i)$ and an associative tensor product of $(\varphi_{i+1},\dots, \varphi_k)$. In particular, if $k=2$, we get back the definition of $\varphi_1\otimes \varphi_2$, while for $k=3$ we get a polymatroid function that is simultaneously a tensor product $(\varphi_1\otimes \varphi_2)\otimes \varphi_3$ and a tensor product $\varphi_1\otimes (\varphi_2 \otimes \varphi_3)$. If $\varphi_1\colon 2^{S_1}\to \bZ$ is $\bF$-representable and $\varphi_i\colon 2^{S_i}\to \bZ$ is $Z(\bF)$-representable for $i \in \{2,\dots, k\}$, then it is not difficult to check that the proof of \cref{lem:representable} implies the existence of an associative tensor product of $(\varphi_1,\dots, \varphi_k)$ since taking the Kronecker product of matrices is an associative operation.  We can also analogously define the associative tensor product of a $k$-tuple of matroids. 
\end{rem}

Finally, we give a sufficient condition for two matroids or polymatroid functions to have a tensor product that is representable over some skew field of a given characteristic.

\begin{lem} \label{lem:skew-large} 
    Let $\bF$ be a skew field of characteristic $p$, $\varphi_1\colon 2^{S_1}\to \bZ$ a polymatroid function representable over $\bF$, and $\varphi_2\colon 2^{S_2}\to \bZ$ a polymatroid function representable over all infinite fields of characteristic $p$. 
    Then, $\varphi_1$ and $\varphi_2$ admit a tensor product $\varphi$ which is representable over a skew field of characteristic $p$. Moreover, if $\varphi_1$ and $\varphi_2$ are matroid rank functions, then $\varphi$ is a matroid rank function as well. 
\end{lem}
\begin{proof}
    Let $C$ denote the center of $\bF$. 
    Consider the skew field of rational functions $\bF(x)$ in a central indeterminate $x$ (see \cite{cohn1995skew} for definitions). Note that $\bF(x)$ has characteristic $p$ and its center is $C(x)$; see \cite[Proposition~2.1.5]{cohn1995skew} for details. The polymatroid function $\varphi_1$ is representable over $\bF(x)$ since $\bF(x)$ is an extension of $\bF$, and $\varphi_2$ is representable over $Z(\bF(x)) = C(x)$ since $C(x)$ is an infinite field of characteristic $p$. Therefore, \cref{lem:representable} implies that $\varphi_1$ and $\varphi_2$ have a tensor product that is representable over $\bF(x)$.   
\end{proof}

\begin{rem}\label{rem:reg}
    \cref{lem:representable} implies that if $M$ and $N$ are regular matroids, then for each field $\bF$ they admit an $\bF$-representable tensor product. However, these tensor products are not necessarily the same across fields: $M$ and $N$ do not necessarily admit a tensor product that is regular. 
    In particular, the cographic matroid $M^*(K_{3,3})$ of $K_{3,3}$ does not admit a regular tensor product with $U_{2,3}$.
    Indeed, we will see in \cref{lem:U232_unique} that $U_{2,3}^{\otimes 2} = \{M^*(K_{3,3})\}$; hence, by \cref{lem:assoc}, every tensor product of $M^*(K_{3,3})$ and $U_{2,3}$ belongs to $U_{2,3}^{\otimes 3}$. 
    By \cref{lem:unique-power}, this set contains a unique member representable over a skew field of characteristic $2$, and that matroid is not representable over any skew field of characteristic different from $2$.
\end{rem}

\section{Skew-representability through tensor products}
\label{sec:tensor}

This section is devoted to one of the main results of the paper: a characterization of skew-representability in terms of tensor products. In \cref{sec:tba}, we establish a connection between a matroid being $1$-tensor-compatible with the uniform matroid $U_{2,3}$ and being $1$-modular extendable. In \cref{sec:cip}, we show that $k$-tensor-compatibility for every $k \in \bZ_+$ implies full modular extendability. \cref{sec:skewfield} presents the main result, providing a tensor product characterization of skew-representable connected matroids, as well as matroids representable over skew fields of a prescribed characteristic. As a consequence, we obtain that both skew-representability and representability over a skew field of prescribed characteristic are co-recursively enumerable. Finally, in \cref{sec:polymatroids}, we prove a corresponding implication for polymatroid functions directly. The proof uses a polymatroid version of the one-step construction from Section~\ref{sec:tba}, together with factorization and associativity properties of tensor products.

\subsection{Tensor products with \texorpdfstring{$U_{2,3}$}{U23} and 1-modular extendability}
\label{sec:tba}

The next theorem establishes a connection between a matroid having a tensor product with $U_{2,3}$ and being $1$-modular extendable. A polymatroid analogue was obtained earlier by the fifth author~\cite{padro2025tensor}; the precise relation between the two constructions is explained in \cref{sec:polymatroids}.

\begin{thm}\label{thm:1modular}
Assume that $M=(S_1,r_1)$ and $U_{2,3}=(S_2,r_2)$ admit a tensor product $N=(S_1\times S_2,r)$, and let $A,B\subseteq S_1$. Then, $N$ has a minor $N'$ such that $N'$ is a one-step extension of $M$ with respect to $A$ and $B$.
\end{thm}
\begin{proof} 
Let $S_2=\{u,v,w\}$ denote the elements of $U_{2,3}$. We consider the following independent sets in $M$. Let $B_0$ be a maximum independent set of $B$, and let $A_0$ be an independent set of $A\setminus B$ such that $A_0\cup B_0$ is a maximum independent set of $A\cup B$. Finally, let $C$ be an independent set of $A\setminus A_0$ such that $A_0\cup C$ is a maximum independent set of $A$. Consider the minor $N'=(S_1^u\cup C^v,r')$ of $N$ obtained as $N'=N/(A^v_0\cup B_0^w)|(S_1^u\cup C^v)$. Note that, by Lemma~\ref{lem:direct} and by the choice of $A_0$ and $B_0$, we have $r(A_0^v\cup B_0^w)=r(A_0^v)+r(B_0^w)=|A_0|+|B_0|=r(A^x\cup B^x)$ for all $x\in\{u,v,w\}$.

\begin{cla}\label{cl:restriction}
$r'(X)=r(X)$ for all $X\subseteq S_1^u$.
\end{cla}
\begin{claimproof}
By definition, $r'(X)\leq r(X)$ for all $X\subseteq S_1^u$. We first show that it suffices to verify the statement for $X=S_1^u$. Indeed, suppose that $X\subseteq S_1^u$ and that $r'(X)<r(X)$, that is, $r(X\cup A_0^v\cup B_0^w)-r(A_0^v\cup B_0^w)<r(X)$. Then, since $r$ is submodular and increasing, we get
    \begin{align*}
        r'(S_1^u)
        &=
        r(S_1^u\cup A_0^v\cup B_0^w)-r(A_0^v\cup B_0^w)\\
        &\leq 
        r(X\cup A_0^v\cup B_0^w)+r(S_1^u)-r(X)-r(A_0^v\cup B_0^w)\\
        &=
        r(S_1^u)+r'(X)-r(X)\\
        &< 
        r(S_1^u).
    \end{align*}
This would imply $r'(S_1^u) < r(S_1^u)$ as well. Now assume that $X=S_1^u$. Then, by Lemmas~\ref{lem:direct} and~\ref{lem:gen}, we get
\begin{align*}
        r'(S_1^u)
        &=
        r(S_1^u\cup A_0^v\cup B_0^w)-r(A_0^v\cup B_0^w)\\
        &= 
        r(S_1^u\cup A_0^v\cup A_0^w\cup B_0^w)-r(A_0^v\cup B_0^w)\\
        &= 
        r(S_1^u\cup A_0^w\cup B_0^w)-r(A_0^v\cup B_0^w)\\
        &=
        r(S_1^u)+ r(A_0^w\cup B_0^w)-r(A_0^v\cup B_0^w)\\
        &=
        r(S_1^u),
    \end{align*}
concluding the proof of the claim.
\end{claimproof}

We next verify that $r'$ assigns the correct values to the relevant sets.

\begin{cla}\label{cl:modular}
$r'(A^u\cup C^v)=r(A^u)$, $r'(B^u\cup C^v)=r(B^u)$, and $r'(C^v)=r(A^u)+r(B^u)-r(A^u\cup B^u)$.
\end{cla}
\begin{claimproof}
By Lemmas~\ref{lem:direct} and~\ref{lem:gen}, together with the facts that $A_0\cup C$ is a maximum independent set in $A$ and $A_0\cup B_0$ is a maximum independent set in $A\cup B$, we get
\begin{align*}
 r'(A^u\cup C^v)
 &=
 r(A^u\cup C^v\cup A_0^v\cup B_0^w )-r(A_0^v\cup B_0^w)\\
 &=
 r(A^u\cup A^v \cup A^w\cup B_0^w)-r(A_0^v\cup B_0^w)\\
 &=
 r(A^u\cup A_0^w\cup B_0^w)-r(A_0^v\cup B_0^w)\\
 &=
 r(A^u)+r(A_0^w \cup B_0^w)-r(A_0^v\cup B_0^w)\\
 &=
 r(A^u).
\end{align*}
Although the roles of $A$ and $B$ are not symmetric in the choice of $A_0$, $B_0$ and $C$, a similar computation for $B$ follows, namely 
 \begin{align*}
 r'(B^u\cup C^v)
 &=
 r(B^u\cup C^v\cup A_0^v\cup B_0^w )-r(A_0^v\cup B_0^w)\\
 &=
 r(B^u\cup A^v\cup B_0^v\cup B_0^w)-r(A_0^v\cup B_0^w)\\
 &=
 r(B^u\cup A_0^v \cup B_0^v)-r(A_0^v\cup B_0^w)\\
 &=
 r(B^u)+r(A_0^v \cup B_0^v)-r(A_0^v\cup B_0^w)\\
 &=
 r(B^u).
\end{align*}
Finally, using that $B_0$ is a maximum independent set of $B$ as well, we get
\begin{align*}
    r'(C^v)
    &=
    r(C^v\cup A_0^v\cup B_0^w)-r(A_0^v\cup B_0^w)\\
    &=
    r(A^v\cup B_0^w)-r(A_0^v\cup B_0^w)\\
    &=
    r(A^v)+r(B_0^w)-r(A_0^v\cup B_0^w)\\
    &=
    r(A^u)+r(B^u)-r(A^u\cup B^u),
\end{align*}
and the claim follows.
\end{claimproof}

By Claims~\ref{cl:restriction} and~\ref{cl:modular}, $N'$ is a one-step extension of $M$ with respect to $A$ and $B$, concluding the proof of the theorem.
\end{proof}

As a direct consequence of Theorem~\ref{thm:1modular}, we get that having a tensor product with $U_{2,3}$ implies that the matroid is $1$-modular extendable.

\begin{cor}\label{cor:1me}
Let $M$ be a matroid that has a tensor product with $U_{2,3}$. Then, $M$ is $1$-modular extendable.
\end{cor}

\subsection{\texorpdfstring{$k$}{k}-tensor-compatibility with \texorpdfstring{$U_{2,3}$}{U23} and \texorpdfstring{$k$}{k}-modular extendability}
\label{sec:cip}

\cref{cor:1me} establishes a connection between a matroid having a tensor product with $U_{2,3}$ and being $1$-modular extendable. Using this, we first prove an analogous result for $k$-modular extendability.

\begin{thm}\label{thm:me} 
    Let $M=(S,r)$ be a matroid which is $k$-tensor-compatible with $U_{2,3}$ for some $k \in \bZ_+$. Then, $M$ is $k$-modular extendable.
\end{thm}
\begin{proof}
    We prove the theorem by induction on $k$. For $k=1$, this follows from \cref{thm:1modular}. Assume that $k\geq 2$ and that the statement holds up to $k-1$. As $T_k(M, U_{2,3}) \ne \emptyset$, there is a matroid $N\in M\otimes U_{2,3}$ such that $T_{k-1}(N, U_{2,3}) \ne \emptyset$. Let $A, B \subseteq S$. By \cref{thm:1modular}, $N$ has a minor $N'$ that is a one-step extension of $M$ with respect to $A$ and $B$. By \cref{lem:minor}, $T_{k-1}(N',U_{2,3})\ne \emptyset$, thus $N'$ is $(k-1)$-modular extendable by the induction hypothesis. This completes the proof of the theorem.
\end{proof}

As an immediate corollary, we get the following.

\begin{cor}\label{cor:fme} 
Let $M=(S,r)$ be a matroid which is $k$-tensor-compatible with $U_{2,3}$ for each $k\in \bZ_+$. Then, $M$ is fully modular extendable.
\end{cor}
\begin{proof}
    It suffices to show that $M$ is $k$-modular extendable for all $k\in\bZ_+$, which follows from \cref{thm:me}.
\end{proof}

We will later see in \cref{cor:non_Desargues_not_2_compatible} that there is a fully modular extendable matroid which is not 2-tensor-compatible with $U_{2,3}$. This shows that the converses of \cref{thm:me} for $k \ge 2$ and \cref{cor:fme} do not hold.

\subsection{Characterization of skew-representability}
\label{sec:skewfield}

Using Corollaries~\ref{cor:rank4} and \ref{cor:fme}, we now give a characterization of representability over skew fields of certain characteristics. We are ready to state the first main result of the section.

\begin{thm} \label{thm:char}
    Let $N$ be a connected skew-representable matroid of rank at least two. Let $C$ be the skew characteristic set of $N$ and assume that for each $p \in C$, $N$ is representable over all infinite fields of characteristic $p$. Then, for any matroid $M$, $M$ is the direct sum of matroids each of which is representable over a skew field of characteristic in $C$ if and only if $M$ is $k$-tensor-compatible with $N$ for each $k\in \bZ_+$.
\end{thm}
\begin{proof}
    Assume first that $M=M_1\oplus \dots \oplus M_q$, where each $M_i$ is representable over a skew field of characteristic $p_i \in C$. It follows from \cref{lem:skew-large} by induction on $k$ that for each $k \in \bZ_+$ there exists a matroid $P_{k,i}\in T_k(M_i, N)$ which is representable over a skew field of characteristic $p_i$. It follows from \cref{lem:direct_sum} by induction on $k$ that $P_{k,1} \oplus \dots \oplus P_{k,q} \in T_k(M, N)$ for each $k \in \bZ_+$. This shows that $M$ is $k$-tensor-compatible with $N$ for each $k \in \bZ_+$.

    To see the other direction, let $M_0$ be a component of $M$. We need to show that $M_0$ is representable over a skew field of characteristic in $C$. We may assume that $M_0$ has rank at least two, since every matroid of rank at most one is regular.
    For each $k \in \bZ_+$, there exists a matroid $P_k \in T_k(M_0, N)$. By the definition of $T_k(M_0, N)$, for each $k \in \bZ_+$, there is a matroid $P'_k$ such that $P'_k \in M_0 \otimes N$ and $P_k \in T_{k-1}(P'_k, N)$. As there are only finitely many matroids in $M_0 \otimes N$, there exists a matroid $P'\in M_0 \otimes N$ such that $P'=P'_k$ for infinitely many values of $k \in \bZ_+$. In particular, $T_{k-1}(P',N) \ne \emptyset$ holds for infinitely many values of $k$. Since $T_\ell(P',N) \ne \emptyset$ implies $T_{\ell-1}(P',N) \ne \emptyset$ for each $\ell \ge 2$, it follows that $P'$ is $k$-tensor-compatible with $N$ for each $k \in \bZ_+$. By Lemma~\ref{lem:u23minor}, each connected matroid of rank at least two contains $U_{2,3}$ as a minor, thus $P'$ is also $k$-tensor-compatible with $U_{2,3}$ for each $k\in \bZ_+$ by \cref{lem:minor}. Then, \cref{cor:fme} implies that $P'$ is fully modular extendable. 
    As both $M_0$ and $N$ are connected and have rank at least 2, $P' \in M_0 \otimes N$ is a connected matroid by \cref{lem:connected} and it has rank at least 4. Therefore, \cref{cor:rank4} shows that $P'$ is representable over a skew field $\bF$. As $P'$ contains $N$ as a restriction, the characteristic of $\bF$ must be in the skew characteristic set $C$ of $N$. Since $M_0$ is a restriction of $P'$, we conclude that $M_0$ is also representable over $\bF$. This finishes the proof of the theorem.
\end{proof}

Taking $N=U_{2,3}$ yields \cref{thm:skewrepr}, which we restate here for convenience.

\skewrepr*
\begin{proof}
   Since $U_{2,3}$ is regular, the statement follows directly from \cref{thm:char}.
\end{proof}

Using \cref{prop:charsets}, \cref{thm:char} immediately gives the following characterization for certain prescribed sets of characteristics.

\begin{cor} \label{cor:char_general}
    Let $C\subseteq \bP \cup \{0\}$ be a nonempty set such that either $0 \notin C$ and $|C|<\infty$, or $0 \in C$ and $|\bP\setminus C|<\infty$. 
    Then, there exists a connected rank-3 matroid $N_C$ such that a connected matroid is representable over a skew field of characteristic in $C$ if and only if it is $k$-tensor-compatible with $N_C$ for all $k \in \bZ_+$. Moreover, such a matroid $N_C$ is computable given the finite member of $\{C, \bP\setminus C\}$.
\end{cor}

We will use \cref{thm:char} to give further equivalent characterizations of skew-representability with a fixed characteristic, including \cref{thm:fixed}.
Recall that, by \cref{lem:assoc}, the set $U_{2,3}^{\otimes k}$ does not depend on the parenthesization. We first need the following lemma about the skew-representable members of this family.

\begin{lem} \label{lem:unique-power}
    Let $p \in \bP\cup \{0\}$, $k \in \bZ_+$, and let $N_{p,k}$ denote the matroid represented by the matrix $\begin{bmatrix} 1 & 0 & 1 \\ 0 & 1 & 1 \end{bmatrix}^{\otimes k}$ over $\bF_p$ if $p > 0$ and over $\bQ$ if $p=0$.
    Then, the following statements hold.
    \begin{enumerate}[label=(\alph*)]
        \item \label{it:unique-power} $N_{p,k}$ is the unique matroid in $U_{2,3}^{\otimes k}$ that is representable over a skew field of characteristic $p$.
        \item \label{it:power-char-set} If $p > 0$ and $k \ge p+1$, then the (skew) characteristic set of $N_{p,k}$ is $\{p\}$.
    \end{enumerate}
\end{lem}
\begin{proof}
    Consider $k$ copies of $U_{2,3}$ on the ground set $[3]$.
    Let $M\in U_{2,3}^{\otimes k}$ be represented by a matrix $A\in \bF^{2^k \times 3^k}$ over a skew field $\bF$.
    Let us label the rows of $A$ by $[2]^k$ and the columns of $A$ by $[3]^k$.
    As $[2]^k$ is a basis of $M$, we may assume that the columns indexed by $[2]^k$ form an identity matrix.
    For $s \in [2]^k$ and $t\in [3]^k$, let $a_{s,t}$ denote the entry of $A$ in row $s$ and column $t$, and let $a_t$ denote the $t$th column of $A$.
    For $t\in [3]^k$ and $j \in [3]$, let $I_j(t)\coloneqq \{i \in [k] \mid t_i = j\}$.
    First, we deduce which entries of $A$ are nonzero.
    For $t \in [3]^k$, let $S(t)\coloneqq \{s \in [2]^k \mid  a_{s,t}\ne 0\}$.
    We prove by induction on $|I_3(t)|$ that
    \begin{equation}
        S(t) = \{s\in [2]^k \mid I_1(t)\subseteq I_1(s) \text{ and } I_2(t)\subseteq I_2(s)\}. \label{eq:supp}
    \end{equation}
    This holds for $|I_3(t)|=0$, since the columns of $A$ indexed by $[2]^k$ form the identity matrix.
    Let $t \in [3]^k$ with $|I_3(t)|\ge 1$, and let $i \in I_3(t)$.
    Let $t_1$ and $t_2$ denote the tuples obtained from $t$ by replacing the $i$th entry of $t$ by $1$ and $2$, respectively.
    Since $M\in U_{2,3}^{\otimes k}$,
    $M|\{t_1, t_2, t\}$ is isomorphic to $U_{2,3}$, hence $a_t = a_{t_1}\lambda + a_{t_2}\mu$ for some $\lambda, \mu\in \bF\setminus \{0\}$.
    By the induction hypothesis applied to $t_1$ and $t_2$, we get $S(t_1)\cap S(t_2)=\emptyset$.
    Then, $a_t = a_{t_1}\lambda + a_{t_2}\mu$ implies that $S(t)=S(t_1)\cup S(t_2)$, proving \eqref{eq:supp}.

    Let $\mathbf{3}$ denote the $k$-tuple $(3,\dots, 3)$.
    By \eqref{eq:supp}, we have $S(\mathbf{3}) = [2]^k$.
    Then, by scaling the rows of $A$ from the left and the columns corresponding to $[2]^k$ from the right, we may assume that each entry of $a_{\mathbf{3}}$ is $1$ while keeping the columns indexed by $[2]^k$ equal to the identity matrix.
    We prove by induction on $k-|I_3(t_1)|$ that for $t_1\in [3]^k$,  the column $a_{t_1}$ can be scaled from the right such that each of its nonzero entries becomes $1$.
    This holds for $a_{\mathbf{3}}$; hence we may assume $|I_3(t_1)|<k$.
    Choose $i \in [k]\setminus I_3(t_1)$; we may assume by symmetry that $i \in I_1(t_1)$.
    Let $t_2$ and $t$ denote the $k$-tuples obtained from $t_1$ by replacing its $i$th entry by $2$ and $3$, respectively.
    By the induction hypothesis, we may scale $a_t$ from the right so that each of its nonzero entries is $1$.
    As before, we have $S(t_1)\cap S(t_2)=\emptyset$, $S(t)=S(t_1)\cup S(t_2)$, and $a_t = a_{t_1}\lambda + a_{t_2}\mu$ for some $\lambda, \mu \in \bF\setminus \{0\}$.
    As each nonzero entry of $a_t$ is $1$, it follows that each nonzero entry of $a_{t_1}$ is $\lambda^{-1}$.
    By multiplying the column $a_{t_1}$ from the right by $\lambda$, we get that all its nonzero entries are $1$.
    This finishes the proof that the columns of $A$ can be scaled from the right to get a $(0,1)$-matrix.
    Therefore, every skew-representable matroid in $U_{2,3}^{\otimes k}$ is represented by the unique $(0,1)$-matrix whose nonzero entries are given by \eqref{eq:supp}.
    This proves \ref{it:unique-power}, as a $(0,1)$-matrix represents the same matroid over each skew field of a fixed characteristic.

    Let $A$ now denote the canonical $(0,1)$-matrix above. To show \ref{it:power-char-set}, it suffices to consider $k=p+1$, since for every $k \ge p+1$, $N_{p,k}$ contains $N_{p,p+1}$ as a restriction.
    For $i \in [k]$, let $s_i$ denote the $k$-tuple obtained from the all-ones $k$-tuple by replacing the $i$th entry by $2$.
    For $t\in \{1,3\}^k$, \eqref{eq:supp} shows that 
    \begin{equation} \label{eq:supp-int}
        S(t)\cap \{s_1,\dots, s_k\} = \{s_i \mid i \in [k], t_i = 3\}.
    \end{equation}
    In particular, let $t_j$ denote the $k$-tuple obtained from $\mathbf{3}$ by replacing its $j$th entry by $1$.
    Then, \eqref{eq:supp-int} shows that the submatrix of $A$ formed by the rows $\{s_1,\dots, s_k\}$ and the columns $\{t_1,\dots, t_k\}$ is $J_k-I_k$, where $J_k$ and $I_k$ denote the $k\times k$ all-ones and identity matrix, respectively.
    Let $Z\coloneqq \{t_1,\dots, t_k\} \cup (\{1,2\}^k\setminus \{s_1,\dots, s_k\})$.
    The submatrix of $A$ with column set $Z$ has the form $\begin{bmatrix} I_{2^k-k} & * \\ 0 &  J_k-I_k \end{bmatrix}$ after a suitable ordering of rows and columns, hence its determinant is $\det(J_k-I_k) = (-1)^{k-1} (k-1) = (-1)^{k-1} p$ as an integer.
    Therefore, $Z$ is a basis of the matroid represented by $A$ over a skew field if and only if the characteristic of the skew field is not $p$.
    Then, $N_{p,k}$ differs from $N_{p',k}$ for $p'\ne p$; hence $N_{p,k}$ is representable only in characteristic $p$ by \ref{it:unique-power}.
\end{proof}

\begin{rem}
    In the previous proof, \eqref{eq:supp-int} shows that for each $p \in \bP\cup \{0\}$ and $k \in \bZ_+$, $N_{p,k}$ contains a minor that is isomorphic to the matroid represented by the $k \times (2^k-1)$ $(0,1)$-matrix with all distinct nonzero columns over $\bF_p$ if $p>0$ and over $\bQ$ if $p=0$.
    Indeed, contracting the elements of
    $[2]^k\setminus\{s_1,\dots,s_k\}$ and restricting to
    $\{1,3\}^k\setminus\{(1,\dots,1)\}$ yields precisely this matrix.
    In particular, $N_{2,k}$ contains $PG(k-1,2)$ as a minor.  
    If $p > 2$ and $k \ge p+1$, we further get that $N_{p,k}$ contains a minor isomorphic to the matroid $L_p$ represented by the matrix $\begin{bmatrix} I_{p+1} & J_{p+1}-I_{p+1} \end{bmatrix}$ over $\bF_p$.
    This strengthens statement \ref{it:power-char-set}, as     Lazarson~\cite{lazarson1958representation} showed that $L_p$ has skew characteristic set $\{p\}$; see also \cite[page~203]{oxley2011matroid} for the commutative case.
\end{rem}

We are now ready to prove the following equivalent characterization of representability over a skew field of fixed characteristic.
Note that the equivalence of \ref{it:repr} and \ref{it:tensor-U} was stated as \cref{thm:fixed} in the introduction.

\begin{thm} \label{thm:main-fixed}
    Let $M$ be a matroid and let $p \in \bP\cup \{0\}$. For $k \in \bZ_+$, let $N_{p,k}$ denote the matroid represented by the matrix $A_k\coloneqq \begin{bmatrix} 1 & 0 & 1 \\ 0 & 1 & 1 \end{bmatrix}^{\otimes k}$ over $\bF_p$ if $p > 0$, and over $\bQ$ if $p=0$.
    The following are equivalent.
    \begin{enumerate}[label=(\roman*)]\itemsep0em
        \item \label{it:repr} $M$ is representable over a skew field of characteristic $p$;
        \item \label{it:tensor-all} $M$ has a tensor product with every matroid $N$ such that $N$ is representable over every infinite field of characteristic $p$;
        \item \label{it:tensor-U} $M$ has a tensor product with $N_{p,k}$ for every $k \in \bZ_+$.
    \end{enumerate}
    Moreover, if $p > 0$, then these conditions are also equivalent to
    \begin{enumerate}[label=(\roman*), start=4]\itemsep0em
        \item \label{it:tensor-PG} $M$ is $k$-tensor-compatible with $PG(2,p)$ for every $k \in \bZ_+$.
    \end{enumerate}
\end{thm}
\begin{proof}
The implication $\ref{it:repr} \Rightarrow \ref{it:tensor-all}$ holds by \cref{lem:skew-large}.
The implication $\ref{it:tensor-all} \Rightarrow \ref{it:tensor-U}$ is immediate, since $N_{p,k}$ is represented by $A_k$ over every field of characteristic $p$.
Assume now that \ref{it:tensor-U} holds; we show \ref{it:repr}.
By \cref{lem:direct_sum} and \cref{prop:direct-sum-charset}, we may assume that $M$ is connected.
Let $\ell \in \bZ_+$ be arbitrary.
By \cref{lem:connected}, $N_{p,\ell}$ is connected.
By \cref{lem:assoc} and \ref{it:tensor-U}, for every $k \in \bZ_+$, $M\otimes N_{p,\ell}^{\otimes k} \supseteq M\otimes N_{p,k\ell}\ne \emptyset$, hence $M$ is $k$-tensor-compatible with $N_{p,\ell}$ by \cref{lem:assoc}.
Let $q$ belong to the skew characteristic set of $N_{p,\ell}$.
By \cref{lem:unique-power}\ref{it:unique-power}, we have $N_{p,\ell} = N_{q,\ell}$. Hence $N_{p,\ell}$ is represented by the $(0,1)$-matrix $A_\ell$ considered over $\bF_q$ if $q>0$, and over $\bQ$ if $q=0$, and therefore over every field of characteristic $q$.
Moreover, $N_{p,\ell}$ has rank $2^\ell \ge 2$.
Since $M$ is connected, \cref{thm:char} hence shows that $M$ is representable over a skew field whose characteristic belongs to the skew characteristic set of $N_{p,\ell}$.
If $p>0$, choose $\ell=p+1$. By \cref{lem:unique-power}\ref{it:power-char-set}, the skew characteristic set of $N_{p,\ell}$ is $\{p\}$.
The preceding application of \cref{thm:char} hence shows that $M$ is representable over a skew field of characteristic $p$.
Assume that $p=0$ and, for contradiction, that $M$ is not representable over a skew field of characteristic $0$.
Then, \cref{prop:charsets} shows that the skew characteristic set $C$ of $M$ is finite.
Choose $\ell\in \bZ_+$ such that $\ell > q$ for each $q \in C$.
For each $q\in C$, since $\ell > q$, \cref{lem:unique-power} shows that the unique member of $U_{2,3}^{\otimes \ell}$ that is representable over a skew field of characteristic $q$ has skew characteristic set $\{q\}$.
Since $N_{0,\ell}$ is representable over $\bQ$, it differs from this member.
Hence $N_{0,\ell}$ is not representable over a skew field of characteristic $q$.
Thus, the skew characteristic sets of $M$ and $N_{0,\ell}$ are disjoint, contradicting the preceding application of \cref{thm:char}.
This proves \ref{it:repr}.

It remains to consider the equivalence with \ref{it:tensor-PG} in the $p > 0$ case.
As noted in \cite{kahn1982characteristic}, $PG(2,p)$ is representable over a skew field $\bF$ precisely when $\bF$ has characteristic $p$.
\cref{thm:char} hence shows that \ref{it:tensor-PG} is equivalent to the condition that $M$ is the direct sum of matroids each of which is representable over a skew field of characteristic $p$.
By \cref{prop:direct-sum-charset}, this is equivalent to $M$ being representable over a skew field of characteristic $p$.
This proves the equivalence of \ref{it:repr} and \ref{it:tensor-PG}.
\end{proof}

\begin{rem}\label{rem:prob1} 
    For a fixed characteristic $p$, \cref{thm:main-fixed} gives a sequence of matroids $N_1,N_2,\dots$ such that a matroid is representable over a skew field of characteristic $p$ if and only if it admits a tensor product with $N_i$ for every $i\in\bZ_+$. On the other hand, \cref{thm:skewrepr} characterizes skew-representability by $k$-tensor-compatibility with $U_{2,3}$ for every $k\in\bZ_+$, that is, for each $k$ there exists some $N\in U_{2,3}^{\otimes k}$ admitting a tensor product with $M$. The choice of $N$, however, may depend on $M$. It remains open whether the witness can be chosen universally: is there a sequence $N'_1,N'_2,\dots$ such that a connected matroid $M$ is skew-representable if and only if $M$ admits a tensor product with $N_k'$ for every $k\in\bZ_+$? 
    
    Such a sequence cannot satisfy $N'_k\in U_{2,3}^{\otimes k}$ for every $k$. Indeed, fix $i$ and $p\in\bP\cup\{0\}$. Since $N_{p,\ell}$ is connected and skew-representable for every $\ell$, the existence of such a characterization would imply that $N_{p,\ell}$ has a tensor product with $N'_i$ for every $\ell$, and hence, by symmetry, that $N'_i$ has a tensor product with $N_{p,\ell}$ for every $\ell$. By Theorem~\ref{thm:main-fixed}, $N'_i$ is therefore representable in every characteristic, which is impossible for $N'_i\in U_{2,3}^{\otimes i}$ when $i\geq 3$ by Lemma~\ref{lem:unique-power}.
\end{rem}

\begin{rem}\label{rem:forintro}
    Tensor-compatibility of a matroid and tensor-compatibility of its rank function as a polymatroid function are different notions.     
    Let $\ell\in \bZ_+$ and let $M$ be a connected matroid that is $\ell$-folded skew-representable over a skew field $\bF$ of characteristic $p$.
    Let $r$ denote the rank function of $M$.
    For each $k \in \bZ_+$, by \cref{lem:representable}, $\ell \cdot r$ and the rank function of the matroid $N_{p,k}$ from \cref{thm:main-fixed} admit an $\bF$-representable tensor product.
    Dividing the resulting polymatroid function by $\ell$ shows that $r$, viewed as a polymatroid function, admits a tensor product with the rank function of $N_{p,k}$ for every $k \in \bZ_+$.
    On the other hand, there exist connected folded skew-representable matroids that are not skew-representable; Pendavingh and van Zwam~\cite{pendavingh2013skew} give such a matroid $M$ that is folded skew-representable in characteristic $3$.
    Since $M$ is not skew-representable, \cref{thm:main-fixed} shows that $M$ does not admit a matroid tensor product with $N_{3,k}$ for some $k \in \bZ_+$.
    Thus, two matroids may fail to admit a matroid tensor product even if their rank functions admit a tensor product as polymatroid functions.
\end{rem}

For a given $k \in \bZ_+$ and a given matroid $M$ on $n$ elements, it is decidable if $M$ is $k$-tensor-compatible with $U_{2,3}$, as, for example, one might enumerate all matroids on a ground set of size $n\cdot 3^k$. Hence \cref{thm:skewrepr} yields a simple verifiable certificate of the non-skew-representability of a connected matroid $M$: an integer $k\in \bZ_+$ for which $M$ is not $k$-tensor-compatible with $U_{2,3}$. Similarly, for a fixed $p\in\bP\cup\{0\}$, \cref{thm:main-fixed} gives a certificate for a matroid being non-representable over skew fields of characteristic $p$: an integer $k\in\bZ_+$ for which $M$ does not admit a tensor product with the matroid $N_{p,k}$ defined there.

\begin{rem} \label{rem:unbounded}
    Combining \cref{prop:undecidable}\ref{it:skew-undec} with \cref{thm:skewrepr}, we obtain that for any computable function $f\colon \bZ_+\to\bZ_+$, there exists an $n$-element matroid $M$ such that $M$ is $k$-tensor-compatible with $U_{2,3}$ for every $k\leq f(n)$, but not for some $k>f(n)$.
    Similarly, combining \cref{prop:undecidable}\ref{it:fixed-undec} with \cref{thm:main-fixed}, we obtain that for every $p\in\bP\cup\{0\}$ and every computable function $f\colon\bZ_+\to\bZ_+$, there exists an $n$-element matroid $M$ that admits a tensor product with $N_{p,k}$ for every $k\leq f(n)$, but not with $N_{p,k}$ for some $k>f(n)$.
\end{rem}

We call a set $C\subseteq\bP$ {\it co-recursively enumerable} if its membership problem is co-recursively enumerable; equivalently, if there is an algorithm which, given $p\in\bP$, halts if and only if $p\notin C$. When a co-recursively enumerable set $C\subseteq\bP$ is part of the input, we assume that it is given together with an algorithm which, on input $p\in\bP$, halts if and only if $p\notin C$. The following co-recursive enumerability statements can be obtained from Gödel's completeness theorem~\cite[Section~2.5]{enderton2001mathematical}. However, the tensor-product characterization gives matroid-theoretic obstructions to representability. Using \cref{cor:char_general}, we show the following computability-theoretic consequences.

\begin{thm}\label{thm:core}
    The following problems are co-recursively enumerable.
    \begin{enumerate}[label=(\alph*)]\itemsep0em
        \item \label{it:finite} Given a finite set $C \subseteq \bP$ and a matroid $M$, decide if $M$ is representable over a skew field of characteristic in $C$. 
        \item \label{it:cofinite} Given a co-recursively enumerable set $C \subseteq \bP$ and a matroid $M$, decide if $M$ is representable over a skew field of characteristic in $C\cup\{0\}$.
    \end{enumerate}
\end{thm}
\begin{proof}
    We first show~\ref{it:finite}. For each $p\in C$, check whether $M$ is $k$-tensor-compatible with $PG(2,p)$ for $k=1,2,\dots$, and stop this algorithm once a negative answer is obtained. Run these algorithms in parallel, and stop once the algorithm corresponding to every $p\in C$ has terminated. By \cref{thm:main-fixed}\ref{it:tensor-PG}, the algorithm corresponding to $p$ terminates if and only if $M$ is not representable over a skew field of characteristic $p$. Since $C$ is finite, the whole algorithm therefore terminates if and only if $M$ is not representable over any skew field of characteristic in $C$, proving~\ref{it:finite}.

    For~\ref{it:cofinite}, observe that $M$ is not representable over any skew field of characteristic in $C\cup\{0\}$ if and only if there exist a finite set $P\subseteq\bP$ and a component $N$ of $M$ such that $N$ is not representable over any skew field of characteristic in $(\bP\cup\{0\})\setminus P$ and $M$ is not representable over any skew field of characteristic in $C\cap P$. Indeed, for the `only if' direction, by \cref{prop:direct-sum-charset}, $M$ has a component $N$ which is not representable over a skew field of characteristic zero, and the skew characteristic set $P$ of $N$ is a finite subset of $\bP$ by \cref{prop:charsets}. For the `if' direction, assume that there exist a finite set $P\subseteq\bP$ and a component $N$ of $M$ satisfying the two conditions above. If $M$ were representable over a skew field of characteristic $p \in C\cup\{0\}$, then $N$ would be representable over the same skew field. The first condition implies that $p\in P$, and hence $p\in C\cap P$, contradicting the second condition.

    For a component $N$ and a given finite set $P\subseteq\bP$, the first condition is co-recursively enumerable by \cref{cor:char_general}. To check the second condition, for every $p\in P$ run in parallel the algorithm witnessing that $C$ is co-recursively enumerable and the algorithm from~\ref{it:finite} for the singleton set $\{p\}$, and stop once, for every $p\in P$, at least one of the two algorithms has terminated. This procedure terminates if and only if $M$ is not representable over any skew field of characteristic in $C\cap P$.

    Finally, enumerate all finite sets $P\subseteq\bP$ and all components $N$ of $M$, and
    run in parallel the algorithms for the two conditions above.
    Stop once both algorithms have terminated for some pair $(N,P)$. By the equivalence above, this happens if and only if $M$ is not representable over any skew field of characteristic in $C\cup\{0\}$. This proves~\ref{it:cofinite}.
\end{proof}

\begin{rem}
The proof of part~\ref{it:finite} extends directly to finite sets $C\subseteq \bP\cup\{0\}$: instead of using $PG(2,p)$, one can use the matroids $N_{p,k}$ from Theorem~\ref{thm:main-fixed}. We stated part~\ref{it:finite} only for $C\subseteq\bP$ so that parts~\ref{it:finite} and~\ref{it:cofinite} correspond to the disjoint cases according to whether characteristic zero is excluded or allowed.
\end{rem}

Clearly, every explicitly given finite or cofinite subset of $\bP\cup\{0\}$ is co-recursively enumerable. Thus \cref{thm:core} gives the following.

\begin{cor}\label{cor:core}
     The following problems are co-recursively enumerable.
    \begin{enumerate}[label=(\alph*)]\itemsep0em
        \item \label{it:finitec} Given a finite set $C \subseteq \bP\cup\{0\}$  and a matroid $M$, decide if $M$ is representable over a skew field of characteristic in $C$. 
        \item \label{it:cofinitec} Given a finite set $\overline{C} \subseteq \bP$ and a matroid $M$, decide if $M$ is representable over a skew field of characteristic in $(\bP\cup \{0\}) \setminus \overline{C}$. 
    \end{enumerate}
\end{cor}

\begin{rem}\label{rem:prob2}
    The preceding results do not cover the case when the allowed set of characteristics is cofinite in $\bP$ but does not contain $0$. That is, for a finite set $C\subseteq\bP$, it remains open whether the problem of deciding if a given matroid is representable over a skew field of characteristic in $\bP\setminus C$ is co-recursively enumerable. 
\end{rem}

\subsection{Polymatroid functions and fractional modular extendability}
\label{sec:polymatroids}

We have seen that modular extendability is closely connected to tensor-compatibility with the uniform matroid $U_{2,3}$. We now establish a similar implication for polymatroid functions under a stronger tensor-product assumption.

The construction below is the polymatroid analogue of the minor used in the proof of Theorem~\ref{thm:1modular}, and in the matroid case the two constructions agree after taking a quotient.\footnote{This construction first appeared in an earlier manuscript of Padró~\cite{padro2025tensor} and was one of the starting points of the present collaboration; the manuscript was subsequently incorporated into the present joint work.} Indeed, use the notation of that proof and let $r'$ be the rank function of $N'=N/(A_0^v\cup B_0^w)|(S_1^u\cup C^v)$. Take the quotient of $r'$ whose parts are the singletons $\{x^u\}$, $x\in S_1$, and the set $C^v$, represented by a new element $z$, and identify $\{x^u\}$ with $x$ for every $x\in S_1$. By Claim~\ref{cl:restriction}, the quotient agrees with $r_1$ on $S_1$. Since $A_0\cup C$ and $B_0$ are maximum independent sets of $A$ and $B$, respectively, \cref{lem:gen} gives $r(X^u\cup C^v\cup A_0^v\cup B_0^w)=r(X^u\cup A^v\cup B^w)$ for every $X\subseteq S_1$. Moreover, as observed at the beginning of the proof of Theorem~\ref{thm:1modular}, $r(A_0^v\cup B_0^w)=r_1(A\cup B)$. Consequently,
\[
r'(X^u\cup C^v)=r(X^u\cup A^v\cup B^w)-r_1(A\cup B)
\]
for every $X\subseteq S_1$. This is exactly the extension formula used below. Thus, in the matroid setting, the following construction is the quotient of the extension constructed in Theorem~\ref{thm:1modular}; the subtraction of $r_1(A\cup B)$ plays the role of the contraction by $A_0^v\cup B_0^w$.

\begin{lem}\label{lem:polymatroid_extension}
Let $\varphi_1\colon 2^{S_1}\to\bR$ be a polymatroid function, let $r_2$ be the rank function of $U_{2,3}$ on $S_2=\{u,v,w\}$, and let $\varphi\in\varphi_1\otimes r_2$. Let $A,B\subseteq S_1$ and let $z\notin S_1$. Define $\varphi_1'\colon 2^{S_1+z}\to\bR$ by $\varphi_1'(X)=\varphi_1(X)$ and $\varphi_1'(X+z)=\varphi(X^u\cup A^v\cup B^w)-\varphi_1(A\cup B)$ for $X\subseteq S_1$. Then, $\varphi_1'$ is a one-step extension of $\varphi_1$ with respect to $A,B$. Moreover, if $\varphi$ is integer-valued, then so is $\varphi_1'$.
\end{lem}

\begin{proof}
We first show that $\varphi_1'$ is a polymatroid function. Clearly, $\varphi_1'(\emptyset)=0$. Monotonicity between sets containing $z$ follows from that of $\varphi$, so it remains to show that $\varphi_1'(X+z)\geq\varphi_1'(X)$ for $X\subseteq S_1$. Set $W=X^u\cup A^v\cup B^w$ and $Y=(X\cup A)^u\cup B^w$. By \cref{lem:direct}, $\varphi(Y)=\varphi_1(X\cup A)+\varphi_1(B)$ and $\varphi(W\cap Y)=\varphi_1(X)+\varphi_1(B)$. By \cref{lem:gen}, $\varphi((W\cup Y)\cup A^w)=\varphi(W\cup Y)$. The resulting set contains $(X\cup A)^u\cup(A\cup B)^w$, so monotonicity and \cref{lem:direct} give $\varphi(W\cup Y)\geq\varphi_1(X\cup A)+\varphi_1(A\cup B)$. Therefore, submodularity gives
\begin{align*}
\varphi(W)&\geq \varphi(W\cap Y)+\varphi(W\cup Y)-\varphi(Y)\geq \varphi_1(X)+\varphi_1(A\cup B),
\end{align*}
which proves $\varphi_1'(X+z)\geq\varphi_1'(X)$.

We next verify submodularity using the equivalent characterization \eqref{eq:submod} from the preliminaries. If $a$ and $b$ are both in $S_1$, the required inequality follows from the submodularity of $\varphi_1$ when $X$ does not contain $z$, and from that of $\varphi$ when it does. It remains to consider $X\subseteq S_1$ and $a\in S_1\setminus X$ and $b=z$. With $W=X^u\cup A^v\cup B^w$, submodularity of $\varphi$ gives
\begin{align*}
&\varphi_1'(X+a)+\varphi_1'(X+z)-\varphi_1'(X)-\varphi_1'(X+a+z)=\varphi((X+a)^u)+\varphi(W)-\varphi(X^u)-\varphi(W+a^u)\geq 0.
\end{align*}
Thus $\varphi_1'$ is a polymatroid function.

We now show that $\varphi_1'$ is a one-step extension of $\varphi_1$ with respect to $A,B$. By \cref{lem:direct}, $\varphi(A^v\cup B^w)=\varphi_1(A)+\varphi_1(B)$, and hence $\varphi_1'(z)=\varphi_1(A)+\varphi_1(B)-\varphi_1(A\cup B)$. To determine $\varphi_1'(A+z)$, observe first that \cref{lem:gen} gives $\varphi(A^u\cup A^v\cup B^w)=\varphi((A\times S_2)\cup(A\cup B)^w)$. The latter set contains $A^u\cup(A\cup B)^w$, whose value is $\varphi_1(A)+\varphi_1(A\cup B)$ by \cref{lem:direct}. On the other hand, submodularity gives
\begin{align*}
\varphi((A\times S_2)\cup(A\cup B)^w)\leq \varphi(A\times S_2)+\varphi((A\cup B)^w)-\varphi(A^w)=\varphi_1(A)+\varphi_1(A\cup B).
\end{align*}
Therefore, $\varphi_1'(A+z)=\varphi_1(A)$. The same argument, with $A,v$ and $B,w$ interchanged, shows that $\varphi_1'(B+z)=\varphi_1(B)$. Thus $\{z\}$ is a witnessing set for the extension step.

Finally, if $\varphi$ is integer-valued, then $\varphi_1$ is integer-valued since $\varphi_1(X)=\varphi(X^u)$ for every $X\subseteq S_1$. Hence the definition of $\varphi_1'$ shows that $\varphi_1'$ is integer-valued as well.
\end{proof}

Next we show that the one-step construction of \cref{lem:polymatroid_extension} can be performed while preserving all but one factor of a tensor product with a power of $U_{2,3}$.

\begin{thm}\label{thm:fme}
    Let $\varphi\colon 2^S\to \bR$ be a polymatroid function and let $k \in \bZ_+$.
    Assume that there exists a matroid $U\in U_{2,3}^{\otimes k}$, with rank function $r_U$, and a tensor product $\psi\in \varphi\otimes r_U$.
    Then, for every pair of sets $X, Y\subseteq S$, there exists a one-step extension $\varphi'$ of $\varphi$ with respect to $X$ and $Y$ so that, if $k\geq 2$, there exists a matroid $U'\in U_{2,3}^{\otimes (k-1)}$, with rank function $r_{U'}$, and a tensor product $\psi'\in\varphi'\otimes r_{U'}$. Moreover, if $\psi$ is integer-valued, then $\varphi'$ and, when $k\geq 2$, $\psi'$ can be chosen integer-valued.
\end{thm}
\begin{proof}
    For $i\in[k]$, let $Q_i=\{u_i,v_i,w_i\}$ be the ground set of the $i$th copy of $U_{2,3}$. If $k=1$, the statement follows from \cref{lem:polymatroid_extension}, including the assumption on integrality. Assume now that $k\geq 2$, and define $Q\coloneqq\prod_{i=2}^k Q_i$. By \cref{lem:assoc}, there exists a matroid $U'=(Q,r_{U'})\in U_{2,3}^{\otimes(k-1)}$ such that $U\in U_{2,3}\otimes U'$. Since $U_{2,3}$ is connected and has positive rank, the repeated application of \cref{lem:connected} shows that $U'$ is connected; moreover, $U'$ has positive rank. Let $r_2$ denote the rank function of the first copy of $U_{2,3}$. Since $\psi\in\varphi\otimes r_U$ and $U\in U_{2,3}\otimes U'$, for $Z\subseteq S$, $R_1\subseteq Q_1$, and $R\subseteq Q$ we have
    \begin{equation*}
        \psi(Z\times R_1\times R)=\varphi(Z)\cdot r_2(R_1)\cdot r_{U'}(R).
    \end{equation*}
    Thus, \cref{lem:box}\ref{it:poly-box}, applied with $S$ and $Q_1$ as the first two factors and $Q$ as the last factor, gives a polymatroid function $\psi_1$ on $S\times Q_1$ such that $\psi_1\in\varphi\otimes r_2$ and $\psi\in \psi_1\otimes r_{U'}$.

    Let $z\notin S$ and set $S'\coloneqq S+z$. Applying \cref{lem:polymatroid_extension} to $\psi_1$, we get a one-step extension $\varphi'$ of $\varphi$ with respect to $X$ and $Y$ given by  $\varphi'(Z)=\varphi(Z)$ and $\varphi'(Z+z)=\psi_1(Z^{u_1}\cup X^{v_1}\cup Y^{w_1})-\varphi(X\cup Y)$ for $Z\subseteq S$. Since a one-step extension preserves the value on the whole ground set, we have $\varphi'(S+z)=\varphi(S)=\psi_1(S^{u_1})$, and so $\psi_1(S^{u_1}\cup X^{v_1}\cup Y^{w_1})=\psi_1(S^{u_1})+\varphi(X\cup Y)$. We will use the following technical claim.

    \begin{cla}\label{cl:1}
        For every $R\subseteq Q$, we have $\psi\left((S^{u_1}\times Q)\cup((X^{v_1}\cup Y^{w_1})\times R)\right)=\psi_1(S^{u_1})\cdot r_{U'}(Q)+\varphi(X\cup Y)\cdot r_{U'}(R)$.
    \end{cla}
    \begin{claimproof}
        Let $I$ be a maximum independent set of $U'$ in $R$, and extend $I$ to a basis $B$ of $U'$. Since $B$ spans $Q$ and $I$ spans $R$, applying \cref{lem:gen} twice gives
        \begin{equation*}
            \psi\left((S^{u_1}\times Q)\cup((X^{v_1}\cup Y^{w_1})\times R)\right)= \psi\left((S^{u_1}\times B)\cup((X^{v_1}\cup Y^{w_1})\times I)\right).
        \end{equation*}
        Since $B$ is independent, \cref{lem:direct} gives
        \begin{align*}
            \psi\left((S^{u_1}\times B)\cup((X^{v_1}\cup Y^{w_1})\times I)\right)
            &=
            |I|\cdot\psi_1(S^{u_1}\cup X^{v_1}\cup Y^{w_1})+(|B|-|I|)\cdot\psi_1(S^{u_1})\\
            &=
            \psi_1(S^{u_1})\cdot r_{U'}(Q)+\varphi(X\cup Y)\cdot r_{U'}(R),
        \end{align*}    
        concluding the proof of the claim.
    \end{claimproof}

    For $A\subseteq S'\times Q$, let $Q_A\coloneqq\{t\in Q\mid (z,t)\in A\}$ be the set of coordinates in $Q$ whose fiber in $A$ contains $z$. Define 
    $\obt A\coloneqq\{(s,u_1,t)\mid (s,t)\in A,\, s\in S\}\cup\left((X^{v_1}\cup Y^{w_1})\times Q_A\right)$. By definition, $\obt A\subseteq\obt B$ whenever $A\subseteq B$, and $\obt{A\cup B}=\obt A\cup \obt B$ and $\obt{A\cap B}=\obt A\cap\obt B$. Then, it follows from the monotonicity and submodularity of $\psi$ that $\psi(\obt A)$ is increasing and submodular as a function of $A$. Set 
    \begin{equation}
        \psi'(A)\coloneqq\psi(\obt A)-\varphi(X\cup Y)\cdot r_{U'}(Q_A).\label{eq:deffiprime}
    \end{equation} 
    We show that $\psi'$ is a polymatroid function. The equality $\psi'(\emptyset)=0$ immediately follows from the definition~\eqref{eq:deffiprime}.

    \begin{cla}\label{cl:monotone}
        $\psi'$ is increasing.
    \end{cla}
    \begin{claimproof}
        Let $A\subseteq S'\times Q$ and let $e\notin A$. If $e=(s,t)$ for some $s\in S$, then $Q_{A+e}=Q_A$, and hence \eqref{eq:deffiprime} gives $\psi'(A+e)-\psi'(A)=\psi(\obt{A+e})-\psi(\obt A)\geq 0$. Suppose now that $e=(z,t)$ and set $R\coloneqq Q_A$. If $r_{U'}(R+t)=r_{U'}(R)$, then again by \eqref{eq:deffiprime}, $\psi'(A+e)-\psi'(A)=\psi(\obt{A+e})-\psi(\obt A)\geq 0$. It remains to consider the case when $r_{U'}(R+t)=r_{U'}(R)+1$. Since $\obt A\subseteq(S^{u_1}\times Q)\cup\left((X^{v_1}\cup Y^{w_1})\times R\right)$, submodularity and \cref{cl:1} give
        \begin{align*}
            \psi(\obt{A+e})-\psi(\obt A)
            &\geq
            \psi\left((S^{u_1}\times Q)\cup((X^{v_1}\cup Y^{w_1})\times(R+t))\right)-\psi\left((S^{u_1}\times Q)\cup((X^{v_1}\cup Y^{w_1})\times R)\right)\\
            &=
            \varphi(X\cup Y).
        \end{align*}
        Since $r_{U'}(Q_{A+e})-r_{U'}(Q_A)=1$, \eqref{eq:deffiprime} yields $\psi'(A+e)-\psi'(A)=\psi(\obt{A+e})-\psi(\obt A)-\varphi(X\cup Y)\geq 0$. 
    \end{claimproof}

    \begin{cla}\label{cl:submodular}
        $\psi'$ is submodular.
    \end{cla}
    \begin{claimproof}
        Let $A\subseteq S'\times Q$ and let $e,f\notin A$ be distinct. Suppose first that $r_{U'}(Q_{A+e})+r_{U'}(Q_{A+f})=r_{U'}(Q_A)+r_{U'}(Q_{A+e+f})$. Then, by \eqref{eq:deffiprime}, we have
        \begin{equation*}
            \psi'(A+e)+\psi'(A+f)-\psi'(A)-\psi'(A+e+f)=\psi(\obt{A+e})+\psi(\obt{A+f})-\psi(\obt A)-\psi(\obt{A+e+f})\geq 0.
        \end{equation*}
        
        If $r_{U'}(Q_{A+e})+r_{U'}(Q_{A+f})\neq r_{U'}(Q_A)+r_{U'}(Q_{A+e+f})$, then necessarily $e=(z,s)$ and $f=(z,t)$ for some $s,t\in Q$. Set $R\coloneqq Q_A$. Then $r_{U'}(R+s)=r_{U'}(R+t)=r_{U'}(R)+1$ and $r_{U'}(R+s+t)=r_{U'}(R)+1$. In particular, $r_{U'}(R+s+t)=r_{U'}(R+s)$. By definition, $\obt{A+e+f}=\obt{A+e}\cup\left((X^{v_1}\cup Y^{w_1})\times\{t\}\right)$, and $\obt{A+e}$ contains $(X^{v_1}\cup Y^{w_1})\times (R+s)$. Since $r_{U'}(R+s+t)=r_{U'}(R+s)$ and $(X^{v_1}\cup Y^{w_1})\times (R+s)\subseteq \obt{A+e}$, \cref{lem:gen} shows $\psi\left((\obt{A+e})\cup\left((X^{v_1}\cup Y^{w_1})\times\{t\}\right)\right)=\psi(\obt{A+e})$.  Hence $\psi(\obt{A+e+f})=\psi(\obt{A+e})$. Moreover, $r_{U'}(Q_{A+e+f})=r_{U'}(Q_{A+e})$, so \eqref{eq:deffiprime} gives $\psi'(A+e+f)=\psi'(A+e)$. Using \cref{cl:monotone}, we get $\psi'(A+e)+\psi'(A+f)-\psi'(A)-\psi'(A+e+f)=\psi'(A+f)-\psi'(A)\geq 0$.
    \end{claimproof}

    By $\psi'(\emptyset)=0$ and \cref{cl:monotone,cl:submodular}, $\psi'$ is a polymatroid function. It remains to show that it is a tensor product of $\varphi'$ and $r_{U'}$.

    \begin{cla}\label{cl:tensor}
        $\psi'\in\varphi'\otimes r_{U'}$.
    \end{cla}
    \begin{claimproof}
        Let $Z\subseteq S$ and $R\subseteq Q$. By \eqref{eq:deffiprime}, we have
        \begin{align*}
            \psi'(Z\times R)
            &=
            \psi(Z^{u_1}\times R)\\
            &=
            \psi_1(Z^{u_1})\cdot r_{U'}(R)\\
            &=
            \varphi'(Z)\cdot r_{U'}(R),
        \end{align*}
        while
        \begin{align*}
            \psi'\left((Z+z)\times R\right)
            &=
            \psi\left((Z^{u_1}\cup X^{v_1}\cup Y^{w_1})\times R\right)-\varphi(X\cup Y)\cdot r_{U'}(R)\\
            &=
            \left(\psi_1(Z^{u_1}\cup X^{v_1}\cup Y^{w_1})-\varphi(X\cup Y)\right)\cdot r_{U'}(R)\\
            &=
            \varphi'(Z+z)\cdot r_{U'}(R).
        \end{align*}
        Since every subset of $S'$ is either of the form $Z$ or $Z+z$ for some $Z\subseteq S$, it follows that $\psi'\in\varphi'\otimes r_{U'}$.
    \end{claimproof}

    Finally, suppose that $\psi$ is integer-valued. Since $U'$ is connected and has positive rank, every element of $U'$ is a nonloop, and hence $r_{U'}(t)=1$ for every $t\in Q$. Thus, for every set $W\subseteq S\times Q_1$ and $t\in Q$, we have $\psi_1(W)=\psi(W\times\{t\})$, so $\psi_1$ is integer-valued. Therefore, \cref{lem:polymatroid_extension} provides an integer-valued $\varphi'$. Moreover, $\varphi(X\cup Y)=\psi_1((X\cup Y)^{u_1})$ is an integer, hence \eqref{eq:deffiprime} shows that $\psi'$ is integer-valued as well.
\end{proof}

As a corollary, we get the following weaker analogue of the implication from tensor-compatibility to modular extendability.

\begin{cor}\label{cor:ffme}
    Let $\varphi\colon 2^S\to\bR$ be a polymatroid function and let $k\in\bZ_+$. Assume that there exists a matroid $U\in U_{2,3}^{\otimes k}$, with rank function $r_U$, and a tensor product $\psi\in\varphi\otimes r_U$. Then $\varphi$ is $k$-fractionally modular extendable. Moreover, if $\psi$ can be chosen integer-valued, then $\varphi$ is $k$-modular extendable.    
\end{cor}
\begin{proof}
    The first half follows by repeated application of \cref{thm:fme}. If $\psi$ is integer-valued, then at each step the one-step extension and the remaining tensor product can be chosen integer-valued, proving the second half.
\end{proof}

\begin{rem}\label{rem:soda_gap}
        The conference version of this work stated the stronger result that $k$-tensor-compatibility with the rank function of $U_{2,3}$ implies $k$-fractional modular extendability, and that the corresponding integer-valued version implies $k$-modular extendability~\cite[Theorem~5.13]{berczi2026interaction}; see also \cite[Corollary~5.14]{berczi2026interaction} for the corresponding statement on full fractional modular extendability. The proof given there contained an error, which was identified with the assistance of ChatGPT-5.6 Sol. The replacement argument used here was developed by the authors without AI assistance. \cref{thm:fme} uses a stronger assumption, namely that $\varphi$ admits a tensor product with the rank function of some $U\in U_{2,3}^{\otimes k}$, but also gives a stronger conclusion: the one-step extension can be chosen so that it still admits a tensor product with the rank function of a member of $U_{2,3}^{\otimes(k-1)}$. Iterating the statement results in \cref{cor:ffme}, which is weaker than the result stated in the conference version.
\end{rem}

\section{Uniform and rank-3 matroids} 
\label{sec:uniform}

In this section, we prove that every rank-$3$ matroid $M$ admits a freest tensor product with every uniform matroid $U_{k,n}$. By \emph{freest}, we mean a tensor product $P\in M\otimes U_{k,n}$ such that every set that is independent in some tensor product $P'\in M\otimes U_{k,n}$ is also independent in $P$. Equivalently, $P$ is maximum under the weak order among all matroids in $M\otimes U_{k,n}$.
We conclude the section with several further consequences and remarks.
We show that every fully modular extendable matroid admits a tensor product with every uniform matroid (\cref{cor:tensor}), compare our construction with the tensor matroids studied in~\cite{brakensiek2024rigidity} (\cref{rem:uniuni}), and give an example showing that the freest tensor product need not be skew-representable even when the rank-$3$ factor is representable over every sufficiently large field (\cref{ex:nonskewfreest}).

We restate \cref{thm:uniform} for convenience.

\uniform*
\begin{proof} 
    Consider the rank-$3$ matroid $M$ on ground set $S$  with rank function $r$ and closure operator $\cl$. By \cref{lem:simple} we may assume that $M$ is simple. 
    Let $\cL$ denote the set of lines of $M$, let $\cP\coloneqq \{\{L_1, L_2\} \mid L_1, L_2 \in \cL, L_1 \cap L_2 = \emptyset\}$, and consider pairwise distinct new elements $e_P$ for $P \in \cP$ which are also distinct from the elements of $S$. For a line $L\in \cL$, define $\tpm L \coloneqq L\cup \{e_{\{L, L'\}} \mid L' \in \cL, \{L, L'\} \in \cP \}$.
    As $|\tpm{L} \cap \tpm{L}'| = 1$ for any two distinct $L, L' \in \cL$, by \cite[Proposition~1.5.6]{oxley2011matroid} there is a simple rank-3 matroid $\tpm{M}$ on ground set $\tpm S \coloneqq S \cup \{e_P \mid P \in \cP\}$ whose nontrivial lines are precisely $\{\tpm{L} \mid L \in \cL, |\tpm{L}|\ge 3\}$. Note that $\tpm{M} | S = M$ and $\tpm{L}$ is a line of $\tpm{M}$ for each $L\in \cL$. Let $\tpm r$ denote the rank function and $\tpm \cl$ the closure operator of $\tpm M$.
    
    To motivate our construction of a tensor product of $M$ with a uniform matroid $N$, we first show the following claim. As the proof of the claim does not use the uniformity assumption, we present it for a general matroid $N$.

    \begin{cla}  \label{cl:freest}
        Let $N$ be a matroid on ground set $[n]$ and $P \in M \otimes N$ be any tensor product of $M$ and $N$. Then, for any basis $A=A_1^1 \cup \dots \cup A_n^n$ of $P$, the following statements hold:
        \begin{enumerate}[label={\textup{(\alph*)}}]\itemsep0em
            \item $|A| = 3 \cdot r_N([n])$ and $A_i$ is an independent set of $M$ for each $i \in [n]$,\label{it:triv} 
            \item $\bigcap_{i \in C} \tpm \cl(A_i) = \emptyset$ for any circuit $C$ of $N$, \label{it:cl}
            \item $\bigcup_{i \in C^*} A_i$ is a spanning set of $M$ for any cocircuit $C^*$ of $N$. \label{it:generator}
        \end{enumerate}
    \end{cla}
    \begin{claimproof}
        It is clear that \ref{it:triv} holds.
        Observe that \ref{it:cl} is equivalent to the conjunction of the conditions
        \begin{align}
            & \bigcap_{i \in C} \cl(A_i) = \emptyset \text{ for each circuit $C$ of $N$}, and\label{eq:cl} \\
            & \{i\in [n] \mid L_1 \subseteq \cl(A_i) \text{ or } L_2 \subseteq \cl(A_i) \} \text{ is independent in $N$ for any disjoint lines $L_1, L_2$ of $M$}. \label{eq:two_lines} 
        \end{align}
        Indeed, \eqref{eq:cl} is equivalent to the condition that $\left(\bigcap_{i \in C} \tpm\cl(A_i)\right) \cap S = \emptyset$ for each circuit $C$ of $N$, while \eqref{eq:two_lines} is equivalent to $\left(\bigcap_{i \in C} \tpm \cl(A_i)\right) \setminus S = \emptyset$ for each circuit $C$ of $N$. To see the latter, observe that it is equivalent to the condition that for any two disjoint lines $L_1$ and $L_2$ of $M$, $\{i\in [n] \mid e_{\{L_1, L_2\}} \in \tpm\cl(A_i)\}$ is independent in $N$, and also observe that this set coincides with  $\{i\in [n] \mid L_1 \subseteq \cl(A_i) \text{ or } L_2 \subseteq \cl(A_i) \}$.

        To show \eqref{eq:cl}, assume that there exists a circuit $C$ of $N$ and an element $e \in \bigcap_{i \in C} \cl(A_i)$.
        For each $i \in C$, let $A'_i$ be an independent set of $M$ which is maximal in $A_i+e$ and contains $e$.
        Since $e \in \cl(A_i)$, we have $|A'_i| = |A_i|$ and $A_i+e \subseteq \cl(A'_i)$. Pick an index $i_0 \in C$ and define $A'_C\coloneqq (A'_{i_0}-e)^{i_0} \cup \bigcup_{i \in C\setminus \{i_0\}} (A'_i)^i$.  As $P$ is a tensor product of $M$ and $N$,  $A'_C$ spans $(e, i_0)$ in $P$ as $^e C$ is a circuit of this matroid.
        Hence $A_C'$ spans $\bigcup_{i \in C} (A'_i)^i$ in $P$, thus it also spans $A_C\coloneqq \bigcup_{i \in C} A_i^i$ as $A'_i$ spans $A_i$ in $M$. This is a contradiction as $|A'_C| = |A_C|-1$ and $A_C\subseteq A$ is independent in $P$.

        We next show that \eqref{eq:two_lines} holds. Let $L_1$ and $L_2$ be two disjoint lines of $M$ and let $I_t \coloneqq \{i \in [n] \mid L_t \subseteq \cl(A_i)\}$ for $t \in [2]$. We have to show that $I_1 \cup I_2$ is independent in $N$. 
        Define $X_1 \coloneqq (L_1 \times (I_1 \cup I_2)) \cup (S \times I_2)$ and $X_2 \coloneqq (L_2 \times (I_1 \cup I_2)) \cup (S \times I_1)$.
        We will use the submodularity inequality $r_P(X_1 \cap X_2) + r_P(X_1 \cup X_2) \le r_P(X_1)+r_P(X_2)$. 
        We have     
        \begin{align*}
            r_P(X_1) & = r_P((L_1 \times (I_1 \cup I_2)) \cup (S \times I_2)) \\ & \le r_P(L_1 \times (I_1 \cup I_2)) + r_P(S \times I_2) - r_P(L_1 \times I_2) \\
            & = 2 r_N(I_1 \cup I_2) + 3r_N(I_2) - 2r_N(I_2) \\
            & = 2 r_N(I_1 \cup I_2) + r_N(I_2),
        \end{align*}
        and similarly $r_P(X_2) \le 2r_N(I_1 \cup I_2)+r_N(I_1)$. 
        Since $r_P(X_1\cup X_2)=r_P(S\times (I_1\cup I_2))=3r_N(I_1\cup I_2)$, we get
        \begin{align*}
            r_P(X_1 \cap X_2) & \le r_P(X_1) + r_P(X_2) - r_P(X_1 \cup X_2) \\
            & \le (2r_N(I_1 \cup I_2)+r_N(I_2)) + (2r_N(I_1 \cup I_2)+r_N(I_1)) - 3r_N(I_1 \cup I_2) \\
            & = r_N(I_1) + r_N(I_2) + r_N(I_1 \cup I_2).
        \end{align*}
        On the other hand, if $i \in I_1 \cap I_2$, then $L_1 \cup L_2 \subseteq \cl(A_i)$, thus $|A_i| = 3$. If $i \in I_1 \setminus I_2$ then $L_1 \subseteq \cl(A_i)$ and $L_2 \not \subseteq \cl(A_i)$, thus $L_1 = \cl(A_i)$ and $|A_i| = 2$. Similarly, if $i \in I_2 \setminus I_1$ then $L_2 = \cl(A_i)$ and $|A_i| = 2$. As $X_1 \cap X_2 = (S\times (I_1 \cap I_2)) \cup (L_1\times (I_1 \setminus I_2)) \cup (L_2\times (I_2 \setminus I_1))$, these imply that 
        \[|X_1 \cap X_2 \cap A| = 3\cdot|I_1 \cap I_2| + 2\cdot|I_1 \setminus I_2| + 2\cdot|I_2 \setminus I_1| = |I_1| + |I_2| + |I_1 \cup I_2|.\]
        As $A$ is a basis of $P$, we obtain
        \[|I_1| + |I_2| + |I_1 \cup I_2|  = |X_1 \cap X_2 \cap A | \le r_P(X_1 \cap X_2) \le r_N(I_1)+r_N(I_2) + r_N(I_1\cup I_2),
        \]
        thus $I_1$, $I_2$, and $I_1 \cup I_2$ are independent in $N$. This shows \eqref{eq:two_lines}.
    
        Finally, to show \ref{it:generator}, assume that there exists a cocircuit $C^*$ of $N$ such that $\bigcup_{i \in C^*} A_i$ is not a spanning set of $M$, that is, $\bigcup_{i \in C^*} A_i \subseteq L$ for a line $L$ of $M$. This implies that $A \subseteq (S \times ([n]\setminus C^*)) \cup (L \times [n])$. Using the submodularity of $r_P$,
        \begin{align*}
            r_P((S \times ([n]\setminus C^*)) \cup (L \times [n])) & \le r_P(S \times ([n] \setminus C^*)) + r_P(L \times [n])  - r_P(L \times ([n] \setminus C^*)) \\
            & = 3 \cdot (r_N([n])-1) + 2 \cdot r_N([n]) - 2 \cdot (r_N([n])-1) \\
            & = 3 r_N([n]) -1,
        \end{align*}
        contradicting $A$ being a basis of $P$. This finishes the proof of the claim.
    \end{claimproof}

    Consider a uniform matroid $U_{k,n}$ on ground set $[n]$. We may assume that $k \ge 1$. We define $\cB$ as the family of subsets of $S \times [n]$ satisfying \ref{it:triv}, \ref{it:cl}, and \ref{it:generator} for $N=U_{k,n}$, that is, 
    \begin{align*}
        \cB \coloneqq 
        \{A_1^1 \cup \dots \cup A_n^n  \mid &  A_i\text{ is an independent set of $M$ for $i \in [n]$}, \\
        & \sum_{i=1}^n |A_i| = 3\cdot k, \\
        & \bigcap_{i \in I} \tpm\cl(A_i) = \emptyset \text{ for each $I\subseteq [n]$ with $|I| = k+1$}, \\
        & \bigcup_{i \in J} A_i \text{ is a spanning set of $M$ for each $J\subseteq [n]$ with $|J| = n-k+1$}\}.
    \end{align*}

    We will show that $\cB$ forms the family of bases of a matroid on $S \times [n]$ which is a tensor product of $M$ and $U_{k,n}$. The following easy observation will be used later.

\begin{cla} \label{cl:AtimesK}
    Let $A$ be a basis of $M$ and $K\subseteq [n]$ be a subset with $|K| = k$. Then, $A \times K \in \cB$.
\end{cla}
\begin{claimproof}
    The statement follows by observing that each $(k+1)$-element subset $I\subseteq [n]$ intersects $[n] \setminus K$ and each $(n-k+1)$-element subset $J \subseteq [n]$ intersects $K$. 
\end{claimproof}

Next, we show two technical claims that will be used in proving that $\cB$ forms the basis family of a matroid.

\begin{cla} \label{cl:one_nongen}
    Let $A'_1,\dots, A'_n \subseteq S$, $p \in [n]$, and $x \in S \setminus A'_p$ such that $A_1^{\prime 1} \cup\dots\cup A^{\prime n}_n + (x, p) \in \cB$. Then, there exists at most one subset $J \subseteq [n]$ such that $|J| = n-k+1$ and $\bigcup_{i\in J} A'_i$ is not a spanning set of $M$. 
\end{cla}
\begin{claimproof} 
    Define $A_i \coloneqq A'_i$ for $i \in [n]\setminus \{p\}$, $A_p \coloneqq A'_p + x$, and $A \coloneqq A_1^1 \cup \dots \cup A_n^n$. 
    Suppose to the contrary that there exist subsets $J_1, J_2 \subseteq [n]$ with $J_1 \ne J_2$ such that $|J_1| = |J_2| = n-k+1$ and neither $\bigcup_{i \in J_1} A'_i$ nor $\bigcup_{i \in J_2} A'_i$ is a spanning set of $M$.    
    Observe that $p \in J_1 \cap J_2$ since $A'_i = A_i$ for $i \in [n]\setminus \{p\}$ and both $\bigcup_{i \in J_1} A_i$ and $\bigcup_{i \in J_2} A_i$ are spanning sets of $M$. 
    For $t \in [2]$, let $L_t$ be a line in $M$ such that $\bigcup_{i \in J_t} A'_i \subseteq L_t$. 
    Note that $L_1 \ne L_2$, since $\bigcup_{i \in (J_1\cup J_2) \setminus \{p\}} A'_i = \bigcup_{i \in (J_1\cup J_2) \setminus \{p\}} A_i$ is a spanning set of $M$ contained in $L_1 \cup L_2$. 
    Let $e$ be the unique element of  $\tpm{L_1} \cap \tpm{L_2}$. For $i \in [n]$, define $s_i \coloneqq 1$ if $e \in \tpm \cl(A_i)$, and $s_i \coloneqq 0$ otherwise. Since $\bigcap_{i \in I} \tpm\cl(A_i) = \emptyset$ for each $I \subseteq [n]$ with $|I| = k+1$, there exist at most $k$ indices $i$ with $e \in \tpm\cl(A_i)$, that is, we have $\sum_{i=1}^n s_i \le k$.
    For an index $i \in J_1 \setminus J_2$, we have that $A'_i \subseteq L_1$ and $A'_i$ is an independent set, thus $|A'_i| \le r(L_1) = 2$. If $|A'_i| = 2$, then $A'_i$ spans $L_1$ in $M$ and thus spans $\tpm{L_1}$ in $\tpm M$, hence $|A'_i|\le s_i+1$ holds as well. Similarly, $|A'_i| \le s_i+1$ holds for $i \in J_2 \setminus J_1$. For $i \in J_1 \cap J_2$, we have $A'_i \subseteq L_1 \cap L_2 \subseteq \{e\}$, thus $|A'_i| \le s_i$. Finally, it is clear that $|A'_i| \le s_i+2$ holds for $i \in [n] \setminus (J_1 \cup J_2)$. These together imply that
    \begin{align*}
        \sum_{i=1}^n |A'_i| & \le \sum_{i \in (J_1 \cup J_2)\setminus (J_1\cap J_2)} (s_i + 1) + \sum_{i \in J_1 \cap J_2} s_i + \sum_{i \in [n]\setminus (J_1 \cup J_2)} (s_i+2) \\
        & = \sum_{i=1}^n s_i + 2n - |J_1|-|J_2| \\
        & \le k + 2n - (n-k+1) - (n-k+1) \\
        & = 3\cdot k-2.
    \end{align*}
    This contradicts $\sum_{i=1}^n |A'_i| = |A|-1 = 3\cdot k-1$. 
\end{claimproof} 

\begin{cla} \label{cl:one_intersecting}
    Let $A'_1,\dots, A'_n \subseteq S$, $p \in [n]$, and $x \in S \setminus A'_p$ such that $A^{\prime 1}_1 \cup\dots\cup A^{\prime n}_n + (x, p) \in \cB$.
    Assume that $\bigcup_{i \in J} A'_i$ is a spanning set of $M$ for each subset $J\subseteq [n]$ with $|J|=n-k+1$. Then, $|\bigcap_{i \in I} \tpm{\cl}(A'_i)| \le 1$ holds for each $I\subseteq [n]$ with $|I|=k$, and there exists at most one $k$-element subset $I\subseteq [n]$ for which this intersection is nonempty.
\end{cla}
\begin{claimproof} 
    Define $A_i \coloneqq A'_i$ for $i \in [n]\setminus \{p\}$, $A_p \coloneqq A'_p + x$, and $A \coloneqq A^1_1 \cup \dots \cup A^n_n$.
    Suppose first that $|\bigcap_{i \in I} \tpm{\cl}(A'_i)| \ge 2$ for a subset $I\subseteq [n]$ with $|I|=k$. As $\tpm{M}$ is simple, this implies that there exists a line $\tpm{L}$ of $\tpm{M}$ with $\tpm L \subseteq \bigcap_{i \in I} \tpm{\cl}(A'_i)$.     
    For $i \in [n]$, let $s_i \coloneqq 1$ if $\tpm \cl (A'_i) \setminus \tpm L \ne \emptyset$ and $s_i \coloneqq 0$ otherwise. It is clear that $|A'_i| = \tpm r(\tpm \cl(A'_i)) \ge s_i + 2$ for $i \in I$ and $|A'_i| \ge s_i$ for $i \in [n]\setminus I$. This shows that \[3\cdot k-1 = \sum_{i=1}^n |A'_i| \ge 2\cdot|I| + \sum_{i=1}^n s_i = 2\cdot k + \sum_{i=1}^n s_i,\]
    thus $\sum_{i=1}^n s_i \le k-1$.     Let $J\coloneqq \{i \in [n] \mid A'_i \subseteq \tpm{L}\}$, then $|[n]\setminus J| \le \sum_{i=1}^n s_i \le k-1$, thus $|J| \ge n-k+1$. Therefore, $\bigcup_{i \in J} A'_i$ is a spanning set by our assumption, contradicting $\bigcup_{i \in J} A'_i \subseteq \tpm{L}$. 
 
    Suppose now that there exist subsets $I_1, I_2 \subseteq [n]$ with $I_1 \ne I_2$ such that $|I_1|=|I_2| = k$ and neither $\bigcap_{i\in I_1} \tpm{\cl}(A'_i)$ nor $\bigcap_{i \in I_2} \tpm{\cl}(A'_i)$ is empty. 
    Let $e_1 \in \bigcap_{i \in I_1} \tpm{\cl}(A'_i)$ and $e_2 \in \bigcap_{i \in I_2} \tpm{\cl}(A'_i)$.
    We have $e_1 \ne e_2$, as otherwise $e_1 = e_2 \in \bigcap_{i \in I_1 \cup I_2} \tpm{\cl}(A'_i) \subseteq \bigcap_{i \in I_1 \cup I_2} \tpm{\cl}(A_i)$, where $|I_1 \cup I_2| \ge k+1$, contradicting $A \in \cB$.
    Define $\tpm L\coloneqq \tpm{\cl}(\{e_1, e_2\})$.
    For $i \in I_1\cup I_2$, we have $e_1 \in \tpm{\cl}(A'_i)$ or $e_2 \in \tpm{\cl}(A'_i)$, thus $\tpm{r}(\tpm{\cl}(A'_i) \cap \tpm L) \ge 1$.
    For $i \in I_1 \cap I_2$, we have $\{e_1, e_2\} \subseteq \tpm{\cl}(A'_i)$, hence $\tpm{r} (\tpm{\cl}(A'_i) \cap \tpm L) \ge 2$.
    These imply that \[\sum_{i=1}^n \tpm{r}(\tpm{\cl}(A'_i) \cap \tpm L) \ge |I_1 \cup I_2| + |I_1 \cap I_2| = |I_1|+|I_2| = 2\cdot k.\]
    Since $\bigcup_{i \in J} A'_i$ is a spanning set of $M$ for each $J \subseteq [n]$ with $|J| = n-k+1$, we have $|J| \le n-k$ for $J\coloneqq \{i \in [n] \mid A'_i \subseteq \tpm L\}$. For each $i \in [n] \setminus J$, we have $\tpm{\cl}(A'_i) \cap \tpm L \subsetneq \tpm{\cl}(A'_i)$ for the flats $\tpm{\cl}(A'_i) \cap \tpm L$ and $\tpm{\cl}(A'_i)$, hence $\tpm{r}(\tpm{\cl}(A'_i) \cap \tpm L) + 1 \le  \tpm{r}(\tpm{\cl}(A'_i)) = |A'_i|$.  It is also clear that $\tpm{r}(\tpm{\cl}(A'_i) \cap \tpm L) \le |A'_i|$ for each $i \in [n]$.
    These together imply that
    \begin{align*}
        \sum_{i=1}^n |A'_i| & \ge \sum_{i \in [n] \setminus J} (\tpm{r}(\tpm{\cl}(A'_i) \cap \tpm L) +1) + \sum_{i \in J} \tpm{r}(\tpm{\cl}(A'_i) \cap \tpm L) \\
        & \ge |[n]\setminus J| + \sum_{i=1}^n \tpm{r}(\tpm{\cl}(A'_i) \cap \tpm L) \\
        & \ge k + 2\cdot k = 3\cdot k, 
    \end{align*}
    contradicting $\sum_{i=1}^n |A'_i| = |A|-1 = 3\cdot k-1$. This finishes the proof of the claim. 
\end{claimproof}

Using these two claims, we are ready to show that $\cB$ satisfies the basis exchange axiom.

\begin{cla} \label{cl:matroid}
    $\cB$ forms the family of bases of a matroid on $S\times [n]$. 
\end{cla}
\begin{claimproof}
    $\cB \ne \emptyset$ follows from \cref{cl:AtimesK}.
    To show that $\cB$ satisfies the basis exchange axiom,
    let $A=A^1_1\cup \dots \cup A^n_n$ and $B=B^1_1\cup \dots \cup B^n_n$ be members of $\cB$ and let $(x, p) \in A\setminus B$. We need to show that there exists $(y, q) \in B\setminus A$ such that $A - (x, p) + (y, q) \in \cB$. Let $A'_i \coloneqq A_i$ for $i \in [n]\setminus \{p\}$ and $A'_p \coloneqq A_p - x$.

    First, consider the case when there exists a subset $J\subseteq [n]$ such that $|J| = n-k+1$ and $\bigcup_{i \in J} A'_i$ is not a spanning set of $M$. 
    Since $A\in \cB$, $\bigcup_{i \in J} A_i$ is a spanning set of $M$, thus $p \in J$ and $r(\bigcup_{i \in J} A'_i) = 2$. 
    Since $B\in \cB$, $\bigcup_{i \in J} B_i$ is also a spanning set of $M$, thus there exists an index $q \in J$ and an element $y \in B_q \setminus \cl(\bigcup_{i \in J} A'_i)$.
    Define $A''_i \coloneqq A'_i$ for $i \in [n] \setminus \{q\}$ and $A''_q \coloneqq A'_q + y$.
    We claim that $A''\coloneqq A^{\prime\prime 1}_1 \cup \dots \cup A^{\prime\prime n}_n = A-(x, p)+ (y, q) \in \cB$.
    We have $|A''| = |A| = 3\cdot k$ and it is also clear that $A''_i$ is independent for each $i \in [n]$ since $y \notin \cl(A'_q)$.
    Since $r(\bigcup_{i \in J} A'_i ) = 2$, $\bigcup_{i \in J} A''_i$ is a spanning set by the choice of $y$.
    By \cref{cl:one_nongen}, $\bigcup_{i  \in J'} A'_i$ is a spanning set of $M$ for each $J'\subseteq [n]$ such that $|J'| = n-k+1$ and $J'\ne J$, hence $\bigcup_{i \in J'} A''_i$ also spans $M$.
    It remains to show that $\bigcap_{i \in I} \tpm{\cl}(A''_i) = \emptyset$ for each $I \subseteq [n]$ with $|I|=k+1$.
    Suppose to the contrary that $e \in \bigcap_{i \in I} \tpm\cl(A''_i) $ for some $I\subseteq [n]$ with $|I| = k+1$.
    Since $A\in \cB$, we have $\bigcap_{i \in I}\tpm\cl(A_i)=\emptyset$, and thus $\bigcap_{i\in I}{\tpm\cl(A'_i)}=\emptyset$.
    As $e \in \bigcap_{i \in I} \tpm\cl(A''_i)$ and $A''_i=A'_i$ for $i \in [n]\setminus \{q\}$, it follows that $q \in I$, $e \notin \tpm\cl(A'_q)$, and $e \in \tpm\cl(A''_q)=\tpm\cl(A'_q+y)$.
    The assertions $e \notin \tpm\cl(A'_q)$ and $e\in \tpm\cl(A'_q + y)$ together imply that $y \in \tpm\cl (A'_q+e)$.
    Since $|J| = n-k+1$ and $|I\setminus \{q\}| = k$, there exists an index $i_0 \in (I \cap J) \setminus \{q\}$.
    Then, $e \in \tpm \cl(A''_{i_0}) = \tpm \cl(A'_{i_0})$ holds, thus $\tpm \cl(A'_q \cup A'_{i_0}) \supseteq \tpm \cl(A'_q+e) \ni y$. 
    Hence $y \in \cl(A'_q\cup A'_{i_0}) \subseteq \cl(\bigcup_{i \in J} A'_i)$, a contradiction.
        
    It remains to consider the cases when $\bigcup_{i \in J} A'_i$ is a spanning set of $M$ for each $J \subseteq [n]$ with $|J| = n-k+1$. 
    Assume first that there exists a subset $I\subseteq [n]$ such that $|I|=k$ and $\bigcap_{i \in I} \tpm \cl(A'_i) \ne \emptyset$.
    By \cref{cl:one_intersecting}, $\bigcap_{i \in I'} \tpm\cl (A'_i) = \emptyset$ holds for each $I'\subseteq [n]$ with $|I'| = k$ and $I \ne I'$ and there is an element $e \in \tpm S$ such that $\bigcap_{i \in I} \tpm \cl(A'_i) = \{e\}$.
    We claim that there is an index $q\in [n]$ such that $B_q\setminus \tpm \cl(A'_q+e) \ne \emptyset$.
    Suppose to the contrary that $B_i \subseteq \tpm\cl (A'_i+e)$ for each $i \in [n]$.
    Let $s_i \coloneqq 1$ if $B_i$ spans $e$ in $\tpm M$ and $s_i \coloneqq 0$ otherwise.
    As $B \in \cB$, we have $\sum_{i=1}^n s_i \le k$. For $i \in I$, we have $|B_i|\le |A'_i|-1+s_i$.
    Indeed, $B_i\subseteq \tpm\cl(A'_i)$ implies that $|B_i|\le |A'_i|$, and if equality holds, then $\tpm\cl(B_i)=\tpm\cl(A'_i)\ni e$ thus $s_i=1$.
    For $i \in [n] \setminus I$, we have $B_i \subseteq \tpm \cl(A'_i+e)$, thus $|B_i| \le |A'_i| + s_i$. Combining these, we get 
    \begin{align*}
        3\cdot k 
        &= \sum_{i=1}^n |B_i|\\
        &\le \sum_{i \in I} (|A'_i|-1+s_i) + \sum_{i \in [n]\setminus I} (|A'_i|+s_i)\\
        &= \sum_{i=1}^n |A'_i| -|I| + \sum_{i = 1}^n s_i\\
        &\le 3\cdot k-1-k+k\\
        &= 3\cdot k-1,
    \end{align*}
    a contradiction.
    This shows our claim that there exists an index $q \in [n]$ such that $B_q \setminus \tpm{\cl}(A'_q+e) \ne \emptyset$.
    Let $y \in B_q \setminus \tpm{\cl}(A'_q+e)$; we show that $A''\coloneqq A-(x,p)+(y,q) \in \cB$.
    Define $A''_i \coloneqq A'_i$ for $i \in [n] \setminus \{q\}$ and $A''_q \coloneqq A'_q+y$.
    It is clear that $|A''| = 3\cdot k $ and $A''_i$ is independent for each $i \in [n]$ as $y \notin \cl(A'_q)$.
    It is also clear that $\bigcup_{i \in J} A''_i$ is a spanning set for each $J\subseteq [n]$ with $|J|=n-k+1$ since $\bigcup_{i \in J} A'_i$ is already a spanning set. 
    For $I' \subseteq [n]$ with $|I'| = k+1$ and $q \notin I'$, we have $\bigcap_{i \in I'} \tpm \cl(A''_i) = \bigcap_{i \in I'} \tpm \cl(A'_i) = \emptyset$.
    For $I'\subseteq [n]$ with $|I'| = k+1$, $q \in I'$ and $I'\setminus \{q\} \ne I$, we have $\bigcap_{i \in I'} \tpm \cl(A''_i) \subseteq \bigcap_{i \in I'\setminus \{q\}} \tpm \cl(A''_i) = \bigcap_{i \in I'\setminus \{q\}} \tpm \cl(A'_i) = \emptyset$.
    Therefore, we only have to consider the case $q \notin I$ and show that $\bigcap_{i \in I\cup \{q\}} \tpm \cl (A''_i) = \emptyset$.
    Note that $\bigcap_{i \in I} \tpm \cl (A''_i) = \bigcap_{i \in I} \tpm \cl (A'_i) = \{e\}$.
    As $A \in \cB$, we have $\bigcap_{i \in I\cup \{q\}} \tpm \cl (A'_i) \subseteq \bigcap_{i \in I\cup \{q\}} \tpm \cl (A_i) = \emptyset$, thus $e \notin \tpm \cl(A'_q)$.
    Since $y \notin \tpm \cl (A'_q+e)$, we get that $e \notin \tpm \cl (A'_q+y)=\tpm \cl (A''_q)$.
    This shows $\bigcap_{i \in I\cup \{q\}} \tpm\cl (A''_i) = \emptyset$, finishing the proof of the case.

    Finally, it only remains to consider the case when $\bigcup_{i \in J} A'_i$ is a spanning set of $M$ for each $J\subseteq [n]$ with $|J|=n-k+1$ and $\bigcap_{i \in I} \tpm \cl(A'_i) = \emptyset$ for each $I \subseteq [n]$ with $|I|=k$. In this case, $A-(x, p) + (y, q) \in \cB$ for each $q\in [n]$ and $y \in B_q \setminus \cl(A'_q)$. Since $\sum_{i=1}^n |B_i| = 3\cdot k > 3\cdot k-1 = \sum_{i=1}^n |A'_i| = \sum_{i=1}^n  r(\cl(A'_i))$, there exists an index $q \in [n]$ with $B_q \setminus \cl(A'_q) \ne \emptyset$, finishing the proof of the claim.
\end{claimproof}

By \cref{cl:matroid}, $\cB$ forms the basis family of a matroid $P$ on $S\times [n]$. The next claim shows that it is a tensor product of $M$ and $U_{k,n}$.

\begin{cla} \label{cl:Ntensor}
    $P$ is a tensor product of $M$ and $U_{k,n}$.
\end{cla}
\begin{claimproof}
    For $i \in [n]$, any independent set of $P|S^i$ corresponds to an independent set of $M$ by the definition of $\cB$, and any basis of $M$ gives a basis of $P|S^i$ by \cref{cl:AtimesK}, thus $P|S^i$ is isomorphic to $M$.
    Since $\bigcap_{i \in I} \tpm{\cl}(A_i) = \emptyset$ holds for each $(k+1)$-element subset $I \subseteq [n]$ and basis $A^1_1\cup \dots \cup A^n_n \in \cB$, the restriction $P|(\{e\} \times [n])$ has rank at most $k$ for each $e \in S$. Then, \cref{cl:AtimesK} implies that $P|{}^e[n]$ is isomorphic to $U_{k,n}$ for each $e \in S$. Finally, the rank of $P$ is $3\cdot k$, thus \cref{prop:tensor} proves our claim.
\end{claimproof}

The theorem follows by combining \cref{cl:freest} and \cref{cl:Ntensor}.
\end{proof}

As a kind of converse to \cref{cor:1me}, we get that full modular extendability implies the existence of a tensor product with uniform matroids.

\begin{cor}\label{cor:tensor} 
    Let $M=(S,r)$ be a fully modular extendable matroid. Then, $M$ has a tensor product with any uniform matroid.
\end{cor}
\begin{proof} 
    By \cref{thm:equiv}, every connected component of $M$ has either rank $3$ or is skew-representable. Since any uniform matroid $U$ is representable over any large enough field, by \cref{lem:skew-large} and \cref{thm:uniform}, each of these components has a tensor product with $U$. Hence the statement follows by \cref{lem:direct_sum}.
\end{proof}

We conclude this section by a series of observations on the tensor products constructed in \cref{thm:uniform}.

\begin{rem} \label{rem:uniuni}
    In the special case $M=U_{3,m}$, the freest tensor product of $M$ and $U_{r,n}$ constructed in \cref{thm:uniform} coincides with the tensor matroid $T_{m,n}(3,r,p)$ studied by Brakensiek, Dhar, Gao, Gopi, and Larson~\cite{brakensiek2024rigidity}.  This follows by comparing the description of the bases in the proof above with the characterization in~\cite[Proposition~3.4]{brakensiek2024rigidity}.  
    Here $T_{m,n}(3,r,p)$ denotes the matroid represented by the Kronecker product $A\otimes B$, where $A\in\mathbb F^{3\times m}$ and $B\in\mathbb F^{r\times n}$ have generic columns over an infinite field $\mathbb F$ of characteristic $p$. Their results also show that this matroid is independent of the characteristic $p$. \cref{thm:uniform} additionally shows that it is freest among all matroid tensor products of $U_{3,m}$ and $U_{r,n}$, and extends the construction from $U_{3,m}$ to an arbitrary rank-3 matroid.
\end{rem}

In the proof of \cref{thm:uniform}, we used an extension $\tpm{M}$ of $M$ such that for any two disjoint lines $L_1$ and $L_2$ of $M$, their closures in $\tpm{M}$ meet in an element $e_{\{L_1, L_2\}}$ that is not contained in any other nontrivial line of $\tpm{M}$.
We next observe that the constructed tensor product $M\otimes U_{2,3}$ contains a minor isomorphic to $\tpm{M}|(S+e_{\{L_1, L_2\}})$.

\begin{lem} \label{lem:rank3ext}
    Let $M=(S,r)$ be a rank-3 matroid, let $N$ be the freest tensor product of $M$ and $U_{2,3}$, let $L_1$ and $L_2$ be disjoint lines of $M$, and let $p \notin S$.
    Then $N$ has a minor isomorphic to the unique matroid $M'=(S+p,r')$ that is an extension of $M$ and whose nontrivial lines containing $p$ are precisely $L_1+p$ and $L_2+p$.
\end{lem}
\begin{proof}
    We may assume by \cref{lem:simple} that $M$ is simple.
    Let $\{u,v,w\}$ be the ground set of $U_{2,3}$, and let $\{a, b\}\subseteq L_1$ and $\{c,d\}\subseteq L_2$ be two-element subsets.
    Following the proof of \cref{thm:1modular}, we get that $N'\coloneqq N/\{(a,v),(c,w),(d,w)\}|(S^u+(b,v))$ is a one-step extension of $M$ with respect to $L_1$ and $L_2$ under the natural identification of $S$ and $S^u$.
    This shows that $L_1^u+(b,v)$ and $L_2^u+(b,v)$ are lines of $N'$.
    Suppose that a third nontrivial line of $N'$ contains $(b,v)$.
    Then it has the form $L^u+(b,v)$ for some line $L$ of $M$ with $L\notin \{L_1, L_2\}$.
    As $L_1^u+(b,v)$ and $L_2^u+(b,v)$ are lines, it follows that $L$ is disjoint from both $L_1$ and $L_2$.
    As $L^u+(b,v)$ is a line of $N'$, it follows that $L^u\cup \{(a,v), (c,w), (d,w)\}$ spans $(b,v)$ in $N$.
    In particular, for a two-element subset $\{e,f\}\subseteq L$,   
    $\{e,f\}^u \cup \{a,b\}^v \cup \{c,d\}^w$ is not a basis of $N$.
    On the other hand such a set satisfies each condition on the members of $\cB$ in the proof of \cref{thm:uniform}: each of $\{e,f\}$, $\{a,b\}$, and $\{c,d\}$ is independent, the union of any two of these sets spans $M$, and $\tpm{L} \cap \tpm{L_1} \cap \tpm{L_2} = \emptyset$.
    This is a contradiction.
\end{proof}

We have seen in \cref{rem:uniuni} that the freest tensor product of a rank-3 uniform matroid and an arbitrary uniform matroid is representable.
We use \cref{lem:rank3ext} to show that this need not hold for an arbitrary representable rank-3 matroid: we give a rank-3 matroid that is representable over all sufficiently large fields, yet whose freest tensor product with $U_{2,3}$ is not skew-representable.

\begin{ex} \label{ex:nonskewfreest}
    Let $D$ denote the non-Desargues matroid on ground set $\{a_1,a_2,a_3,b_1,b_2,b_3,c_1,c_2,c_3,d\}$; its nontrivial lines are listed in \cref{sec:new} and it is illustrated in Figure~\ref{fig:non-desargues}.
    The matroid $D$ has rank 3, is not skew-representable, and is an excluded minor for $\bF$-representability for each field $\bF$ with $|\bF|\ge 7$; see \cite[page 657]{oxley2011matroid}.
    Its only nontrivial lines containing $c_1$ are $\{a_2,a_3,c_1\}$ and $\{b_2,b_3,c_1\}$. 
    Hence $\{a_2,a_3\}$ and $\{b_2,b_3\}$ are disjoint lines of $D-c_1$, and \cref{lem:rank3ext} shows that the freest tensor product $N\in (D- c_1)\otimes U_{2,3}$ has a minor isomorphic to $D$.
    In particular, $N$ is not skew-representable, while $D-c_1$ is representable over each field $\bF$ with $|\bF|\ge 7$.
\end{ex}

\section{Rank inequalities}
\label{sec:rank}

In this section, we focus on linear rank inequalities and introduce a new method for deriving them. As an illustration, we first present a simpler proof of Ingleton’s inequality in \cref{sec:ing}. Building on this approach, \cref{sec:two} derives two inequalities for polymatroid functions that admit tensor products with the rank functions of the Fano and non-Fano matroids, respectively, leading to characteristic-dependent linear rank inequalities for polymatroid functions that are skew-representable over skew fields of characteristic 2 or of characteristic different from 2. The main result of the section appears in \cref{sec:new}, where we derive a characteristic-independent linear rank inequality that does not follow from the common information property -- to the best of our knowledge, the first such inequality.

\subsection{Ingleton's inequality}
\label{sec:ing}

Ingleton~\cite{ingleton1971representation} proved a dimension inequality for subspaces that, in our terminology, implies that every skew-representable polymatroid function $\varphi\colon 2^S\to\bR$ satisfies the following inequality:
\begin{equation*}
\begin{aligned}
\varphi(A\cup B) + \varphi(A\cup C)+\varphi(A\cup D)+\varphi(B\cup C)+\varphi(B\cup D) \ge  \\
\varphi(A)+\varphi(B)+\varphi(A\cup B \cup C) + \varphi(A \cup B \cup D)+\varphi(C\cup D) 
\end{aligned}
\tag{\textsc{Ing}}
\label{ineq:ingleton}
\end{equation*}
The goal of this section is to present a new proof of the inequality using tensor products. While several proofs exist -- including two that rely on tensor product techniques~\cite{padro2025tensor, berczi2025matroid} -- our argument is more intuitive and offers a framework that leads to new linear rank inequalities. Our proof relies on an argument similar to that of Las Vergnas~\cite{las1981products}, who showed that the Vámos matroid $V_8$ and the uniform matroid $U_{2,3}$ do not admit a tensor product. Since every skew-representable polymatroid admits a tensor product with the rank function of the uniform matroid $U_{2,3}$ by Lemma~\ref{lem:representable}, it suffices to verify the following more general statement that was first shown in~\cite{berczi2025matroid}. The reason for including the proof is twofold: it is significantly simpler than the one given there, and the idea will serve as a starting point for deriving new linear rank inequalities in Sections~\ref{sec:two} and~\ref{sec:new}.

\begin{figure}[t!]
    \centering

    \begin{subfigure}{0.24\textwidth}
        \centering
        \begin{tikzpicture}[scale=0.7]
            \draw[step=1cm,black,thin] (0,0) grid (4,3);

    \node at (-0.5,2.5) {$u$};
    \node at (-0.5,1.5) {$v$};
    \node at (-0.5,0.5) {$w$};
    
    \node at (0.5,3.5) {$A$};
    \node at (1.5,3.5) {$B$};
    \node at (2.5,3.5) {$C$};
    \node at (3.5,3.5) {$D$};

    \fill[gray, opacity=0.4] (0,0) rectangle (1,3); 
    \fill[gray, opacity=0.4] (1,2) rectangle (2,3);
    \fill[gray, opacity=0.4] (2,1) rectangle (3,2);
    \fill[gray, opacity=0.4] (3,0) rectangle (4,1);

        \end{tikzpicture}
        \caption{Definition of $X$.}
        \label{fig:ingletona}
    \end{subfigure}
    \hfill
    \begin{subfigure}{0.24\textwidth}
        \centering
        \begin{tikzpicture}[scale=0.7]
            \draw[step=1cm,black,thin] (0,0) grid (4,3);

    \node at (-0.5,2.5) {$u$};
    \node at (-0.5,1.5) {$v$};
    \node at (-0.5,0.5) {$w$};
    
    \node at (0.5,3.5) {$A$};
    \node at (1.5,3.5) {$B$};
    \node at (2.5,3.5) {$C$};
    \node at (3.5,3.5) {$D$};

    \fill[gray, opacity=0.4] (0,2) rectangle (1,3); 
    \fill[gray, opacity=0.4] (1,0) rectangle (2,3);
    \fill[gray, opacity=0.4] (2,1) rectangle (3,2);
    \fill[gray, opacity=0.4] (3,0) rectangle (4,1);
    
        \end{tikzpicture}
        \caption{Definition of $Y$.}
        \label{fig:ingletonb}
    \end{subfigure}
    \hfill
    \begin{subfigure}{0.24\textwidth}
        \centering
        \begin{tikzpicture}[scale=0.7]
            \draw[step=1cm,black,thin] (0,0) grid (4,3);

    \node at (-0.5,2.5) {$u$};
    \node at (-0.5,1.5) {$v$};
    \node at (-0.5,0.5) {$w$};
    
    \node at (0.5,3.5) {$A$};
    \node at (1.5,3.5) {$B$};
    \node at (2.5,3.5) {$C$};
    \node at (3.5,3.5) {$D$};

    \fill[gray, opacity=0.4] (0,2) rectangle (1,3); 
    \fill[gray, opacity=0.4] (1,2) rectangle (2,3);
    \fill[gray, opacity=0.4] (2,1) rectangle (3,2);
    \fill[gray, opacity=0.4] (3,0) rectangle (4,1);
    
        \end{tikzpicture}
        \caption{The set $W\subseteq X\cap Y$.}
        \label{fig:ingletonc}
    \end{subfigure}
    \hfill
    \begin{subfigure}{0.24\textwidth}
        \centering
        \begin{tikzpicture}[scale=0.7]
            \draw[step=1cm,black,thin] (0,0) grid (4,3);

    \node at (-0.5,2.5) {$u$};
    \node at (-0.5,1.5) {$v$};
    \node at (-0.5,0.5) {$w$};
    
    \node at (0.5,3.5) {$A$};
    \node at (1.5,3.5) {$B$};
    \node at (2.5,3.5) {$C$};
    \node at (3.5,3.5) {$D$};

    \fill[gray, opacity=0.4] (0,0) rectangle (1,3); 
    \fill[gray, opacity=0.4] (1,0) rectangle (2,3);
    \fill[gray, opacity=0.4] (2,1) rectangle (3,2);
    \fill[gray, opacity=0.4] (3,0) rectangle (4,1);
    
        \end{tikzpicture}
        \caption{Union of $X$ and $Y$.}
        \label{fig:ingletond}
    \end{subfigure}
    \caption{Illustration of the proof of \cref{thm:ingleton}. Note that, in general, the sets $A$, $B$, $C$, and $D$ are not necessarily disjoint.}
    \label{fig:ingleton}
\end{figure}

\begin{thm}[Bérczi, Gehér, Imolay, Lovász, Maga, and Schwarcz] \label{thm:ingleton}
Let $\varphi_1\colon 2^{S_1}\to \bR$ denote the rank function of the uniform matroid $U_{2,3}$, and let $\varphi_2\colon 2^{S_2} \to \bR$ be a polymatroid function. If $\varphi_1$ and $\varphi_2$ have a tensor product, then  $\varphi_2$ satisfies Ingleton's inequality \eqref{ineq:ingleton} for all sets $A, B, C, D \subseteq S_2$.
\end{thm}
\begin{proof}
Let $S_1=\{u,v,w\}$ denote the elements of $U_{2,3}$ and let $\varphi\colon 2^{S_1\times S_2}\to\bR$ be a polymatroid tensor product of $\varphi_1$ and $\varphi_2$. Consider four sets $A,B,C,D \subseteq S_2$. Since we can restrict $\varphi_2$ to the union of these sets, we may assume without loss of generality that $S_2=A\cup B\cup C\cup D$. We define
\[X \coloneqq (S_1 \times A) \cup \prd{u}{B} \cup \prd{v}{C} \cup \prd{w}{D}\] 
and 
\[Y \coloneqq \prd{u}{A} \cup (S_1 \times B) \cup \prd{v}{C} \cup \prd{w}{D};\]
see Figures~\ref{fig:ingletona} and~\ref{fig:ingletonb} for illustrations of these sets. We aim to bound the values of $\varphi$ on $X$, $Y$, $X \cap Y$, and $X \cup Y$ so that, together with the submodular inequality $\varphi(X) + \varphi(Y) \geq \varphi(X \cap Y) + \varphi(X \cup Y)$, these bounds yield Ingleton's inequality. We begin with an upper bound on the values $\varphi(X)$ and $\varphi(Y)$.

\begin{cla} \label{cl:xrang}
    $\varphi(X) \leq \varphi_2(A \cup B)+\varphi_2(A \cup C)+\varphi_2(A \cup D)-\varphi_2(A)$ and $\varphi(Y) \leq \varphi_2(A \cup B)+\varphi_2(B \cup C)+\varphi_2(B \cup D)-\varphi_2(B)$.
\end{cla}
\begin{claimproof}
    We prove the bound for $\varphi(X)$; the argument for $\varphi(Y)$ is analogous.
    Since $\varphi_1(\{u,v,w\})=\varphi_1(\{u,v\})$ and $\varphi_1(\{u,v\})=\varphi_1(u)+\varphi_1(v)$, \cref{lem:gen} and \cref{lem:direct} show that \[\varphi((S_1 \times A)\cup \prd{u}{B} \cup \prd{v}{C}) = \varphi(\prd{u}{(A\cup B)}\cup \prd{v}{(A\cup C)}) = \varphi_2(A\cup B)+\varphi_2(A\cup C).\]
    Using \cref{lem:gen}, 
    we have 
    \[\varphi((S_1 \times A) \cup \prd{u}{B} \cup \prd{v}{C} \cup \prd{w}{D}) \le \varphi((S_1 \times A)\cup \prd{u}{B} \cup \prd{v}{C}) + \varphi_2(A\cup D)-\varphi_2(A).\]  
    Combining the two (in)equalities, the claim follows.
\end{claimproof}

Now we lower bound $\varphi(X\cap Y)$ and $\varphi(X\cup Y)$; see Figures~\ref{fig:ingletonc} and~\ref{fig:ingletond} for illustrations of these sets.

\begin{cla} \label{cl:xcapy}
    $\varphi(X \cap Y) \geq \varphi_2(A \cup B)+\varphi_2(C \cup D)$.
\end{cla}
\begin{claimproof} 
    Let $W\coloneqq \prd{u}{(A\cup B)} \cup \prd{v}{C} \cup \prd{w}{D}$. Observe that $X \cap Y=\prd{u}{(A\cup B)}\cup \prd{v}{(C\cup (A\cap B))}\cup \prd{w}{(D\cup (A\cap B))} \supseteq W$, hence it suffices to prove that $\varphi(W)\ge \varphi_2(A\cup B)+\varphi_2(C\cup D)$. 
    Let $Z_1 \coloneqq \prd{u}{S_2}\cup \prd{v}{C} \cup \prd{w}{D} =  W \cup \prd{u}{(C\cup D)}$.  By \cref{lem:gen}, we have
    \begin{align} \label{eq:XcapY_1}
        \varphi(Z_1)\le \varphi(W)+\varphi_2(S_2)-\varphi_2(A\cup B).
    \end{align}
    Let $Z_2 \coloneqq \prd{u}{S_2} \cup (S_1 \times (C\cup D)) = Z_1 \cup \prd{v}{D}\cup \prd{w}{C}$. Since $\varphi_1(S_1) = \varphi_1(\{u,v\}) = \varphi_1(\{u,w\})$, by \cref{lem:gen} we have $\varphi(Z_2)=\varphi(Z_1)$. Let $Z_3 \coloneqq Z_2 \cup \prd{v}{(A\cup B)} = (\{u,v\}\times S_2) \cup (S_1 \times (C\cup D))$. By \cref{lem:gen}, we have  
    \begin{align} \label{eq:XcapY_2}
        \varphi(Z_3) \le \varphi(Z_2)+\varphi_2(S_2)-\varphi_2(C\cup D).
    \end{align}
    Finally, as $\varphi_1(\{u,v\})=\varphi_1(S_1)$, by \cref{lem:gen} we have
    \begin{align} \label{eq:XcapY_3}
        2 \cdot \varphi_2(S_2) = \varphi(S_1 \times S_2) = \varphi(Z_3). 
    \end{align}
    Then, the claim follows after cancellations by combining \eqref{eq:XcapY_1}, $\varphi(Z_2)=\varphi(Z_1)$, \eqref{eq:XcapY_2}, and \eqref{eq:XcapY_3}. 
\end{claimproof}

\begin{cla} \label{cl:xcupy}
    $\varphi(X \cup Y) \geq \varphi_2(A \cup B \cup C)+\varphi_2(A \cup B \cup D)$.
\end{cla}
\begin{claimproof}
     Note that $X \cup Y = \prd{u}{(A\cup B)} \cup \prd{v}{(A\cup B\cup C)} \cup \prd{w}{(A\cup B\cup D)}$.
    By \cref{lem:gen}, we have \[\varphi(X\cup Y \cup \prd{v}{D} \cup \prd{w}{C})\le \varphi(X\cup Y)+(\varphi_2(S_2)-\varphi_2(A\cup B \cup C)) + (\varphi_2(S_2)-\varphi_2(A\cup B \cup D)).\]
    Since $\varphi_1(S_1)=\varphi_1(\{v,w\})$, by \cref{lem:gen} we have \[2\cdot\varphi_2(S_2)=\varphi(S_1\times S_2)=\varphi(X\cup Y \cup \prd{v}{D} \cup \prd{w}{C}).\]
    The claim follows by combining the two (in)equalities.
\end{claimproof}

Combining Claims~\ref{cl:xrang}, \ref{cl:xcapy}, \ref{cl:xcupy}, and the submodular inequality $\varphi(X) + \varphi(Y) \geq \varphi(X \cap Y)+\varphi(X \cup Y)$, we get exactly \eqref{ineq:ingleton}.
\end{proof}

\subsection{Fields of characteristic two}
\label{sec:two}

We define the {\it Fano matroid} $F_7$ on ground set $S=\{s_1,\dots, s_7\}$ as the simple rank-3 matroid with set of lines \[\{\{s_1,s_2,s_6\}, \{s_1,s_3,s_5\}, \{s_1,s_4,s_7\}, \{s_2,s_3,s_4\}, \{s_2,s_5,s_7\}, \{s_3,s_6,s_7\}, \{s_4,s_5,s_6\}\},\]
see also Figure~\ref{fig:fano}. Similarly, we define the {\it non-Fano matroid} $F_7^-$ on ground set $T=\{t_1,\dots, t_7\}$ whose nontrivial lines are 
 \[\big\{\{t_1,t_2,t_6\}, \{t_1,t_3,t_5\}, \{t_1,t_4,t_7\}, \{t_2,t_3,t_4\}, \{t_2,t_5,t_7\}, \{t_3,t_6,t_7\}\big\}.\]
It is known that $F_7$ is representable over a field precisely when the field has characteristic 2, while $F_7^-$ is representable over a field precisely when the field has characteristic different from 2~\cite[Proposition~6.4.8]{oxley2011matroid}. As the analogous result also holds for skew fields, 
\cref{thm:char} shows that $T_k(F_7, F_7^-) = \emptyset$ and $T_\ell(F_7^-, F_7) = \emptyset$ for some positive integers $k$ and $\ell$. We strengthen this by showing $T_1(F_7, F_7^-) = T_1(F_7^-, F_7)=\emptyset$, that is, $F_7$ and $F_7^-$ do not admit a tensor product. More generally, we use the idea of our proof to derive linear inequalities satisfied by any polymatroid function having a tensor product with the rank function of $F_7$ or $F_7^-$, respectively.

\begin{thm} \label{thm:Fano_non_Fano}
Let $\varphi_1\colon 2^{S_1}\to \bR$ and $\varphi_2 \colon 2^{S_2} \to \bR$ be polymatroid functions admitting a tensor product.
\begin{enumerate}[label={\textup{(\alph*)}}]\itemsep0em
    \item \label{it:Fano} If $\varphi_1$ is the rank function of the Fano matroid $F_7$, then 
    \begin{align} \label{eq:Fano}
        & \varphi_2(B_4) + \varphi_2(B_6) + \varphi_2(B_{234}) + \varphi_2(B_{135}) + \varphi_2(B_{126}) + \varphi_2(B_{367}) + \varphi_2(B_{147}) + \varphi_2(B_{257}) \nonumber \\
        \geq\ & \varphi_2(B_{36}) + \varphi_2(B_{14}) + \varphi_2(B_{27}) + \varphi_2(B_{456}) + \varphi_2(B_{12347}) + \varphi_2(B_{123567})
    \end{align}
    holds for any subsets $B_1,\dots, B_7 \subseteq S_2$, where we define $B_{i_1\dots i_k} \coloneqq \bigcup_{j=1}^k B_{i_j}$ for $\{i_1,\dots, i_k\}\subseteq [7]$. Furthermore, this inequality is violated if $\varphi_2$ is the rank function of the non-Fano matroid $F_7^-$ and $B_i = \{t_i\}$ for $i \in [7]$.

    \item \label{it:non_Fano} If $\varphi_2$ is the rank function of the non-Fano matroid $F_7^-$, then      
    \begin{align} \label{eq:non_Fano}
        & \varphi_1(A_4) + \varphi_1(A_6) + \varphi_1(A_{126}) + \varphi_1(A_{135}) + \varphi_1(A_{234}) + \varphi_1(A_{367}) + \varphi_1(A_{147}) + \varphi_1(A_{456}) + \varphi_1(A_{257}) \nonumber \\
        \geq\ & \varphi_1(A_{36}) + \varphi_1(A_{14}) + \varphi_1(A_{46}) + \varphi_1(A_{25}) + \varphi_1(A_{4567}) + \varphi_1(A_{12347}) + \varphi_1(A_{123567})
    \end{align}
    holds for any subsets $A_1,\dots, A_7 \subseteq S_1$, where  we define $A_{i_1\dots i_k} \coloneqq \bigcup_{j=1}^k A_{i_j}$ for $\{i_1,\dots, i_k\}\subseteq [7]$.
    Furthermore, this inequality is violated if $\varphi_1$ is the rank function of the Fano matroid $F_7$ and $A_i=\{s_i\}$ for $i \in [7]$.
\end{enumerate}
\end{thm}

\begin{figure}[t!]
    \centering
    \begin{minipage}[t]{0.45\textwidth}
        \centering
        \begin{tikzpicture}[scale=1.3]
            \fill (0,2) circle (3.3pt);
            \node at (0.23,2.23) {$s_1$};
            \fill (-1.732,-1) circle (3.3pt);
            \node at (-1.962,-0.77) {$s_2$};
            \fill (1.732,-1) circle (3.3pt);
            \node at (1.962,-0.77) {$s_3$};
            \fill (0,0) circle (3.3pt);
            \node at (0.16,0.3) {$s_7$};
            \fill (0,-1) circle (3.3pt);
            \node at (0.2,-0.75) {$s_4$};
            \fill (-1.732/2,1/2) circle (3.3pt);
            \node at (-1.732/2-0.23,1/2+0.23) {$s_6$};
            \fill (1.732/2,1/2) circle (3.3pt);
            \node at (1.732/2+0.23,1/2+0.23) {$s_5$};
        
            \draw (0,2) -- (-1.732,-1) -- (1.732,-1) -- (0,2);
            \draw (0,2) -- (0,-1);
            \draw (-1.732,-1) -- (1.732/2,1/2);
            \draw (1.732,-1) -- (-1.732/2,1/2);
            \draw (-1.732/2,1/2) -- (0,0) -- (1.732/2,1/2);
            \draw (0,-1) -- (0,0);
            \draw[dashed] (0,0) circle [radius=1];        
        \end{tikzpicture}
        \subcaption{Point-line representation of the Fano and non-Fano matroids. Note that $\{s_4,s_5,s_6\}$ is a circuit in $F_7$ while it is a basis in $F_7^-$.} \label{fig:fano}
    \end{minipage}
    \hfill
    \begin{minipage}[t]{0.45\textwidth}
        \centering
        \begin{tikzpicture}[scale=0.7]
            \draw[step=1cm,black,thin] (0,0) grid (7,7); 

            \foreach \i in {1,...,7} {
                \node at (-0.55,{7.5-\i}) {$A_{\i}$};
                \node at ({\i-0.5},7.45) {$B_{\i}$};
            }
        
            \fill[black] (0,6) rectangle (2,7);
            \fill[black] (0,5) rectangle (1,6);
            \fill[black] (1,4) rectangle (3,5);
            \fill[black] (2,5) rectangle (3,6);
            \fill[black] (3,1) rectangle (4,2);
            \fill[black] (5,3) rectangle (6,4);

            \node at (0.5,1.5) {\small $X_{2,1}$};
            \node at (1.5,2.5) {\small $X_{2,2}$};
            \node at (2.5,3.5) {\small $X_{2,3}$};
            \node at (3.5,4.5) {\small $X_{1,3}$};
            \node at (4.5,5.5) {\small $X_{1,2}$};
            \node at (5.5,6.5) {\small $X_{1,1}$};
            \node at (3.5,0.5) {\small $X_{2,4}$};
            \node at (6.5,3.5) {\small $X_{3,1}$};
            \node at (5.5,0.5) {\small $X_{2,5}$};
            \node at (6.5,1.5) {\small $X_{3,2}$};
            \node at (6.5,2.5) {\small $X_{4,1}$};
            \node at (4.5,2.5) {\small $X_{5,1}$};
            \node at (4.5,0.5) {\small $X_{6,1}$};
            \node at (0.5,0.5) {\small $X_{7,1}$};
            \node at (1.5,0.5) {\small $X_{7,2}$};
            \node at (2.5,0.5) {\small $X_{7,3}$};
            \node at (6.5,0.5) {\small $X_{7,4}$};

            \fill[gray, opacity=0.4] (0,1) rectangle (1,2); 
            \fill[gray, opacity=0.4] (1,2) rectangle (2,3); 
            \fill[gray, opacity=0.4] (2,3) rectangle (3,4); 
            \fill[gray, opacity=0.4] (3,4) rectangle (4,5); 
            \fill[gray, opacity=0.4] (4,5) rectangle (5,6); 
            \fill[gray, opacity=0.4] (5,6) rectangle (6,7); 
            \fill[gray, opacity=0.4] (3,0) rectangle (4,1); 
            \fill[gray, opacity=0.4] (4,0) rectangle (5,1); 
            \fill[gray, opacity=0.4] (5,0) rectangle (6,1);
            \fill[gray, opacity=0.4] (6,1) rectangle (7,2);
            \fill[gray, opacity=0.4] (6,2) rectangle (7,3); 
            \fill[gray, opacity=0.4] (6,3) rectangle (7,4); 
            \fill[gray, opacity=0.4] (4,2) rectangle (5,3); 
        
            \fill[gray, opacity=0.4] (0,0) rectangle (3,1); 
            \fill[gray, opacity=0.4] (6,0) rectangle (7,1); 
        \end{tikzpicture}
        \subcaption{Illustration of $X$ (black cells) and $X_1,\dots, X_7$. $X_i \setminus X_{i-1} = X_{i,1}\cup \dots \cup X_{i,j_i}$ for $i \in [7]$, where $X_0 \coloneqq X$.} \label{fig:Fano_non_Fano_X}
    \end{minipage}
    \vspace{0.5cm}
    
    \begin{minipage}[t]{0.45\textwidth}
        \centering
    \begin{tikzpicture}[scale=0.7]
        \draw[step=1cm,black,thin] (0,0) grid (7,7); 

        \foreach \i in {1,...,7} {
            \node at (-0.55,{7.5-\i}) {$A_{\i}$};
            \node at ({\i-0.5},7.45) {$B_{\i}$};
        }
    
        \fill[black] (0,6) rectangle (2,7);
        \fill[black] (0,5) rectangle (1,6);
        \fill[black] (1,4) rectangle (3,5);
        \fill[black] (2,5) rectangle (3,6);
        \fill[black] (3,1) rectangle (4,2);
        \fill[black] (5,3) rectangle (6,4);

        \node at (0.5,4.5) {\scriptsize $X_{8,1}$};
        \node at (0.5,3.5) {\scriptsize $X_{8,2}$};
        \node at (0.5,2.5) {\scriptsize $X_{8,3}$};
        \node at (1.5,5.5) {\scriptsize $X_{8,4}$};
        \node at (1.5,3.5) {\scriptsize $X_{8,5}$};
        \node at (1.5,1.5) {\scriptsize $X_{8,6}$};
        \node at (2.5,6.5) {\scriptsize $X_{8,7}$};
        \node at (2.5,2.5) {\scriptsize $X_{8,8}$};
        \node at (2.5,1.5) {\scriptsize $X_{8,9}$};
        \node at (6.5,6.5) {\scriptsize $X_{8,10}$};
        \node at (6.5,5.5) {\scriptsize $X_{8,11}$};
        \node at (6.5,4.5) {\scriptsize $X_{8,12}$};
        \node at (4.5,1.5) {\scriptsize $X_{9,1}$};
        \node at (5.5,1.5) {\scriptsize $X_{9,2}$};
        \node at (4.5,6.5) {\scriptsize $X_{10,1}$};
        \node at (4.5,4.5) {\scriptsize $X_{10,2}$};
        \node at (4.5,3.5) {\scriptsize $X_{10,3}$};
        \node at (5.5,5.5) {\scriptsize $X_{10,4}$};
        \node at (5.5,4.5) {\scriptsize $X_{10,5}$};
        \node at (5.5,2.5) {\scriptsize $X_{10,6}$};
        \node at (3.5,2.5) {\scriptsize $X_{11,1}$};
        \node at (3.5,6.5) {\scriptsize $X_{12,1}$};
        \node at (3.5,5.5) {\scriptsize $X_{12,2}$};
        \node at (3.5,3.5) {\scriptsize $X_{12,3}$};

        \fill[black] (0,1) rectangle (1,2); 
        \fill[black] (1,2) rectangle (2,3); 
        \fill[black] (2,3) rectangle (3,4); 
        \fill[black] (3,4) rectangle (4,5); 
        \fill[black] (4,5) rectangle (5,6); 
        \fill[black] (5,6) rectangle (6,7); 
        \fill[black] (3,0) rectangle (4,1); 
        \fill[black] (4,0) rectangle (5,1); 
        \fill[black] (5,0) rectangle (6,1);
        \fill[black] (6,1) rectangle (7,2);
        \fill[black] (6,2) rectangle (7,3); 
        \fill[black] (6,3) rectangle (7,4); 
        \fill[black] (4,2) rectangle (5,3); 
    
        \fill[black] (0,0) rectangle (3,1); 
        \fill[black] (6,0) rectangle (7,1); 
    \end{tikzpicture}
    \subcaption{Illustration of $X_7$ (black cells) and $X_8,\dots, X_{12}$. $X_i \setminus X_{i-1} = X_{i,1}\cup \dots \cup X_{i,j_i}$ for $i \in \{8,\dots, 12\}$. \phantom{where $X'_7\coloneqq X_7$.} } \label{fig:Fano_proof}
    \end{minipage} \hfill
    \begin{minipage}[t]{0.45\textwidth}
        \centering
    \begin{tikzpicture}[scale=0.7]
        \draw[step=1cm,black,thin] (0,0) grid (7,7); 

        \foreach \i in {1,...,7} {
            \node at (-0.55,{7.5-\i}) {$A_{\i}$};
            \node at ({\i-0.5},7.45) {$B_{\i}$};
        }
    
        \fill[black] (0,6) rectangle (2,7);
        \fill[black] (0,5) rectangle (1,6);
        \fill[black] (1,4) rectangle (3,5);
        \fill[black] (2,5) rectangle (3,6);
        \fill[black] (3,1) rectangle (4,2);
        \fill[black] (5,3) rectangle (6,4);
    
        \node at (6.5,6.5) {\scriptsize $X'_{8,1}$};
        \node at (6.5,5.5) {\scriptsize $X'_{8,2}$};
        \node at (6.5,4.5) {\scriptsize $X'_{8,3}$};
        \node at (2.5,6.5) {\scriptsize $X'_{9,1}$};
        \node at (3.5,6.5) {\scriptsize $X'_{9,2}$};
        \node at (4.5,6.5) {\scriptsize $X'_{9,3}$};
        \node at (1.5,5.5) {\scriptsize $X'_{9,4}$};
        \node at (3.5,5.5) {\scriptsize $X'_{9,5}$};
        \node at (5.5,5.5) {\scriptsize $X'_{9,6}$};
        \node at (0.5,4.5) {\scriptsize $X'_{9,7}$};
        \node at (4.5,4.5) {\scriptsize $X'_{9,8}$};
        \node at (5.5,4.5) {\scriptsize $X'_{9,9}$};
        \node at (5.5,2.5) {\scriptsize $X'_{10,1}$};
        \node at (5.5,1.5) {\scriptsize $X'_{10,2}$};
        \node at (0.5,2.5) {\scriptsize $X'_{11,1}$};
        \node at (2.5,2.5) {\scriptsize $X'_{11,2}$};
        \node at (3.5,2.5) {\scriptsize $X'_{11,3}$};
        \node at (1.5,1.5) {\scriptsize $X'_{11,4}$};
        \node at (2.5,1.5) {\scriptsize $X'_{11,5}$};
        \node at (4.5,1.5) {\scriptsize $X'_{11,6}$};
        \node at (4.5,3.5) {\scriptsize $X'_{12,1}$};
        \node at (0.5,3.5) {\scriptsize $X'_{13,1}$};
        \node at (1.5,3.5) {\scriptsize $X'_{13,2}$};     
        \node at (3.5,3.5) {\scriptsize $X'_{13,3}$};

        \fill[black] (0,1) rectangle (1,2); 
        \fill[black] (1,2) rectangle (2,3); 
        \fill[black] (2,3) rectangle (3,4); 
        \fill[black] (3,4) rectangle (4,5); 
        \fill[black] (4,5) rectangle (5,6); 
        \fill[black] (5,6) rectangle (6,7); 
        \fill[black] (3,0) rectangle (4,1); 
        \fill[black] (4,0) rectangle (5,1); 
        \fill[black] (5,0) rectangle (6,1);
        \fill[black] (6,1) rectangle (7,2);
        \fill[black] (6,2) rectangle (7,3); 
        \fill[black] (6,3) rectangle (7,4); 
        \fill[black] (4,2) rectangle (5,3); 
    
        \fill[black] (0,0) rectangle (3,1); 
        \fill[black] (6,0) rectangle (7,1); 
    \end{tikzpicture}
    \subcaption{Illustration of $X_7$ (black cells) and $X'_8,\dots, X'_{13}$. $X'_i \setminus X'_{i-1} = X'_{i,1}\cup \dots \cup X'_{i,j_i}$ for $i \in \{8,\dots, 13\}$, where $X'_7\coloneqq X_7$.} \label{fig:non_Fano_proof}
    \end{minipage}
    \caption{Illustration of the proof of \cref{thm:Fano_non_Fano}. Note that, in general, the sets $A_i$ are not necessarily disjoint, and neither are the sets $B_i$.} 
    \label{fig:fano7grid}  
\end{figure}

\begin{proof}
    First, we prove some general inequalities for any polymatroid functions $\varphi_1\colon 2^{S_1}\to \bR$ and $\varphi_2\colon 2^{S_2}\to \bR$ admitting a tensor product $\varphi\colon 2^{S_1\times S_2} \to \bR$. Let $A_1,\dots, A_7 \subseteq S_1$ and $B_1,\dots, B_7 \subseteq S_2$ and define $A_{i_1\dots i_k} \coloneqq \bigcup_{j=1}^k A_{i_j}$ and $B_{i_1\dots i_k} \coloneqq \bigcup_{j=1}^k B_{i_j}$. Define 
    \begin{align*}
        X & \coloneqq  (A_1 \times B_{12}) \cup (A_2 \times B_{13}) \cup (A_3 \times B_{23})\cup (A_4 \times B_6) \cup (A_6 \times B_4) \\
        & = (A_{12} \times B_1) \cup (A_{13} \times B_2) \cup  (A_{23} \times B_3) \cup (A_6 \times B_4) \cup (A_4 \times B_6).
    \end{align*}
    See also Figure~\ref{fig:Fano_non_Fano_X} for an illustration.
    We will give a lower bound on $\varphi(X)$ by providing an upper bound on $\varphi(S_1 \times S_2)$ using $\varphi(X)$.
    Let $X_1 \coloneqq X \cup (A_1 \times B_6)\cup (A_2 \times B_5) \cup (A_3 \times B_4)$.
    Applying \cref{lem:gen} three times,
    we obtain 
    \begin{align} \label{eq:X1}
        \varphi(X_1) \le \varphi(X) & + \varphi_1(A_1)\cdot (\varphi_2(B_{126})-\varphi_2(B_{12})) \nonumber \\ & + \varphi_1(A_2)\cdot (\varphi_2(B_{135})-\varphi_2(B_{13})) \\ & + \varphi_1(A_3)\cdot (\varphi_2(B_{234})-\varphi_2(B_{23})). \nonumber
    \end{align}
    Let $X_2 \coloneqq X_1 \cup (A_6 \times B_1)\cup (A_5 \times B_2) \cup (A_4 \times B_3) \cup (A_7 \times B_4) \cup (A_7 \times B_6)$.
    Applying \cref{lem:gen} five times, we obtain
    \begin{align} \label{eq:X2}
        \varphi(X_2) \le \varphi(X_1) & + (\varphi_1(A_{126})-\varphi_1(A_{12}))\cdot \varphi_2(B_1) \nonumber \\
        & + (\varphi_1(A_{135})-\varphi_1(A_{13})) \cdot \varphi_2(B_2) \nonumber \\
        & +(\varphi_1(A_{234})-\varphi_1(A_{23})) \cdot \varphi_2(B_3)  \\
        & + (\varphi_1(A_{367})-\varphi_1(A_{36})) \cdot \varphi_2(B_4)\nonumber \\
        & + (\varphi_1(A_{147})-\varphi_1(A_{14})) \cdot \varphi_2(B_6). \nonumber  
    \end{align}
    Let $X_3 \coloneqq X_2 \cup (A_4 \times B_7) \cup (A_6 \times B_7)$.
    Applying \cref{lem:gen} twice, we obtain
    \begin{align} \label{eq:X3}
        \varphi(X_3) \le \varphi(X_2) + \varphi_1(A_4)\cdot (\varphi_2(B_{367})-\varphi_2(B_{36})) + \varphi_1(A_6)\cdot (\varphi_2(B_{147})-\varphi_2(B_{14})).
    \end{align}
    Let $X_4 \coloneqq X_3 \cup (A_5\times B_7)$.
    Applying \cref{lem:gen}, we obtain
    \begin{align} \label{eq:X4}
        \varphi(X_4) \le \varphi(X_3) + (\varphi_1(A_{456})-\varphi_1(A_{46}))\cdot \varphi_2(B_7).
    \end{align}
    Let $X_5 \coloneqq X_4 \cup (A_5 \times B_5)$. Applying \cref{lem:gen}, we obtain
    \begin{align} \label{eq:X5}
        \varphi(X_5)\le \varphi(X_4) + \varphi_1(A_5) \cdot (\varphi_2(B_{257})-\varphi_2(B_{27})).
    \end{align}
    Let $X_6 \coloneqq X_5 \cup  (A_7 \times B_5)$. Applying  \cref{lem:gen}, we obtain
    \begin{align} \label{eq:X6}
        \varphi(X_6) \le \varphi(X_5) + (\varphi_1(A_{257})-\varphi_1(A_{25}))\cdot \varphi_2(B_5).
    \end{align}
    Let $X_7 \coloneqq  X_6 \cup (A_7 \times B_{1237}) = X_6 \cup (A_7 \times B_{1234567})$. Applying \cref{lem:gen}, we obtain
    \begin{align} \label{eq:X7}
        \varphi(X_7) \le \varphi(X_6) + \varphi_1(A_7)\cdot (\varphi_2(B_{1234567})-\varphi_2(B_{456})).
    \end{align}

From now on, our proof of \ref{it:Fano} and \ref{it:non_Fano} will be different; see Figures~\ref{fig:Fano_proof} and \ref{fig:non_Fano_proof} for illustrations. To prove \ref{it:Fano}, assume that $\varphi_1$ is the rank function of the Fano matroid and $A_i = \{s_i\}$ for $i \in [7]$. Observe that in this case inequalities \eqref{eq:X2}, \eqref{eq:X4}, and \eqref{eq:X6} simplify to $\varphi(X_{2i}) \le \varphi(X_{2i-1})$ for $i \in [3]$, implying $\varphi(X_{2i}) = \varphi(X_{2i-1})$ for $i \in [3]$.
Let $X_8 \coloneqq X_7 \cup (A_{345}\times B_1)\cup(A_{246}\times B_2)\cup(A_{156}\times B_{3})\cup(A_{123}\times B_7)$. Applying \cref{lem:gen} and using $\varphi_1(A_{1267}) = \varphi_1(A_{1357}) = \varphi_1(A_{2347}) = \varphi_1(A_{4567}) = \varphi_1(A_{1234567})$, we obtain $\varphi(X_8)=\varphi(X_7)$.
Let $X_9 \coloneqq X_8 \cup (A_6 \times B_{56})$. Applying \cref{lem:gen} and using $\varphi_1(A_6)=1$, we obtain
\begin{align} \label{eq:X9}
    \varphi(X_9) \le \varphi(X_8)+\varphi_2(B_{1234567})-\varphi_2(B_{12347}).
\end{align}
Let $X_{10} \coloneqq X_9 \cup (A_{134}\times B_5)\cup (A_{235}\times B_6)$. Applying \cref{lem:gen} and using $\varphi_1(A_{2567})=\varphi_1(A_{1467})=\varphi_1(A_{1234567})$, we obtain $\varphi(X_{10})=\varphi(X_9)$. Let $X_{11}\coloneqq X_{10} \cup (A_5 \times B_4)$. Applying \cref{lem:gen}, we obtain
\begin{align} \label{eq:X11}
    \varphi(X_{11}) \le \varphi(X_{10}) + \varphi_2(B_{1234567})-\varphi_2(B_{123567}).
\end{align}
Finally, let $X_{12} \coloneqq X_{11} \cup (A_{124}\times B_4) = A_{1234567}\times B_{1234567}$ and observe that $\varphi_1(A_{3567})=\varphi_1(A_{1234567})$. Thus, \cref{lem:gen} shows that
\begin{align} \label{eq:Fano_whole}
    3\cdot\varphi_2(B_{1234567}) = \varphi(X_{12}) = \varphi(X_{11}).
\end{align}
Summing up the (in)equalities \eqref{eq:X1}, \eqref{eq:X3}, \eqref{eq:X5}, \eqref{eq:X7}, \eqref{eq:X9}, \eqref{eq:X11}, and \eqref{eq:Fano_whole} and using that $\varphi(X_{2i})=\varphi(X_{2i-1})$ for $i \in [5]$ and $\varphi_1(A_j)=1$ for $j \in [7]$, we obtain
\begin{align} \label{eq:Fano_final}
    3\cdot\varphi_2(B_{1234567})\le \varphi(X) & + (\varphi_2(B_{126})-\varphi_2(B_{12}) + \varphi_2(B_{135})-\varphi_2(B_{13})+\varphi_2(B_{234})-\varphi_2(B_{23})) \nonumber \\ & + (\varphi_2(B_{367})-\varphi_2(B_{36})+\varphi_2(B_{147})-\varphi_2(B_{14})) + (\varphi_2(B_{257})-\varphi_2(B_{27})) \\
    & + (\varphi_2(B_{1234567})-\varphi_2(B_{456})) + (\varphi_2(B_{1234567})-\varphi_2(B_{12347})) \nonumber \\& +(\varphi_2(B_{1234567})-\varphi_2(B_{123567})) \nonumber
\end{align}
after cancellations.
Using the submodularity of $\varphi$, it is clear that
\begin{align} \label{eq:FanoX}
    \varphi(X) \le \varphi_2(B_{12})+\varphi_2(B_{13})+\varphi_2(B_{23})+\varphi_2(B_6)+\varphi_2(B_4).
\end{align}
Combining \eqref{eq:Fano_final} and \eqref{eq:FanoX}, the inequality \eqref{eq:Fano} follows after cancellations. If $\varphi_2$ is the rank function of the non-Fano matroid and $B_i = \{t_i\}$ for $i \in [7]$, \eqref{eq:Fano_final} is equivalent to $9 \le \varphi(X)$, while \eqref{eq:FanoX} evaluates to $\varphi(X) \le 8$. These together show that \eqref{eq:Fano} is violated in this case. This finishes the proof of \ref{it:Fano}.

We turn to the proof of \ref{it:non_Fano}, which is similar to that of \ref{it:Fano}. 
We assume that $\varphi_2$ is the rank function of the non-Fano matroid and $B_i=\{t_i\}$ for $i \in [7]$. 
In this case, inequalities \eqref{eq:X1}, \eqref{eq:X3}, \eqref{eq:X5}, and \eqref{eq:X7} imply $\varphi(X_1)=\varphi(X)$ and $\varphi(X_{2i+1})=\varphi(X_{2i})$ for $i \in [3]$. Let $X'_8\coloneqq X_7\cup (A_{123}\times B_7)$. Applying \cref{lem:gen}, we obtain
\begin{align} \label{eq:Xp8}
    \varphi(X'_8) \le \varphi(X_7)+\varphi_1(A_{1234567})-\varphi_1(A_{4567}).
\end{align}
Let $X'_9 \coloneqq X'_8 \cup (A_1\times B_{345})\cup (A_2 \times B_{246})\cup (A_3 \times B_{156})$; then \cref{lem:gen} shows $\varphi(X'_9)=\varphi(X'_8)$. Let $X'_{10} \coloneqq X'_9 \cup (A_{56}\times B_6)$. Applying \cref{lem:gen}, we obtain
\begin{align} \label{eq:Xp10}
    \varphi(X'_{10}) \le \varphi(X'_9) + \varphi_1(A_{1234567})-\varphi_1(A_{12347}).
\end{align}
Let $X'_{11} \coloneqq X'_{10} \cup (A_5 \times B_{134})\cup (A_6\times B_{235})$; then \cref{lem:gen} shows $\varphi(X'_{11}) = \varphi(X'_{10})$. Let $X'_{12} \coloneqq X'_{11}\cup (A_4 \times B_5)$. Applying \cref{lem:gen}, we obtain
\begin{align} \label{eq:Xp12}
    \varphi(X'_{12}) \le \varphi(X'_{11}) + \varphi_1(A_{1234567})-\varphi_1(A_{123567}).
\end{align}
Finally, let $X'_{13}\coloneqq X'_{12} \cup (A_4 \times B_{124}) =A_{1234567}\times B_{1234567}$; then \cref{lem:gen} shows
\begin{align} \label{eq:non_fano_total}
    3\cdot\varphi_1(A_{1234567}) = \varphi(A_{1234567} \times B_{1234567}) = \varphi(X'_{13}) = \varphi(X'_{12}).
\end{align}
Summing up the (in)equalities \eqref{eq:X2}, \eqref{eq:X4}, \eqref{eq:X6}, \eqref{eq:Xp8}, \eqref{eq:Xp10}, \eqref{eq:Xp12}, \eqref{eq:non_fano_total} and using $\varphi(X_1)=\varphi(X)$, $\varphi(X_{2i+1})=\varphi(X_{2i})$ for $i \in [3]$, $\varphi(X'_9)=\varphi(X'_8)$, and $\varphi(X'_{11})=\varphi(X'_{10})$, we obtain
\begin{align}
    3\cdot\varphi_1(A_{1234567}) \le \varphi(X) & +(\varphi_1(A_{126})-\varphi_1(A_{12})+\varphi_1(A_{135})-\varphi_1(A_{13})+\varphi_1(A_{234})-\varphi_1(A_{23}) \nonumber
    \\ & +\varphi_1(A_{367})-\varphi_1(A_{36})+\varphi_1(A_{147})-\varphi_1(A_{14})) + (\varphi_1(A_{456})-\varphi_1(A_{46})) \label{eq:non_fano_final} \\ & +(\varphi_1(A_{257})-\varphi_1(A_{25})) + (\varphi_1(A_{1234567})-\varphi_1(A_{4567})) \nonumber \\ & + (\varphi_1(A_{1234567})-\varphi_1(A_{12347}))
 + (\varphi_1(A_{1234567})-\varphi_1(A_{123567})) \nonumber
\end{align}
after cancellations. Using the submodularity of $\varphi$, it is clear that 
\begin{align} \label{eq:non_FanoX}
    \varphi(X) \leq \varphi_1(A_{12})+\varphi_1(A_{13})+\varphi_1(A_{23})+\varphi_1(A_6)+\varphi_1(A_4).
\end{align}
Combining \eqref{eq:non_fano_final} and \eqref{eq:non_FanoX}, the inequality \eqref{eq:non_Fano} follows after cancellations. If $\varphi_1$ is the rank function of the Fano matroid and $A_i=\{s_i\}$ for $i \in [7]$, then \eqref{eq:non_fano_final} turns into $9 \le \varphi(X)$ while \eqref{eq:non_FanoX} turns into $\varphi(X)\le 8$, showing that \eqref{eq:non_Fano} is violated. This finishes the proof. 
\end{proof}

As a corollary, we get the following.

\begin{cor} \label{cor:characteristic_dependent}
Let $\varphi_1\colon 2^{S_1}\to \bR$ and $\varphi_2\colon 2^{S_2}\to \bR$ be polymatroid functions.
    \begin{enumerate}[label=(\alph*)] \itemsep0em
        \item \label{it:cor_two} If $k \cdot \varphi_2$ is representable over a skew field of characteristic 2 for some positive integer $k$, then $\varphi_2$ satisfies inequality \eqref{eq:Fano} for any subsets $B_1,\dots, B_7 \subseteq S_2$, where we define $B_{i_1\dots i_m} \coloneqq \bigcup_{j=1}^m B_{i_j}$ for $\{i_1,\dots, i_m\}\subseteq [7]$. Moreover, this inequality is violated if $\varphi_2$ is the rank function of the non-Fano matroid $F_7^-$ and $B_i = \{t_i\}$ for $i \in [7]$.        
        \item \label{it:cor_not_two} If $k \cdot \varphi_1$ is representable over a skew field of characteristic different from 2 for some positive integer $k$, then $\varphi_1$ satisfies inequality \eqref{eq:non_Fano} for any subsets $A_1,\dots, A_7 \subseteq S_1$, where we define $A_{i_1\dots i_m} \coloneqq \bigcup_{j=1}^m A_{i_j}$ for $\{i_1,\dots, i_m\}\subseteq [7]$.
        Moreover, this inequality is violated if $\varphi_1$ is the rank function of the Fano matroid $F_7$ and $A_i = \{s_i\}$ for $i \in [7]$.        
    \end{enumerate}
\end{cor}
\begin{proof}
    To prove \ref{it:cor_two}, assume that $k \cdot \varphi_2$ is representable over a skew field $\bF$ of characteristic 2 for some positive integer $k$. Since the Fano matroid $F_7$ is representable over every field of characteristic 2~\cite[Lemma~6.4.4]{oxley2011matroid}, it is representable over $Z(\bF)$. 
    Hence,  by \cref{lem:representable}, $k\cdot \varphi_2$ and the rank function of $F_7$ admit a tensor product. Therefore, \cref{thm:Fano_non_Fano}\ref{it:Fano} implies that $k\cdot \varphi_2$, and hence $\varphi_2$, satisfies \eqref{eq:Fano}. This finishes the proof of \ref{it:cor_two}.
    Statement \ref{it:cor_not_two} similarly follows from \cref{thm:Fano_non_Fano}\ref{it:non_Fano} as $F_7^-$ is representable over every field of characteristic different from 2; see \cite[Lemma~6.4.4]{oxley2011matroid}.
\end{proof}

\begin{rem}
    We note that our proof of \cref{thm:Fano_non_Fano} and \cref{cor:characteristic_dependent} gives a new proof of the well-known fact that $F_7$ is representable over a skew field $\bF$ if and only if $\bF$ has characteristic 2, while $F_7^-$ is representable over a skew field $\bF$ if and only if $\bF$ has characteristic different from 2 (see \cite[Proposition~6.4.8]{oxley2011matroid} for fields).
\end{rem}

\begin{figure}[t!]
    \centering

\begin{tikzpicture}[scale=0.7]
    \draw[step=1cm,black,thin] (0,0) grid (7,7); 

    \node at (-0.5,6.5) {$s_1$};
    \node at (-0.5,5.5) {$s_2$};
    \node at (-0.5,4.5) {$s_3$};
    \node at (-0.5,3.5) {$s_4$};
    \node at (-0.5,2.5) {$s_5$};
    \node at (-0.5,1.5) {$s_6$};
    \node at (-0.5,0.5) {$s_7$};
    
    \foreach \i in {1,...,7} {
        \node at (\i-0.5,7.4) {$t_{\i}$};
    }

    \fill[black] (0,2) rectangle (1,5);
    \fill[black] (1,5) rectangle (2,6);
    \fill[black] (2,6) rectangle (5,7);
    \fill[black] (1,3) rectangle (2,4);
    \fill[black] (3,5) rectangle (4,6);
    \fill[black] (1,1) rectangle (2,2);
    \fill[black] (2,2) rectangle (3,3);
    \fill[black] (3,3) rectangle (4,4);
    \fill[black] (4,4) rectangle (5,5);
    \fill[black] (5,5) rectangle (6,6);
    \fill[black] (6,0) rectangle (7,1);

    \node at (0.5,0.5) {$y_1$};
    \node at (6.5,6.5) {$y_2$};
    \node at (1.5,0.5) {$y_3$};
    \node at (6.5,5.5) {$y_4$};
    \node at (3.5,2.5) {$y_5$};
    \node at (4.5,3.5) {$y_6$};
    \node at (5.5,2.5) {$y_7$};
    \node at (4.5,1.5) {$y_8$};
    \node at (6.5,4.5) {$y_9$};
    \node at (5.5,4.5) {$y_{10}$};
    \node at (5.5,1.5) {$y_{11}$};
    \node at (2.5,1.5) {$y_{12}$};
    \node at (2.5,0.5) {$y_{13}$};

    \fill[gray, opacity=0.4] (0,0) rectangle (3,1); 
    \fill[gray, opacity=0.4] (2,1) rectangle (3,2); 
    \fill[gray, opacity=0.4] (3,2) rectangle (4,3); 
    \fill[gray, opacity=0.4] (4,3) rectangle (5,4); 
    \fill[gray, opacity=0.4] (5,4) rectangle (6,5); 
    \fill[gray, opacity=0.4] (6,4) rectangle (7,7); 
    \fill[gray, opacity=0.4] (4,1) rectangle (6,2); 
    \fill[gray, opacity=0.4] (5,2) rectangle (6,3);  
\end{tikzpicture}

\caption{Illustration of \cref{rem:dual_fano}. The set $X$ corresponding to the black cells has size 15 and it would span $S\times T$ in a tensor product $(F_7)^*\otimes (F_7^-)^*$. This can be checked by first showing that $y_i$ would be spanned by $X\cup \{y_1,\dots, y_{i-1}\}$ for $i \in [13]$.} 
\label{fig:dual_fano}
\end{figure}

At present, we do not know whether the characteristic-dependent linear rank inequalities obtained above are consequences of previously known inequalities; verifying this seems difficult since they might follow from a combination of known inequalities, submodularity, and monotonicity.

\begin{rem} \label{rem:dual_fano}
    We obtained the proof of \cref{thm:Fano_non_Fano} by generalizing a proof that $F_7$ and $F_7^-$ do not admit a matroid tensor product. 
    For the matroid specialization, the argument proceeds as follows.
    Assuming that $F_7$ and $F_7^-$ admit a tensor product $M$, we find a set $X\subseteq S\times T$ whose rank is forced to be strictly less than the rank of $M$ but which would nevertheless span $S\times T$ in $M$. We note that the same idea can be applied to the duals of these matroids to show that $(F_7)^*$ and $(F_7^-)^*$ do not admit a (matroid) tensor product; see Figure~\ref{fig:dual_fano}. One can also use the idea of this proof to obtain linear inequalities satisfied by any polymatroid function admitting a tensor product with the rank function of $(F_7)^*$ or $(F_7^-)^*$, respectively.
\end{rem}

\subsection{Beyond known inequalities}
\label{sec:new}

As we have seen earlier, rank-3 matroids are particularly interesting from the viewpoint of skew-rep\-re\-sen\-ta\-bi\-li\-ty, since they all satisfy the common information property -- which therefore cannot be used to determine whether such a matroid is skew-representable.  This motivates the study of one of the classical rank-3 matroids: the non-Desargues matroid. The {\it non-Desargues matroid} is a simple rank-3 matroid on 10 elements that arises from the classical Desargues configuration in projective geometry by omitting one line. We denote its ground set by $\{d,a_1,a_2,a_3,b_1,b_2,b_3,c_1,c_2,c_3\}$, and its nontrivial lines are
\[\big\{\{d,a_1,b_1\}, \{d,a_2,b_2\}, \{d,a_3,b_3\}, \{a_1,a_2,c_3\}, \{a_1,a_3,c_2\}, \{a_2,a_3,c_1\}, \{b_1,b_2,c_3\}, \{b_1,b_3,c_2\}, \{b_2,b_3,c_1\}\big\},\]
see also Figure~\ref{fig:non-desargues}. 

\begin{figure}[t!]
    \centering
    \begin{minipage}[t]{0.45\textwidth}
    \centering
    \begin{tikzpicture}[scale=1.25]
            \fill (0,2) circle (3.3pt);
            \node at (0.2,2.2) {$d$};
            \fill (-1.75,-0.6) circle (3.3pt);
            \node at (-2,-0.4) {$c_3$};
            \fill (0.25,-0.7) circle (3.3pt);
            \node at (0.02,-0.47) {$b_1$};
            \fill (3.25,-0.85) circle (3.3pt);
            \node at (3.45,-0.65) {$b_2$};
            \fill (0.45,1.6) circle (3.3pt);
            \node at (0.65,1.8) {$a_2$};
            \fill (0.07,1.2) circle (3.3pt);
            \node at (-0.2,1.3) {$a_1$}; 
            \fill (0.4,1) circle (3.3pt);
            \node at (0.6,1.2) {$a_3$}; 
            \fill (0.7,0.25) circle (3.3pt);
            \node at (0.9,0.05) {$b_3$};
            \fill (0.9,0.7) circle (3.3pt);
            \node at (1.15,0.85) {$c_2$};
            \fill (0.35,0.4) circle (3.3pt);
            \node at (0.35,0.15) {$c_1$};
            
            \draw (0,2) -- (0.25,-0.7);
            \draw (0,2) -- (0.7,0.25);
            \draw (0,2) -- (3.25,-0.85);
            \draw (-1.75,-0.6) -- (0.45,1.6);
            \draw (-1.75,-0.6) -- (3.25,-0.85);
            \draw (0.07,1.2) -- (0.9,0.7);
            \draw (0.25,-0.7) -- (0.9,0.7);
            \draw (0.45,1.6) -- (0.35,0.4);
            \draw (3.25,-0.85) -- (0.35,0.4);
    \end{tikzpicture}
    \subcaption{The non-Desargues matroid.} \label{fig:non-desargues}
\end{minipage}
\hfill
    \begin{minipage}[t]{0.45\textwidth}
        \centering
        \begin{tikzpicture}[scale=1.25]
            \fill (0,2) circle (3.3pt);
            \node at (0.23,2.23) {$e_1$};
            \fill (-1.732,-1) circle (3.3pt);
            \node at (-1.962,-0.77) {$e_2$};
            \fill (1.732,-1) circle (3.3pt);
            \node at (1.962,-0.77) {$e_3$};
            \fill (0,-1) circle (3.3pt);
            \node at (0.2,-0.75) {$e_4$};
            \fill (-1.732/2,1/2) circle (3.3pt);
            \node at (-1.732/2-0.23,1/2+0.23) {$e_6$};
            \fill (1.732/2,1/2) circle (3.3pt);
            \node at (1.732/2+0.23,1/2+0.23) {$e_5$};
        
            \draw (0,2) -- (-1.732,-1) -- (1.732,-1) -- (0,2);
            \draw (0,0) circle [radius=1];        
        \end{tikzpicture}
        \subcaption{The graphic matroid $M(K_4)$.} \label{fig:k4}
    \end{minipage}
    \caption{Illustrations of the non-Desargues matroid and of the graphic matroid of $K_4$.} 
    \label{fig:ksjdvvn}
\end{figure}

Using the framework developed in Sections~\ref{sec:ing} and~\ref{sec:two}, we now prove the fourth main result stated in the introduction. The following theorem gives the explicit linear rank inequality underlying \cref{thm:intro-inequality}. Throughout this section, for ease of presentation, we often use compact notation for set unions, writing, for example, $XYZ$ instead of $X \cup Y \cup Z$ -- a compromise made necessary by the otherwise unmanageably long formulas.

\begin{thm} \label{thm:new}
    Let $\varphi_1 \colon 2^{S_1} \to \bR$ denote the rank function of $M(K_4)$, and let $\varphi_2\colon 2^{S_2} \to \bR$ be a polymatroid function. If $\varphi_1$ and $\varphi_2$ have a tensor product $\varphi\colon 2^{S_1\times S_2} \to \bR$, then  
    \begin{align} \label{eq:newlinrank}
        & \sum_{i=1}^3 (2\cdot\varphi_2(A_{i+1}A_{i+2}B_{i+1}B_{i+2}C_i) + \varphi_2(A_iD) + \varphi_2(B_iD) + \varphi_2(C_i))  \\  & + \varphi_2(C_2) + \varphi_2(A_3B_3C_1) + \varphi_2(A_1A_2B_1B_2) + \varphi_2(C_1C_2C_3) \nonumber \\ 
        \le & 2\cdot \sum_{i=1}^3 (\varphi_2(A_{i+1}A_{i+2}C_i) + \varphi_2(B_{i+1}B_{i+2}C_i) + \varphi_2(A_iB_iD))
        + \varphi_2(C_1C_2) + 4\varphi_2(D) \nonumber
    \end{align}
    holds for all sets $A_1,A_2,A_3,B_1,B_2,B_3,C_1,C_2,C_3,D \subseteq S_2$, where indices of $A_j$, $B_j$, and $C_j$ are meant in a cyclic order (e.g.\ $A_4 = A_1$ and $A_5 = A_2$). 
    This inequality is violated if $\varphi_2$ is the rank function of the non-Desargues matroid, $A_i = \{a_i\}$, $B_i = \{b_i\}$, $C_i = \{c_i\}$ for $i \in [3]$, and $D = \{d\}$. 
\end{thm}
\begin{proof} 
    We may assume that $S_2=A_1A_2A_3B_1B_2B_3C_1C_2C_3D$ by restricting $\varphi_2$ to this set. 
    Consider $M(K_4)$ on ground set $S_1=\{e_1,\dots, e_6\}$ with triangles $\{e_1,e_2, e_6\}$, $\{e_1, e_3, e_5\}$, $\{e_2,e_3,e_4\}$, and $\{e_4,e_5,e_6\}$; see also Figure~\ref{fig:k4}.
    Define 
    \begin{align*}
        X & \coloneqq \prd{e_1}{C_1} \cup \prd{e_2}{C_2} \cup \prd{e_3}{C_3} \cup \prd{e_4}{A_1} \cup \prd{e_5}{A_2} \cup \prd{e_6}{A_3} \\
        Y & \coloneqq \prd{e_1}{C_1} \cup \prd{e_2}{C_2} \cup \prd{e_3}{C_3} \cup \prd{e_4}{B_1} \cup \prd{e_5}{B_2} \cup \prd{e_6}{B_3},
    \end{align*}
    see also Figure~\ref{fig:X}.

    We will use the submodular inequality $\varphi(X) + \varphi(Y) \ge \varphi(X\cap Y) + \varphi(X \cup Y)$. First, we give upper bounds on $\varphi(X)$ and $\varphi(Y)$.

    \begin{cla} \label{cl:pX}
        We have
        \begin{align*}
        \varphi(X) &\le 2\cdot\varphi_2(D) + \sum_{i=1}^3 \big( 
            \varphi_2(A_{i+1}A_{i+2}C_i) 
            + \varphi_2(B_{i+1}B_{i+2}C_i) 
            + \varphi_2(A_iB_iD) \\
        &\hspace{3.5cm}
            - \varphi_2(A_{i+1}A_{i+2}B_{i+1}B_{i+2}C_i) 
            - \varphi_2(B_iD) \big), \\
        \varphi(Y) &\le 2\cdot\varphi_2(D) + \sum_{i=1}^3 \big( 
            \varphi_2(A_{i+1}A_{i+2}C_i) 
            + \varphi_2(B_{i+1}B_{i+2}C_i) 
            + \varphi_2(A_iB_iD) \\
        &\hspace{3.5cm}
            - \varphi_2(A_{i+1}A_{i+2}B_{i+1}B_{i+2}C_i) 
            - \varphi_2(A_iD) \big).
        \end{align*}
    \end{cla}
    \begin{claimproof}
            \begin{figure}
        \centering
         \begin{minipage}[t]{0.32\textwidth}
            \centering
            \begin{tikzpicture}[scale=0.43]
                \draw[step=1cm,black,thin] (0,0) grid (10,6); 
        
                \foreach \i in {1,...,6} {
                    \node at (-0.5,{6.5-\i}) {\small $e_{\i}$};
                }
                \foreach \i in {1,2,3} {
                    \node at ({\i-0.5},6.5) {\small $A_{\i}$};
                    \node at ({\i+2.5},6.5) {\small $B_{\i}$};
                    \node at ({\i+5.5},6.5) {\small $C_{\i}$};
                }
                \node at (9.5, 6.5) {\small $D$};
                \fill[black] (6,5) rectangle (7,6);                
                \fill[black] (7,4) rectangle (8,5);
                \fill[black] (8,3) rectangle (9,4);
                \fill[black] (0,2) rectangle (1,3);
                \fill[black] (1,1) rectangle (2,2);
                \fill[black] (2,0) rectangle (3,1);
            \end{tikzpicture} \subcaption{Illustration of $X$ (black cells).} \label{fig:X}
        \end{minipage} \hfill
        \begin{minipage}[t]{0.32\textwidth}
            \centering
            \begin{tikzpicture}[scale=0.43]
                \draw[step=1cm,black,thin] (0,0) grid (10,6); 
        
                \foreach \i in {1,...,6} {
                    \node at (-0.5,{6.5-\i}) {\small $e_{\i}$};
                }
                \foreach \i in {1,2,3} {
                    \node at ({\i-0.5},6.5) {\small $A_{\i}$};
                    \node at ({\i+2.5},6.5) {\small $B_{\i}$};
                    \node at ({\i+5.5},6.5) {\small $C_{\i}$};
                }
                \node at (9.5, 6.5) {\small $D$};
                \fill[black] (1,5) rectangle (3,6);
                \fill[black] (6,5) rectangle (7,6);
                \fill[black] (0,4) rectangle (1,5);
                \fill[black] (2,4) rectangle (3,5);
                \fill[black] (7,4) rectangle (8,5);
                \fill[black] (0,3) rectangle (2,4);
                \fill[black] (8,3) rectangle (9,4);
                \fill[gray, opacity=0.4] (0,2) rectangle (1,3);
                \fill[gray, opacity=0.4] (1,1) rectangle (2,2);
                \fill[gray, opacity=0.4] (2,0) rectangle (3,1);
            \end{tikzpicture} \subcaption{Illustration of $H_0$ (black cells) and $H\setminus H_0$ (gray cells).} \label{fig:H}
        \end{minipage} \hfill
        \begin{minipage}[t]{0.32\textwidth}
            \centering
            \begin{tikzpicture}[scale=0.43]
                \draw[step=1cm,black,thin] (0,0) grid (10,6); 
        
                \foreach \i in {1,...,6} {
                    \node at (-0.5,{6.5-\i}) {\small $e_{\i}$};
                }
                \foreach \i in {1,2,3} {
                    \node at ({\i-0.5},6.5) {\small $A_{\i}$};
                    \node at ({\i+2.5},6.5) {\small $B_{\i}$};
                    \node at ({\i+5.5},6.5) {\small $C_{\i}$};
                }
                \node at (9.5, 6.5) {\small $D$};
                \fill[black] (4,5) rectangle (7,6);
                \fill[black] (3,4) rectangle (4,5);
                \fill[black] (5,4) rectangle (6,5);
                \fill[black] (7,4) rectangle (8,5);
                \fill[black] (3,3) rectangle (5,4);
                \fill[black] (8,3) rectangle (9,4);
                \fill[gray, opacity=0.25] (0,2) rectangle (1,3);
                \fill[gray, opacity=0.65] (3,2) rectangle (4,3);
                \fill[gray, opacity=0.25] (1,1) rectangle (2,2);
                \fill[gray, opacity=0.65] (4,1) rectangle (5,2);
                \fill[gray, opacity=0.25] (2,0) rectangle (3,1);
                \fill[gray, opacity=0.65] (5,0) rectangle (6,1);
                \fill[black] (9,2) rectangle (10,3);
                \fill[black] (9,1) rectangle (10,2);
                \fill[gray, opacity=0.65] (9,0) rectangle (10,1);
            \end{tikzpicture} \subcaption{Illustration of $I_0$ (black cells), $I_1\setminus I_0$ (dark gray cells), and $I\setminus I_1$ (light gray cells).} \label{fig:I}
        \end{minipage} \caption{Illustration of the proof of \cref{cl:pX}. Note that, in general, the sets $A_i, B_i, C_i$ for $i \in [3]$ and $D$ are not necessarily disjoint.}
    \end{figure}
    
        By symmetry, it is enough to prove the first inequality.
        Define 
        \begin{align*}
            H & \coloneqq \prd{e_1}{(A_2A_3C_1)} \cup \prd{e_2}{(A_1A_3C_2)} \cup \prd{e_3}{(A_1A_2C_3)} \cup \prd{e_4}{A_1} \cup \prd{e_5}{A_2} \cup \prd{e_6}{A_3}, \\
            I & \coloneqq \prd{e_1}{(B_2B_3C_1)} \cup \prd{e_2}{(B_1B_3C_2)} \cup \prd{e_3}{(B_1B_2C_3)} \cup \prd{e_4}{(A_1B_1D)} \cup \prd{e_5}{(A_2B_2D)} \cup \prd{e_6}{(A_3B_3D)}.
        \end{align*}
        Observe that $X \subseteq H \cap I$. 
        We will use $\varphi(X)\le \varphi(H\cap I)$ and the submodular inequality $\varphi(H\cap I) \le \varphi(H)+\varphi(I)-\varphi(H\cup I)$. 
        
        First, we compute $\varphi(H)$; see Figure~\ref{fig:H} for an illustration.
        Let \[H_0\coloneqq \prd{e_1}{(A_2A_3C_1)} \cup \prd{e_2}{(A_1A_3C_2)} \cup \prd{e_3}{(A_1A_2C_3)}.\]
        Using \cref{lem:gen} with $\varphi_1(\{e_2,e_3, e_4\}) = \varphi_1(\{e_2,e_3\})$, $\varphi_1(\{e_1,e_3, e_5\}) = \varphi_1(\{e_1,e_3\})$, and $\varphi_1(\{e_1,e_2, e_6\}) = \varphi_1(\{e_1,e_2\})$ and \cref{lem:direct} with $\varphi_1(\{e_1, e_2,e_3\}) = \varphi_1(e_1) + \varphi_1(e_2) + \varphi_1(e_3)$, we obtain
        \begin{equation} \label{eq:phiH}
          \varphi(H) = \varphi(H_0) = \sum_{i=1}^3 \varphi(\prd{e_i}{(A_{i+1}A_{i+2}C_i)}) = \sum_{i=1}^3 \varphi_2(A_{i+1}A_{i+2}C_i).
        \end{equation}
        
        Next, we give an upper bound on $\varphi(I)$; see Figure~\ref{fig:I} for an illustration. Define \[ I_0 \coloneqq \prd{e_1}{(B_2B_3C_1)} \cup \prd{e_2}{(B_1B_3C_2)} \cup \prd{e_3}{(B_1B_2C_3)} \cup \prd{e_4}{D} \cup \prd{e_5}{D}.\]
        It is clear by the submodularity of $\varphi$ that
        \begin{equation} \label{eq:phiJ}
            \varphi(I_0) \le \sum_{i=1}^3 \varphi_2(B_{i+1}B_{i+2}C_i) + 2\cdot\varphi_2(D).
        \end{equation}
        Using \cref{lem:gen} with $\varphi_1(\{e_2,e_3,e_4\}) = \varphi_1(\{e_2,e_3\})$, $\varphi_1(\{e_1,e_3,e_5\}) = \varphi_1(\{e_1,e_3\})$, $\varphi_1(\{e_1,e_2,e_6\}) = \varphi_1(\{e_1,e_2\})$, and $\varphi_1(\{e_4,e_5,e_6\}) = \varphi_1(\{e_4,e_5\})$, we have $\varphi(I_0) = \varphi(I_1)$ for \[ I_1 \coloneqq \prd{e_1}{(B_2B_3C_1)} \cup \prd{e_2}{(B_1B_3C_2)} \cup \prd{e_3}{(B_1B_2C_3)} \cup \prd{e_4}{(B_1D)} \cup \prd{e_5}{(B_2D)} \cup \prd{e_6}{(B_3D)}.\]
        Using \cref{lem:gen} three times, we have
        \begin{align}
           \varphi(I)  = 
          \varphi(I_1\cup \prd{e_4}{A_1}\cup \prd{e_5}{A_2}\cup \prd{e_6}{A_3}) \le \varphi(I_1) + \sum_{i=1}^3 ({\varphi_2(A_iB_iD)} - \varphi_2(B_iD)). \label{eq:phiI}
        \end{align}
        Combining \eqref{eq:phiJ}, $\varphi(I_0) = \varphi(I_1)$, and \eqref{eq:phiI}, we get
        \begin{equation} \label{eq:pI}
            \varphi(I) \le 2\cdot\varphi_2(D) + \sum_{i=1}^3 (\varphi_2(B_{i+1}B_{i+2}C_i) + \varphi_2(A_iB_iD)-\varphi_2(B_iD)).
        \end{equation}

        Finally, we give a lower bound on $\varphi(H\cup I)$. Using \cref{lem:gen} 
        three times with the fact that $(H\cup I)\cap \prd{e_i}{S_2} = \prd{e_i}{(A_{i+1}A_{i+2}B_{i+1}B_{i+2}C_i)}$ holds for $i \in [3]$, we get
        \begin{align*}
            3\cdot\varphi_2(S_2) & = \varphi(\{e_1,e_2,e_3\}\times S_2)\\
            & = \varphi((H\cup I)\cup \prd{e_1}{S_2} \cup \prd{e_2}{S_2} \cup \prd{e_3}{S_2}) \\
            & \le \varphi(H\cup I)+\sum_{i=1}^3 (\varphi_2(S_2)-\varphi_2(A_{i+1}A_{i+2}B_{i+1}B_{i+2}C_i)),
        \end{align*}
        implying
        \begin{equation} \label{eq:phiHI}
            \varphi(H\cup I) \ge \sum_{i=1}^3 \varphi_2(A_{i+1}A_{i+2}B_{i+1}B_{i+2}C_i).
        \end{equation}
        Using \eqref{eq:phiH}, \eqref{eq:pI}, and \eqref{eq:phiHI} together with $\varphi(X) \le \varphi(H\cap I) \le  \varphi(H)+\varphi(I)-\varphi(H\cup I)$, we obtain the claimed inequality.
    \end{claimproof}

    Next, we observe that a lower bound on $\varphi(X\cap Y)$ can be easily computed.
    
    \begin{cla} \label{cl:pXcapY}
        $\varphi(X\cap Y) \ge \sum_{i=1}^3 \varphi_2(C_i)$.
    \end{cla}
    \begin{claimproof}
        We have $X \cap Y \supseteq \prd{e_1}{C_1} \cup \prd{e_2}{C_2} \cup \prd{e_3}{C_3}$ (where equality holds if $A_i\cap B_i=\emptyset$ for $i \in [3]$). As $\{e_1, e_2, e_3\}$ is a basis of $M(K_4)$, \cref{lem:direct} implies
        $\varphi(X \cap Y) \ge \sum_{i=1}^3 \varphi(\prd{e_i}{C_i}) = \sum_{i=1}^3 \varphi_2(C_i)$, proving our claim.
    \end{claimproof}

    Finally, we give a lower bound on $\varphi(X\cup Y)$.
    
    \begin{cla} \label{cl:pXcupY}
        $\varphi(X\cup Y) \ge \varphi_2(C_2)+\varphi_2(A_3B_3C_1) + \varphi_2(A_1A_2B_1B_2)+\varphi_2(C_1C_2C_3)-\varphi_2(C_1C_2)$.
    \end{cla}
    
    \begin{claimproof}    
        Note that \[X \cup Y = \prd{e_1}{C_1} \cup \prd{e_2}{C_2} \cup \prd{e_3}{C_3} \cup \prd{e_4}{(A_1B_1)} \cup \prd{e_5}{(A_2B_2)} \cup \prd{e_6}{(A_3B_3)},\]
        see also Figure~\ref{fig:pXucpY} for an illustration.
        Let $Z_1 \coloneqq (X\cup Y) \cup \prd{e_2}{C_1}$. 
        Using \cref{lem:gen}, we get        
        \begin{align} \label{eq:Z1}
            \varphi(Z_1)  \le \varphi(X\cup Y) + \varphi_2(C_1C_2) - \varphi_2(C_2). 
        \end{align}
        Let $Z_2 \coloneqq Z_1 \cup \prd{e_6}{C_1}$. As $\varphi_1(\{e_1,e_2,e_6\}) = \varphi_1(\{e_1,e_2\})$ and $\{e_1,e_2\}\times C_1 \subseteq Z_1$, \cref{lem:gen} shows that $\varphi(Z_1)=\varphi(Z_2)$.  
        Let $Z_3\coloneqq Z_2\cup \prd{e_6}{S_2}$. Using \cref{lem:gen}, we get
        \begin{align} \label{eq:Z3}
            \varphi(Z_3)\le \varphi(Z_2) + \varphi_2(S_2)-\varphi_2(A_3B_3C_1).
        \end{align}
        Let $Z_4\coloneqq Z_3 \cup \prd{e_5}{(A_1B_1)}$. As $\varphi_1(\{e_4,e_5,e_6\}) = \varphi_1(\{e_4,e_6\})$ and $\{e_4,e_6\} \times (A_1B_1)\subseteq Z_3$, \cref{lem:gen} shows that $\varphi(Z_3)=\varphi(Z_4)$.
        Let $Z_5 \coloneqq Z_4 \cup \prd{e_5}{S_2}$. Since $\prd{e_5}{A_1A_2B_1B_2} \subseteq Z_4$, \cref{lem:gen} yields
        \begin{align} \label{eq:Z5}
            \varphi(Z_5) \le \varphi(Z_4) + \varphi_2(S_2)-\varphi_2(A_1A_2B_1B_2). 
        \end{align}
        Let $Z_6\coloneqq Z_5 \cup (S_1 \times C_1) \cup (S_1 \times C_2)\cup (S_1 \times C_3) \cup (\{e_4,e_5,e_6\}\times S_2) = (\{e_4,e_5,e_6\}\times S_2) \cup (S_1 \times C_1C_2C_3)$.
        Using \cref{lem:gen} four times, together with the facts $\varphi_1(S_1) = \varphi_1(\{e_1,e_5,e_6\})$, $\{e_1,e_5,e_6\} \times C_1 \subseteq Z_5$, $\varphi_1(S_1)=\varphi_1(\{e_2,e_5,e_6\})$, $\{e_2,e_5,e_6\}\times C_2 \subseteq Z_5$, $\varphi_1(S_1) = \varphi_1(\{e_3,e_5,e_6\})$, $\{e_3,e_5,e_6\}\times C_3 \subseteq Z_5$, $\varphi_1(\{e_4,e_5,e_6\}) = \varphi_1(\{e_5,e_6\})$, and $\{e_5,e_6\}\times S_2 \subseteq Z_5$, we conclude that $\varphi(Z_5)=\varphi(Z_6)$.
        Let $Z_7 \coloneqq Z_6 \cup \prd{e_1}{S_2}$. Since $\prd{e_1}{C_1C_2C_3}\subseteq Z_6$, \cref{lem:gen} yields
        \begin{align} \label{eq:Z7}
            \varphi(Z_7) \le \varphi(Z_6)+ \varphi_2(S_2)-\varphi_2(C_1C_2C_3).
        \end{align}
        Finally, using $\{e_1,e_4,e_5\}\times S_2 \subseteq Z_7$ and $\varphi_1(\{e_1,e_4,e_5\}) = \varphi_1(S_1)$, \cref{lem:gen} shows that $\varphi(S_1\times S_2) = \varphi(Z_7)$.
        Summing up the inequalities \eqref{eq:Z1}--\eqref{eq:Z7} and using the fact that $\varphi(Z_{2i-1})=\varphi(Z_{2i})$ for $i\in[3]$ and that $\varphi(Z_7)=\varphi(S_1\times S_2)$, we get
        \begin{align*}
            \varphi(S_1\times S_2) \le \varphi(X\cup Y) + 3\cdot \varphi_2(S_2) + \varphi_2(C_1C_2) -\varphi_2(C_2)-\varphi_2(A_3B_3C_1)-\varphi_2(A_1A_2B_1B_2)-\varphi_2(C_1C_2C_3).
        \end{align*}
        Using $\varphi(S_1\times S_2) = 3\cdot \varphi_2(S_2)$, the claim follows.
   \end{claimproof}

    \begin{figure}
        \centering
        \begin{tikzpicture}[scale=0.7]
        \draw[step=1cm,black,thin] (0,0) grid (10,6); 

        \foreach \i in {1,...,6} {
            \node at (-0.45,{6.5-\i}) {$e_{\i}$};
        }
        \foreach \i in {1,2,3} {
            \node at ({\i-0.5},6.4) {$A_{\i}$};
            \node at ({\i+2.5},6.4) {$B_{\i}$};
            \node at ({\i+5.5},6.4) {$C_{\i}$};
        }
        \node at (9.5, 6.4) {$D$};

        \fill[black] (0,2) rectangle (1,3);
        \fill[black] (1,1) rectangle (2,2);
        \fill[black] (2,0) rectangle (3,1);
        \fill[black] (3,2) rectangle (4,3);
        \fill[black] (4,1) rectangle (5,2);
        \fill[black] (5,0) rectangle (6,1);
        \fill[black] (6,5) rectangle (7,6);
        \fill[black] (7,4) rectangle (8,5);
        \fill[black] (8,3) rectangle (9,4);

        \node at (6.5,4.5) {\scriptsize $Z_{1,1}$};
        \node at (6.5,0.5) {\scriptsize $Z_{2,1}$};
        \node at (0.5,0.5) {\scriptsize $Z_{3,1}$};
        \node at (1.5,0.5) {\scriptsize $Z_{3,2}$};
        \node at (3.5,0.5) {\scriptsize $Z_{3,3}$};
        \node at (4.5,0.5) {\scriptsize $Z_{3,4}$};
        \node at (7.5,0.5) {\scriptsize $Z_{3,5}$};
        \node at (8.5,0.5) {\scriptsize $Z_{3,6}$};
        \node at (9.5,0.5) {\scriptsize $Z_{3,7}$};
        \node at (0.5,1.5) {\scriptsize $Z_{4,1}$};
        \node at (3.5,1.5) {\scriptsize $Z_{4,2}$};
        \node at (2.5,1.5) {\scriptsize $Z_{5,1}$};
        \node at (5.5,1.5) {\scriptsize $Z_{5,2}$};
        \node at (6.5,1.5) {\scriptsize $Z_{5,3}$};
        \node at (7.5,1.5) {\scriptsize $Z_{5,4}$};
        \node at (8.5,1.5) {\scriptsize $Z_{5,5}$};
        \node at (9.5,1.5) {\scriptsize $Z_{5,6}$};
        
        \node at (7.5,5.5) {\scriptsize $Z_{6,1}$};
        \node at (8.5,5.5) {\scriptsize $Z_{6,2}$};
        \node at (8.5,4.5) {\scriptsize $Z_{6,3}$};
        \node at (6.5,3.5) {\scriptsize $Z_{6,4}$};
        \node at (7.5,3.5) {\scriptsize $Z_{6,5}$};
        \node at (1.5,2.5) {\scriptsize $Z_{6,6}$};
        \node at (2.5,2.5) {\scriptsize $Z_{6,7}$};
        \node at (4.5,2.5) {\scriptsize $Z_{6,8}$};
        \node at (5.5,2.5) {\scriptsize $Z_{6,9}$};
        \node at (6.5,2.5) {\scriptsize $Z_{6,10}$};
        \node at (7.5,2.5) {\scriptsize $Z_{6,11}$};
        \node at (8.5,2.5) {\scriptsize $Z_{6,12}$};
        \node at (9.5,2.5) {\scriptsize $Z_{6,13}$};

        \node at (0.5,5.5) {\scriptsize $Z_{7,1}$};
        \node at (1.5,5.5) {\scriptsize $Z_{7,2}$};
        \node at (2.5,5.5) {\scriptsize $Z_{7,3}$};
        \node at (3.5,5.5) {\scriptsize $Z_{7,4}$};
        \node at (4.5,5.5) {\scriptsize $Z_{7,5}$};
        \node at (5.5,5.5) {\scriptsize $Z_{7,6}$};
        \node at (9.5,5.5) {\scriptsize $Z_{7,7}$};
        
        \end{tikzpicture}
        \caption{Illustration of the proof of \cref{cl:pXcupY}: $X\cup Y$ (black cells) and $Z_1,\dots,Z_7$. For each $i \in [7]$, the cells labeled $Z_{i,1},\dots,Z_{i,j_i}$ form $Z_i\setminus Z_{i-1}$, where $Z_0\coloneqq X\cup Y$.} \label{fig:pXucpY}
\end{figure}
   
   To finish the proof, we use the submodular inequality $0 \le \varphi(X)+\varphi(Y)- \varphi(X \cap Y) - \varphi(X\cup Y)$ together with Claims~\ref{cl:pX}, \ref{cl:pXcapY}, and \ref{cl:pXcupY}. If $\varphi_2$ is the rank function of the non-Desargues matroid and a tensor product with $M(K_4)$ existed, the three claims would yield $0 \le 5 + 5 - 3 - 8 = -1$, a contradiction.  In general, we obtain
   \begin{align*}
0 &\le 
\bigg( 
  2\cdot\varphi_2(D) 
  + \sum_{i=1}^3 \big( 
      \varphi_2(A_{i+1}A_{i+2}C_i) 
      + \varphi_2(B_{i+1}B_{i+2}C_i) 
      + \varphi_2(A_iB_iD) \\
&\hspace{3cm}
      - \varphi_2(A_{i+1}A_{i+2}B_{i+1}B_{i+2}C_i) 
      - \varphi_2(B_iD) 
  \big) 
\bigg) \\
&\quad + \bigg( 
  2\cdot\varphi_2(D) 
  + \sum_{i=1}^3 \big( 
      \varphi_2(A_{i+1}A_{i+2}C_i) 
      + \varphi_2(B_{i+1}B_{i+2}C_i) 
      + \varphi_2(A_iB_iD) \\
&\hspace{3cm}
      - \varphi_2(A_{i+1}A_{i+2}B_{i+1}B_{i+2}C_i) 
      - \varphi_2(A_iD) 
  \big) 
\bigg) \\
&\quad - \sum_{i=1}^3 \varphi_2(C_i) 
 - \big( 
    \varphi_2(C_2) 
    + \varphi_2(A_3B_3C_1) 
    + \varphi_2(A_1A_2B_1B_2) 
    + \varphi_2(C_1C_2C_3) 
    - \varphi_2(C_1C_2) 
\big),
\end{align*}
   which proves the theorem after rearranging the terms.
\end{proof}

The following corollary is the explicit form of \cref{thm:intro-inequality} announced in \cref{sec:results}.

\begin{cor} \label{cor:skdf}
    Let $\varphi_2\colon 2^{S_2}\to \bR$ be a folded skew-representable polymatroid function. Then, inequality \eqref{eq:newlinrank} holds for any subsets $A_1,A_2,A_3,B_1,B_2,B_3,C_1,C_2,C_3,D \subseteq S_2$.
    Moreover, the inequality is violated if $\varphi_2$ is the rank function of the non-Desargues matroid, $A_i = \{a_i\}$, $B_i=\{b_i\}$, $C_i = \{c_i\}$, and $D=\{d\}$ for $i \in [3]$.
\end{cor}
\begin{proof}
    Let $k$ be a positive integer such that $k\cdot \varphi_2$ is skew-representable. As $M(K_4)$ is regular, its rank function has a tensor product with $k\cdot \varphi_2$ by \cref{lem:representable}. Then, the result follows from \cref{thm:new} by dividing the obtained inequality by $k$.
\end{proof}

Recall that \cref{thm:uniform} shows that the non-Desargues matroid is $1$-tensor-compatible with every uniform matroid. We close the section by showing, using \cref{thm:new}, that it is not $2$-tensor-compatible with $U_{2,3}$.

We first show that $U_{2,3}$ has a unique tensor product with itself, namely the cographic matroid $M^*(K_{3,3})$ of $K_{3,3}$.
Note that \cref{lem:unique-power} can be used to show that $M^*(K_{3,3})$ is the unique skew-representable matroid in $U_{2,3}\otimes U_{2,3}$; the next lemma is stronger, as it also rules out non-skew-representable members.

\begin{lem} \label{lem:U232_unique}
    The tensor product $U_{2,3}\otimes U_{2,3}$ is unique and is isomorphic to $M^*(K_{3,3})$.
\end{lem}
\begin{proof}
    Consider the graph $K_{3,3}$ on vertex set $\{s_1,s_2,s_3,t_1,t_2,t_3\}$ and edge set $\{(s_i,t_j)\mid i, j \in [3]\}$, and let $N$ be its cographic matroid. Let $S\coloneqq \{s_1,s_2,s_3\}$, $T\coloneqq \{t_1,t_2,t_3\}$, and $M_1=(S, r_1)$ and $M_2=(T,r_2)$ be two copies of the uniform matroid $U_{2,3}$. Let $P=(S\times T, p) \in M_1 \otimes M_2$; we show that $P=N$. We have $N\in M_1 \otimes M_2$; this follows, for example, from \cref{prop:tensor} since $N|\prd{s_i}{T}$ and $N|S^{t_i}$ are isomorphic to $U_{2,3}$ for $i \in [3]$ and $N$ has rank 4. Therefore, it is enough to prove that $P$ is unique. We show this by determining for each subset $Z\subseteq S\times T$ of size 4 whether it is a basis of $P$.
    
    It is clear that $Z$ is dependent in $P$ if $\prd{s_i}{T} \subseteq Z$ or $S^{t_i} \subseteq Z$ for some $i \in [3]$. 
    Otherwise, if $\prd{s_i}{T}\cap Z = \emptyset$ or $S^{t_i} \cap Z = \emptyset$ for some $i \in [3]$, then $Z$ is a basis of $P$. Indeed, if, for example, $\prd{s_1}{T}\cap Z = \emptyset$, then $Z\subseteq \{s_2,s_3\}\times T$ and $|Z\cap \prd{s_i}{T}| = 2$ for $i \in \{2,3\}$, thus $Z$ is a basis of $P$ by \cref{lem:direct}. 
    It remains to consider the cases when $1 \le |Z\cap \prd{s_i}{T}| \le 2$ and $1 \le |Z\cap S^{t_i}| \le 2$ for each $i \in [3]$. We may assume by symmetry that $|Z\cap \prd{s_1}{T}| = |Z\cap S^{t_1}| = 2$ and $|Z\cap \prd{s_i}{T}| = |Z\cap S^{t_i}| = 1$ for $i \in \{2,3\}$. If $(s_1, t_1) \notin Z$, then $Z=\{(s_1,t_2), (s_1,t_3), (s_2, t_1), (s_3,t_1)\}$ is not a basis of $P$ since $p(\prd{s_1}{T}\cup S^{t_1}) = 3$ by \cref{prop:crosses}. Finally, if $(s_1, t_1) \in Z$, then we may assume by symmetry that $Z=\{(s_1,t_1), (s_1, t_2), (s_2, t_1), (s_3, t_3)\}$. Then, $Z$ spans $\prd{s_1}{T}$ and $S^{t_1}$ in $P$, thus it also spans $\prd{s_3}{T}$ and $S^{t_3}$ as $(s_3, t_3) \in Z$. This further implies that $Z$ is a spanning set of $P$, thus it is a basis. This finishes the proof. 
\end{proof}

\cref{lem:U232_unique} and \cref{lem:assoc} show that a matroid is $2$-tensor-compatible with $U_{2,3}$ if and only if it admits a tensor product with $M^*(K_{3,3})$.
It is not difficult to check that $M^*(K_{3,3})$ contains $M(K_4)$ as a minor (this also follows from the fact that the only $M(K_4)$-minor-free binary matroids are the direct sums of graphic matroids of series-parallel graphs; see~\cite[Theorem~10.4.8]{oxley2011matroid}).
Then, \cref{lem:minor} and \cref{thm:new} imply the following.

\begin{cor} \label{cor:non_Desargues_not_2_compatible}
    Any matroid that is 2-tensor-compatible with $U_{2,3}$ admits a tensor product with $M(K_4)$. In particular, the non-Desargues matroid is not 2-tensor-compatible with $U_{2,3}$.
\end{cor}

\section{Discussion and open problems}
\label{sec:open}

In this paper, we introduced a tensor product framework to study skew-representability of matroids and polymatroid functions. In particular, we provided a characterization of skew-representable matroids, as well as of those representable over skew fields of a given characteristic, in terms of tensor products, showing that the corresponding decision problems are co-recursively enumerable and yielding computably verifiable certificates of non-skew-representability. We gave a construction for the freest tensor product of a rank-3 matroid and a uniform matroid. As an application of our framework, we derived new linear rank inequalities, including the first known inequality for folded skew-representable matroids that goes beyond the common information property. We close the paper by mentioning some open problems.

\begin{enumerate}[leftmargin=*,itemsep=0em]
    \item Theorem~\ref{thm:1modular} shows that if a matroid is 1-tensor-compatible with $U_{2,3}$, then it is also 1-modular extendable. It would be interesting to see whether the converse holds: Does $1$-modular extendability imply $1$-tensor-compatibility with $U_{2,3}$? 
    \item As discussed in \cref{rem:prob1}, it remains open whether skew-representability can be characterized by tensor products with a single prescribed sequence of matroids.
    \item The computability results of \cref{sec:tensor} leave open the following (see \cref{rem:prob2}): given a finite set $C\subseteq\bP$, is the problem of deciding whether a matroid is representable over a skew field of characteristic in $\bP\setminus C$ co-recursively enumerable?
    \item The proofs of Theorems~\ref{thm:ingleton},~\ref{thm:Fano_non_Fano} and~\ref{thm:new} relied on a similar idea: generating certain sets in a specific order. However, there is a lot of freedom in choosing this order, and repeating the same argument with a different choice also leads to linear rank inequalities. This raises a natural question: Are all inequalities obtained in this way essentially the same, or can different orders lead to genuinely new inequalities?
    \item The notion of almost folded skew-representability leads to several natural questions. In what follows, we consider connected matroids. Let $\textsc{FolSRep}$ denote the class of folded skew-representable matroids, and let $\textsc{AFolSRep}$ denote the class of matroids whose rank functions are pointwise limits of functions $\varphi_i/c_i$, where $c_i\in\bZ_+$ and the $\varphi_i$ are polymatroid functions on the same ground set that are representable over a single skew field $\bF$. The analogous fixed-field notion for multilinear representability was studied by Kühne and Yashfe~\cite{kuhne2026entropic}. Let $\textsc{LimSRep}$ denote the potentially larger class obtained by allowing the skew field to vary along the approximating sequence, and let $\textsc{UTens}$ denote the class of matroids whose rank functions are $k$-tensor-compatible with the rank function of $U_{2,3}$ for every $k\in\bZ_+$, in the polymatroid sense. We then have
    \[
        \textsc{FolSRep}
        \subseteq
        \textsc{AFolSRep}
        \subseteq
        \textsc{LimSRep}
        \subseteq
        \textsc{UTens}.
    \]
    The first two inclusions are immediate from the definitions. For the last one, apply Lemma~\ref{lem:representable} repeatedly to each approximating polymatroid, divide by the normalizing factor, and pass to a convergent subsequence; the polymatroid inequalities and the tensor-product equalities are preserved under limits. We do not know which of these inclusions are strict, raising the following problems.
    \begin{enumerate}[label=(\alph*)]
        \item Is every almost folded skew-representable matroid folded skew-representable, that is, is
        $\textsc{FolSRep}=\textsc{AFolSRep}$?
    
        \item Does allowing the skew field to vary along the approximating sequence strictly enlarge the class, that is, is
        $\textsc{AFolSRep}\subsetneq\textsc{LimSRep}$?
    
        \item Is every matroid in $\textsc{UTens}$ a limit of normalized skew-representable polymatroid functions, that is, is
        $\textsc{LimSRep}=\textsc{UTens}$?
    
        \item Is every matroid in $\textsc{UTens}$ folded skew-representable, that is, is
        $\textsc{FolSRep}=\textsc{UTens}$?
    \end{enumerate}
\end{enumerate}

\medskip

\paragraph{AI Disclosure.} No generative AI was used in the preparation of the conference version of this work. During the preparation of the present version, ChatGPT-5.6 Sol and Claude Opus 5 were used for checking arguments, identifying relevant references, and generating code for computational searches. In particular, ChatGPT identified a gap in the proof of~\cite[Theorem~5.13]{berczi2026interaction} described in \cref{rem:soda_gap}, which was addressed by the authors. None of the mathematical proofs included in the present paper was obtained by AI.

\paragraph{Acknowledgement.}

The authors are grateful to Dan Král' and Tamás Szőnyi for helpful discussions. 
András Imolay was supported by the EKÖP-KDP-25 University Research Scholarship Program, Cooperative Doctoral Program of the Ministry for Culture and Innovation, from the source of the National Research, Development and Innovation Fund. Carles Padró was partially supported by the Spanish Ministry of Science and Innovation under projects ACITHEC PID2021-124928NB-I00 and MATSE PID2024-156636NB-C21. The research received further support from the Lend\"ulet Programme of the Hungarian Academy of Sciences -- grant number LP2021-1/2021, from the Ministry of Innovation and Technology of Hungary from the National Research, Development and Innovation Fund -- grant numbers ADVANCED 150556, 153096, and from the Dynasnet European Research Council Synergy project -- grant number ERC-2018-SYG 810115. This work was supported in part by EPSRC grant EP/X030989/1.

\bibliographystyle{abbrv}
\bibliography{tensor_journal}

\end{document}